\documentclass[final]{siamltex}
\usepackage{amsfonts}
\usepackage{graphics}
\usepackage{caption2}
\usepackage{amssymb}
\usepackage{psfrag}
\usepackage{amsmath,bm}
\usepackage{graphicx}
\usepackage{multirow}
\usepackage{cite}
\usepackage{subfigure}
\usepackage[normalem]{ulem}
\usepackage{color}
\usepackage{mathrsfs}
\usepackage{subeqnarray}
\usepackage{cases}

\allowdisplaybreaks[4]
\def\x{{\bm   x}}

\def\R{{\mathbb R}}

\def\u{{\bm   u}}

\def\J{{\bm   J}}
\def\n{{\bm   n}}

\def\({\left(}
\def\[{\left[}
\def\){\right)}
\def\]{\right]}
\def\div{\nabla\cdot }
\def\grad{\nabla }

\newtheorem{thm}{Theorem}

\numberwithin{equation}{section}
\numberwithin{thm}{section}
\numberwithin{lem}{section}
\begin{document}
\title{Linearly decoupled energy-stable numerical  methods for    multi-component two-phase compressible flow\thanks{This work is  supported by   National Natural Science Foundation of China (No.11301163),  and KAUST research fund to the
Computational Transport Phenomena Laboratory.}
}

\author{Jisheng Kou\thanks{School of Mathematics and Statistics, Hubei Engineering  University, Xiaogan 432000, Hubei, China. } \and Shuyu Sun\thanks{Corresponding author. Computational Transport Phenomena Laboratory, Division of Physical Science and Engineering,
King Abdullah University of Science and
Technology, Thuwal 23955-6900, Kingdom of Saudi Arabia.   Email: {\tt shuyu.sun@kaust.edu.sa}.}
\and Xiuhua Wang\thanks{School of Mathematics and Statistics, Hubei Engineering  University, Xiaogan 432000, Hubei, China. }
}

 \maketitle

\begin{abstract}
In this paper, for the first time we propose two linear, decoupled,  energy-stable numerical  schemes for    multi-component two-phase compressible flow with a  realistic equation of state (e.g. Peng-Robinson equation of state). The methods are constructed based on the scalar auxiliary variable (SAV) approaches for Helmholtz free energy and the intermediate velocities that are designed to decouple the tight  relationship between velocity and molar densities.  The intermediate  velocities are also involved in the discrete   momentum  equation to ensure the consistency  with  the mass balance equations.  Moreover, we propose a component-wise SAV approach for a multi-component fluid, which requires solving  a sequence of linear, separate  mass balance equations.  We prove that the methods preserve  the unconditional energy-dissipation feature. Numerical results are presented to verify the effectiveness of the proposed methods.

\end{abstract}

\section{Introduction}

It is a very  important issue to simulate multi-component two-phase compressible fluid systems with  a realistic equation of state (e.g. Peng-Robinson equation of state \cite{Peng1976EOS}).  It has a wide range of applications   in  chemical  and reservoir engineering   \cite{kousun2015CMA,qiaosun2014,kousun2015SISC,kousun2015CHE,polivka2014compositional, mikyvska2015General,jindrova2013fast,kousun2018Flash}, especially the pore scale modeling  of subsurface fluid flow including shale gas reservoir.  The classical models  of incompressible two-phase flows  or compositional flows  have been extensively studied and employed   \cite{firoozabadi1999thermodynamics,Moortgat2013compositional,chen2006multiphase,kouandsun2014dgtwophase},  the primal state variables of which are often    pressure, temperature, and chemical composition.  Although the classical models have been widely used, they suffer from a few essential   limitations  as pointed out in \cite{mikyvska2011new, polivka2014compositional};  for example, it is required to construct a pressure  equation since  there is no intrinsic  pressure equation  \cite{polivka2014compositional}.  

An alternative  modeling framework, which uses the moles, volume, and temperature  (the so-called  NVT-based framework)  as the primal state variables, has been intensively studied recently \cite{kousun2015CMA,qiaosun2014,kousun2015SISC,kousun2015CHE,kouandsun2016multiscale,polivka2014compositional,jindrova2013fast,mikyvska2015General,kousun2018Flash}. The NVT-based modeling framework  originates from the phase-splitting  calculations  of multi-component fluids at specified  moles, volume and temperature \cite{michelsen1999,mikyvska2011new,Nagarajan1991}.  Very recently, in the NVT-based framework, a general multi-component two-phase compressible flow model is rigorously derived by Kou and Sun in \cite{kouandsun2017modeling} based on   the thermodynamic laws and  realistic equations of state (e.g. Peng-Robinson equation of state).  This model has at least three important features that are distinguished  from the classical models:
 \begin{itemize}
\item  It has thermodynamically-consistent unified formulations for general average velocities and mass diffusion fluxes;
\item It uses diffusive interfaces and  realistic equations of state, and as a result,  it can characterize the fluid compressibility and partial miscibility between different phases;
\item It uses a general  thermodynamic pressure, which  is  a function of the molar density and temperature, and consequently, it is never required to construct the pressure  equation.
\end{itemize}
In addition,  another formulation of the  momentum conservation equation, which is  convenient for numerical simulation,  has been derived in \cite{kouandsun2017modeling} by the relation between the pressure gradient  and  chemical potential gradients. In this paper, we consider how to develop and analyze efficient   numerical methods for this model problem.

 A key challenge   in numerical simulation of diffuse interface models is  to construct  efficient numerical schemes  preserving  the discrete energy-dissipation law \cite{shen2015SIAM,shen2016JCP}.   In constructing  energy-stable numerical schemes for multi-component two-phase compressible flow model,  there are at least two main difficulties: one is   the strong nonlinearity  of bulk Helmholtz free energy density; the other is  the tightly coupling relationship between molar densities and flow velocity through the convection term in the mass balance equations and the stress force arising  from chemical potential gradients in the momentum balance equation.  
An energy-dissipation  numerical scheme was developed in \cite{kouandsun2017modeling} based on a convex-concave  splitting of Helmholtz free energy density, but it leads to a nonlinear and coupled system of the mass balance equations and momentum balance equation.  In this paper, we focus on the  linear, decoupled, energy stable numerical schemes.

Recently,  for incompressible two-phase flows, a decoupled approach \cite{Minjeaud2013Uncoupled}  was developed by  introducing an intermediate velocity  in the phase equation to resolve the coupling relation between the velocity and phase function, and this technique was used to construct linear, decoupled, efficient numerical  methods    for phase-field models of  incompressible two-phase flows \cite{shen2015SIAM,shen2016JCP}. When applying  this technique to compressible multi-component two-phase  flow model considered in this paper, we encounter two challenging problems: the first  is how to construct the intermediate velocities since the stress force in the momentum balance equation  is different from phase-field models;  the second is how to treat the  momentum balance equation using intermediate velocities.    The second problem is because at the time-discrete level, the  velocity variable   in the convection term of  the  momentum balance equation shall be consistent with the intermediate  velocities when we  combine the mass balance equation of each component and the  momentum balance equation to derive the variation of the kinetic  energy.  In this work, we will  construct two  intermediate  velocities, both of which can   uncouple the relationship between velocity and molar densities; we will also propose a discrete  formulation of the  momentum balance equation, which involves the intermediate  velocities and consequently consistent with  the mass balance equations.
It is noted that one of the introduced intermediate  velocities is for the first time defined by a component-wise way, and thus, it is very efficient for a special multi-component fluid.

There have been at least four approaches in the literature  dealing with the bulk Helmholtz free energy density derived from Peng-Robinson equation of state for constructing  energy-stable numerical schemes. The  first approach is the convex splitting method \cite{Elliott1993ConvexSpliting,Eyre1998ConvexSplitting}, which has been  popularly used in phase-field models \cite{shen2015SIAM,Wise2009Convex,Eyre1998ConvexSplitting,Baskaran2013convexsplitting}.
 The energy-stable numerical scheme based on the convex splitting method have also been developed and analyzed for the diffuse-interface models with Peng-Robinson equation of state \cite{qiaosun2014,fan2017componentwise,kou2017compositional,kouandsun2017modeling,kousun2018Flash,Peng2017convexsplitting}. The second approach is a modified Newton's method with a relaxation parameter that is dynamically chosen to ensure the energy decay property \cite{kousun2015SISC}.  The third approach is  a fully-implicit unconditionally-stable scheme \cite{kousun2015CHE},  which uses the difference of Helmholtz free energy density to approximate  the chemical potential. The fourth numerical scheme is developed in \cite{Li2017IEQ} based on the invariant energy quadratization (IEQ) approach that is  a novel, efficient method   and  has been  applied to many phase-field models intensively recently \cite{Yang2016IEQ,Yang2017IEQ,Yang2017IEQ2}. Very recently, a scalar auxiliary variable (SAV) approach \cite{Shen2018SAV} is built upon the  IEQ approach. It leads to unconditionally stable numerical schemes, which   only need to solve the linear equations with constant coefficients  at each time step. In this paper, we will apply the SAV approach to treat the mass balance equations and construct linear, unconditionally stable numerical schemes. Moreover, we will develop a component-wise SAV approach for a multi-component flow model, which  uncouples the relationships between multiple components and allows us to solve each component mass balance equation separately.  The schemes for gradient flows of multiple functions in \cite{Shen2018SAV} usually require the computation of eigen-matrix and eigenvalues  to  achieve the decoupled forms, but this computation  cost is free for the proposed component-wise SAV approach.  So the proposed scheme is efficient and easy-to-implement for the case of multiple components.

We must note that the proposed numerical schemes for multi-component two-phase flows are perfect combinations of the above intermediate  velocity approaches and   SAV approaches, which lead to a sequence of linearly decoupled equations.  The proposed schemes  are proved to be unconditionally energy stable.

The rest of this paper is organized as   follows. In Section 2, we will give a brief description  of  the multi-component two-phase flow model.
 In Section 3, we will propose the numerical schemes and prove the unconditional energy stability.  In Section 4,  numerical tests are carried out  to show the   effectiveness of the proposed methods.   Finally,   concluding remarks are provided in Section 5.

\section{Mathematical model of multi-component two-phase flow}

In this section, we briefly introduce  the mathematical model of multi-component two-phase flow with Peng-Robinson equation of state, which is very recently proposed in \cite{kouandsun2017modeling}. 

We consider the motion of a mixture fluid   composed of $M$ chemical  components at a constant  temperature.  Let $n_i$ be the molar density of   component $i$, and we denote the molar density vector by $\n = [n_1,n_2,\cdots,n_M]^T$.

Mathematical model developed  in \cite{kouandsun2017modeling} can employ  any  realistic   equation of state, for instance, the van der Waals equation of state and Peng-Robinson equation of state (PR-EOS)  \cite{Peng1976EOS}.  PR-EOS  has been widely  applied  in oil reservoir and chemical engineering due to its accuracy.  In this work, we focus on the PR-EOS-based Helmholtz free energy density $f_{b}(\n)$ of a homogeneous bulk  fluid, which has a form as
\begin{eqnarray}\label{eqHelmholtzEnergy_a0}
    f_b(\n)= f_b^{\textnormal{ideal}}(\n) +f_b^{\textnormal{repulsion}}(\n)+f_b^{\textnormal{attraction}}(\n),
\end{eqnarray}
where $f_b^{\textnormal{ideal}}$, $f_b^{\textnormal{repulsion}}$ and $f_b^{\textnormal{attraction}}$ are formulated in Appendix \ref{appendixHelmholtz}. 

The diffuse interfaces always occurs  between multiple  phases of a realistic fluid.  To  characterize  this feature, a local density gradient  contribution is added to the  free energy density of an inhomogeneous fluid, and consequently, the  general  Helmholtz free energy density (denoted by $f$) is expressed as
\begin{eqnarray}\label{eqHelmholtzfreeenergydensity}
 f(\n)=f_b(\n)+\frac{1}{2}\sum_{i,j=1}^Mc_{ij}\grad n_i\cdot\grad n_j,
 \end{eqnarray}
 where $c_{ij}(1\leq i,j\leq M)$ are the cross influence parameters depending on temperature but independent of molar densities.  The formulations of $c_{ij}$ can be found  in Appendix \ref{appendixInfluenceParameters}. We assume that the influence parameter matrix $\big(c_{ij}\big)_{i,j=1}^M$ is symmetric and moreover it is positive definite or positive semi-definite.

The chemical potential of component $i$ is defined as 
  \begin{eqnarray}\label{eqgeneralchemicalpotential}
   \mu_i(\n)=\frac{\delta  f(\n)}{\delta n_i}=\mu_i^b(\n)-\sum_{j=1}^M\div\(c_{ij}\grad{n_j}\),~~~\mu_i^b(\n)=\frac{\partial  f_b(\n)}{\partial n_i},
   \end{eqnarray}
  where $\frac{\delta}{\delta n_i}$ denotes the variational  derivative. By the thermodynamical relations, the general  thermodynamical pressure can be formulated as  a function of $\n$ at a constant temperature %\cite{kouandsun2017modeling}
\begin{eqnarray}\label{eqMultiComponentDefPresGeneralB}
    p(\n) &=& \sum_{i=1}^Mn_i \mu_i(\n)- f(\n)\nonumber\\
   &=&p_b-  \sum_{i,j=1}^Mn_i\div\(c_{ij}\grad{n_j}\)-\frac{1}{2} \sum_{i,j=1}^Mc_{ij}\nabla n_i\cdot\nabla n_j,
\end{eqnarray}
where $p_b(\n)=\sum_{i=1}^Mn_i\mu_i^b(\n)-f_b(\n).$

The overall molar density of a mixture is denoted by $n=\sum_{i=1}^Mn_i$. Let $M_{w,i}$ denote the molar weight of component $i$, and then we denote  the mass density of component $i$ by $\rho_i=n_iM_{w,i}$ and denote the overall mass density of a mixture by $\rho=\sum_{i=1}^M\rho_i$. 

 We now describe the governing equations. First, the mass balance equation   for component  $i$    is 
\begin{eqnarray}\label{eqGeneralNSEQMass}
\frac{\partial n_i}{\partial t}+\div\(\u n_i\)+\div \J_i=0,
\end{eqnarray}
where $\u$ is a specific or average velocity and $\J_i$ is the diffusion flux of component $i$. In general, we can express the diffusion flux  of component $i$  as   \cite{Cogswell2010thesis,kouandsun2016multiscale,kouandsun2017modeling}
\begin{eqnarray}\label{eqMultiCompononentMassConserveDiffusion}
  \bm J_{i}=-\sum_{j=1}^M\mathcal{M}_{ij}\nabla \mu_{j},~~i=1,\cdots, M,
\end{eqnarray}
where   $\bm{\mathcal{M}}=\(\mathcal{M}_{ij}\)_{i,j=1}^M$ is the mobility tensor.   The mobility matrix $\bm{\mathcal{M}}$ shall be symmetric and at least positive semidefinite so that Onsager's reciprocal principle \cite{Groot2015NET} and the second law of thermodynamics are satisfied. 
 
 Three choices of the mobility $\bm{\mathcal{M}}$ in \eqref{eqMultiCompononentMassConserveDiffusion} are provided in \cite{kouandsun2017modeling} as below.   \begin{description}
\item[(A1)] The first mobility choice is to take $\bm{\mathcal{M}}$  as a diagonal positive definite matrix with diagonal elements 
\begin{equation}\label{eqMultiCompononentMassConserveDiffusionChoiceA}
\mathcal{M}_{i}=\mathcal{M}_{ii}=\frac{D_in_i}{RT},
\end{equation}
where  $R$ stands for the universal gas constant and  $D_i>0$ is the diffusion coefficient of component $i$. The  diffusion flux  has a form \cite{Cogswell2010thesis,kousun2018Flash} as
$\J_i  = -\frac{D_in_i}{RT}\grad\mu_i$.
 In this case, $\u$ and $\J_i$ is viewed as the mean velocity and general  mixture diffusion fluxes at the constant temperature  and pressure, respectively. 
\item[(A2)]     The second choice is to take $\bm{\mathcal{M}}$   as a full matrix 
\begin{eqnarray}\label{eqMultiCompononentMassConserveDiffusionChoiceB}
\mathcal{M}_{ii} = \sum_{j=1}^M\frac{\mathcal{D}_{ij} n_in_j}{nRT},~~~~~\mathcal{M}_{ij} = -\frac{\mathcal{D}_{ij} n_in_j}{nRT},~~j\neq i, 
\end{eqnarray}
where  the mole diffusion coefficients $\mathcal{D}_{ij}$ satisfy $\mathcal{D}_{ii}=0$ and $\mathcal{D}_{ij}=\mathcal{D}_{ji}>0$ for $i\neq j$.  In this case, $\u$ is  the molar-average velocity.
\item[(A3)]  The third mobility $\bm{\mathcal{M}}$ has the following formulation
\begin{eqnarray}\label{eqMultiCompononentMassConserveDiffusionChoiceC}
\mathcal{M}_{ii} = \sum_{j=1}^M\frac{\mathscr{D}_{ij}n_i\rho_j}{M_{w,i}\rho RT},~~~~~\mathcal{M}_{ij} = -\frac{\mathscr{D}_{ij}n_in_j}{\rho RT},~~j\neq i, 
\end{eqnarray}
where   the mass diffusion coefficients $\mathscr{D}_{ij}$ satisfy $\mathscr{D}_{ii}=0$ and $\mathscr{D}_{ij}=\mathscr{D}_{ji}>0$ for $i\neq j$. 
In this case, $\u$ is actually  the mass-average velocity.
\end{description}
 
We now introduce the thermodynamically-consistent  momentum balance equation, which is expressed as  \cite{kouandsun2017modeling} 
\begin{eqnarray}\label{eqGeneralNSEQ}
&&\rho\(\frac{\partial  \u}{\partial t}+ \u\cdot\grad{ \u}\)+\sum_{i=1}^MM_{w,i}\J_i\cdot\grad\u=-\nabla p +\nabla\(\lambda\div\u\)\nonumber\\
&&~~+\div\eta\(\nabla\u+\nabla\u^T\)
-\sum_{i,j=1}^M\div\(c_{ij}\nabla n_i\otimes\nabla n_j\),
\end{eqnarray}
where  $\lambda=\xi-\frac{2}{3}\eta$,  and  $\xi$ and   $\eta$ represent the volumetric  viscosity  and   shear viscosity respectively. We assume that $\eta>0$ and $\xi>\frac{2}{3}\eta$, and thus $\lambda>0$.
If $\u$ is the mass-average velocity, the term   $\sum_{i=1}^MM_{w,i}\J_i\cdot\grad\u$ vanishes, while for the other types of $\u$, it  is crucial  to ensure the thermodynamical consistency.
 It is proved in \cite{kouandsun2017modeling} that the gradients of  the pressure and chemical potentials  have the following relation 
\begin{eqnarray}\label{eqPresChptlMultiGrad01}
 \sum_{i=1}^Mn_i\grad \mu_i 
 =\grad p+\sum_{i,j=1}^M\div\(c_{ij}\grad n_i\otimes\grad n_j\),
\end{eqnarray}
and then we   reformulate the momentum conservation equation \eqref{eqGeneralNSEQ}  as
\begin{eqnarray}\label{eqGeneralNSEQb}
&&\rho\(\frac{\partial  \u}{\partial t}+ \u\cdot\grad{ \u}\)+\sum_{i=1}^MM_{w,i}\J_i\cdot\grad\u=-\sum_{i=1}^Mn_i\grad \mu_i\nonumber\\
   &&~~+\nabla\(\lambda\div\u\) +\div\eta\(\nabla\u+\nabla\u^T\) ,
\end{eqnarray}
which shows that the  fluid motion is  driven by the  chemical potential gradients.

In this work, we consider numerical schemes for the model formulated by  \eqref{eqGeneralNSEQMass} and \eqref{eqGeneralNSEQb} coupling with  the chemical potential \eqref{eqgeneralchemicalpotential} and  the diffusion flux  \eqref{eqMultiCompononentMassConserveDiffusion}. For the boundary conditions, we assume that   all boundary terms in \eqref{eqGeneralNSEQMass} and \eqref{eqGeneralNSEQb} will vanish when integrating by parts is performed; for example, we can use homogeneous Neumann boundary conditions or periodic boundary conditions.

We assume that $\Omega\subset \R^d (1\leq d\leq3)$ is an open, bounded and connected domain with the sufficiently smooth boundary $\partial\Omega$.  The Helmholtz free energy and  kinetic energy within $\Omega$ at a specified time are defined as
\begin{eqnarray}\label{eqPDETotalEnergy}
&F=F_b+F_\grad,~~F_b=\int_{\Omega} f_b(\n)d\x,~~F_\grad=\frac{1}{2}\sum_{i,j=1}^M\int_{\Omega} c_{ij}\grad n_i\cdot\grad n_jd\x,\nonumber\\
&E=\frac{1}{2}\int_{\Omega}\rho|\u|^2d\x.
\end{eqnarray}
It is proved in \cite{kouandsun2017modeling} that
the total energy, i.e. the sum  of    the Helmholtz free energy and  kinetic energy,    is dissipated with time as
\begin{eqnarray}\label{eqPDETotalEnergyDecay}
\frac{\partial (F+E)}{\partial t} \leq 0.
\end{eqnarray}

In order to use the scalar auxiliary variable (SAV) approach \cite{Shen2018SAV} , we define $H(t)=\sqrt{F_b+\sum_{i=1}^MC_{T,i}N_i^t}$, where  $N_i^t=\int_\Omega n_id\x$.  Here,  $C_{T,i}\geq0$   is the thermodynamical coefficient  of component $i$ to ensure  $F_b+\sum_{i=1}^MC_{T,i}N_i^t\geq0$, and the choice of $C_{T,i}\geq0$ may depend on $T$ but independent of molar densities. Then the chemical potential of component $i~(1\leq i\leq M)$ can be reformulated as
\begin{subequations}\label{eqDiscreteChemicalPotential01}
 \begin{equation}\label{eqDiscreteChemicalPotential01A}
\mu_i=\frac{H(t)}{\sqrt{F_b+\sum_{j=1}^MC_{T,j}N_j^t}}\mu_i^{b}-\sum_{j=1}^M\div\(c_{ij}\grad{n_j}\),
  \end{equation}
  \begin{equation}\label{eqDiscreteChemicalPotential01A}
\frac{\partial H}{\partial t}=\sum_{i=1}^M\int_\Omega\frac{\mu_i^{b}}{2\sqrt{F_b+\sum_{j=1}^MC_{T,j}N_j^t}}\frac{\partial n_i}{\partial t}d\x.
  \end{equation}
 \end{subequations}
 The modified Helmholtz free energy  is  defined as  
 \begin{equation*}
\mathcal{F}=H^2+F_\grad-\sum_{i=1}^MC_{T,i}N_i^t.
  \end{equation*}
  In the continuous model, we have $\mathcal{F}=F$, but at the time-discrete level, the modified Helmholtz free energy may be generally different from the original Helmholtz free energy.

\section{Energy-stable numerical methods}
In this section,  we aim to develop efficient energy-dissipated semi-implicit time marching scheme for simulating  the above multi-component flow model. The key difficulties  result from the strong nonlinearity  of Helmholtz free energy density and fully coupling relations between molar densities and velocity.   In this work,  our purpose is to uncouple  this tightly coupling relations between molar densities and velocity,  and from this, we will develop linearly decoupled numerical schemes preserving  the feature of energy dissipation.

For a given  time interval $\mathcal{I}=(0,T_{f}]$, where $T_{f}>0$, we divide $\mathcal{I}$ into   $\mathcal{N}$ subintervals $\mathcal{I}_{k}=(t_{k},t_{k+1}]$, where $t_{0}=0$ and $t_{\mathcal{N}}=T_{f}$, and we denote $\delta t_{k}=t_{k+1}-t_{k}$.  For any scalar $v(t)$ or vector $\bm{v}(t)$, we denote by $v^{k}$ or $\bm{v}^{k}$ its approximation at the time $t_{k}$.
The traditional notations $\(\cdot,\cdot\)$ and $\|\cdot\|$ are used to represent the  inner product and norm of $L^2\(\Omega\)$, $\(L^2\(\Omega\)\)^d$ or $\(L^2\(\Omega\)\)^{d\times d}$ respectively.

\subsection{Velocity-density decoupled  semi-implicit  scheme}\label{secVelocityDensityDecoupledScheme}
We try to develop a linear  semi-implicit  scheme that decouples the tight relationship between molar densities and velocity. This scheme allows us to solve the mass balance equations and momentum equation separately.  This scheme can be applied for the model problems with the general diffusion mobility, especially  the cases that the mobility is a full tensor. 

We denote $\n^{k}= [n_1^k,n_2^k,\cdots,n_M^k]^T$, and define  $\mu_i^{k+1}$  as
 \begin{subequations}\label{eqDiscreteChemicalPotential01}
 \begin{equation}\label{eqDiscreteChemicalPotential01A}
\mu_i^{k+1}=\frac{H^{k+1}+H^k}{2\sqrt{F_b(\n^k)+\sum_{j=1}^MC_{T,j}N_j^t}}\mu_i^{b}\(\n^k\)-\sum_{j=1}^M\div\(c_{ij}\grad{n_j^{k+1}}\),
  \end{equation}
  \begin{equation}\label{eqDiscreteChemicalPotential01B}
\frac{H^{k+1}-H^k}{\delta t_k}=\sum_{i=1}^M\int_\Omega\frac{\mu_i^{b}\(\n^k\)}{2\sqrt{F_b(\n^k)+\sum_{j=1}^MC_{T,j}N_j^t}}\frac{n_i^{k+1}-n_i^k}{\delta t_k}d\x.
  \end{equation}
 \end{subequations}
  Furthermore, we define an intermediate  velocity $\u_\star^{k}$ as
\begin{eqnarray}\label{eqDecoupledDiscreteVelocityStar}
\u_\star^{k} = \u^k-\frac{\delta t_{k}}{\rho^k}\sum_{i=1}^M n_i^{k}\grad \mu_i^{k+1}.
\end{eqnarray}
We note that $\u_\star^{k}$ can be viewed as an approximation of $\u^{k+1} $ obtained by neglecting the three parts: the convection term, $\sum_{i=1}^MM_{w,i}\J_i^{k+1}\cdot\grad\u^{k+1}$, and the viscosity terms,  in \eqref{eqDecoupledDiscreteMultiComponentMomentumConserveC01}.

Using the intermediate  velocity $\u_\star^{k}$, we construct the semi-implicit time   scheme for the molar density balance equation \eqref{eqGeneralNSEQMass}    as
 \begin{subequations}\label{eqDecoupledDiscreteMultiCompononentMassConserve01}
 \begin{equation}\label{eqDecoupledDiscreteMultiCompononentMassConserve01Main}
 \frac{ n_i^{k+1}-n_i^{k}}{\delta t_{k}}+\div(n_i^{k}\u_\star^{k})+\div\J_i^{k+1} =0,
 \end{equation}
 \begin{equation}\label{eqDecoupledDiscreteMultiCompononentMassConserve01Diffusion}
\J_i^{k+1} = -\sum_{j=1}^M\mathcal{M}^k_{ij}\nabla \mu_{j}^{k+1},
\end{equation}
\end{subequations}
where we denote by $\mathcal{M}^k_{ij}$ the mobility coefficients calculated from molar densities $\n^k$ since the mobility coefficients  $\mathcal{M}_{ij}$, generally depending on  molar densities, can be treated explicitly. 

 We can see that  only $\n^{k+1}$ is  the unknown variable of the equations \eqref{eqDecoupledDiscreteMultiCompononentMassConserve01}. This means that the use of $\u_\star^{k}$  eliminates the tight coupling relationship between molar densities and velocity.  We can solve  \eqref{eqDecoupledDiscreteMultiCompononentMassConserve01}   to obtain molar densities $\n^{k+1} $. Once  $\n^{k+1}$   is calculated,  we can get $\mu_i^{k+1}$,  $\J_i^{k+1}$ and $\u_\star^{k} $ from \eqref{eqDiscreteChemicalPotential01}, \eqref{eqDecoupledDiscreteMultiCompononentMassConserve01Diffusion} and  \eqref{eqDecoupledDiscreteVelocityStar} respectively.  A semi-implicit scheme for the momentum balance equation \eqref{eqGeneralNSEQb} is  formulated as
\begin{eqnarray}\label{eqDecoupledDiscreteMultiComponentMomentumConserveC01}
&& \rho^{k}\frac{\u^{k+1}-\u^{k}}{\delta t_{k}}+\rho^{k}\u_\star^{k}\cdot\grad\u^{k+1}+\sum_{i=1}^MM_{w,i}\J_i^{k+1}\cdot\grad\u^{k+1}=-\sum_{i=1}^Mn_i^{k}\grad \mu_i^{k+1}\nonumber\\
   &&~~+\nabla\(\lambda^k\div\u^{k+1}\)+\div\eta^k\(\nabla\u^{k+1}+\(\nabla\u^{k+1}\)^T\),
\end{eqnarray}
which is a linear equation of velocity $\u^{k+1}$ only.   In the convection term of \eqref{eqDecoupledDiscreteMultiComponentMomentumConserveC01}, the use of $\u_\star^{k}$ instead of $\u^k$ is consistent with the mass balance equations as shown in the proof of Theorem \ref{thmDecoupledSchemeEnergy}, and moreover, it avoids to use the existing approach in \cite{Minjeaud2013Uncoupled,shen2015SIAM} that needs to impose the overall mass  equation into the momentum  equation for the sake of achieving energy dissipation for phase-field models with the large density ratios. We note that this treatment (i.e., using $\u_\star^{k}$  instead of $\u^k$ in the convection term of the momentum  equation) can be directly  applied to the phase-field models with different densities.

We now prove that the above linearly  decoupled scheme satisfies the discrete energy dissipation law.  To do this, we define the discrete  kinetic energy and the modified  Helmholtz free energy  as
\begin{subequations}\label{eqDiscreteTotalEnergy}
\begin{eqnarray}\label{eqDiscreteKineticEnergy}
E^k=\frac{1}{2}\int_{\Omega}\rho^k|\u^k|^2d\x,
\end{eqnarray}
\begin{eqnarray}\label{eqDiscreteHelmholtz}
\mathcal{F}^k=|H^k|^2+F^k_\grad-\sum_{i=1}^MC_{T,i}N_i^t,~~~F_\grad^{k}=\frac{1}{2}\int_{\Omega}\sum_{i,j=1}^Mc_{ij}\grad n_i^k\cdot\grad n_j^kd\x.
\end{eqnarray}
\end{subequations}

\begin{thm}\label{thmDecoupledSchemeEnergy}
The modified  total (free) energy, i.e., the sum of    the modified  Helmholtz free energy and   kinetic energy,  determined by  \eqref{eqDecoupledDiscreteMultiCompononentMassConserve01} and \eqref{eqDecoupledDiscreteMultiComponentMomentumConserveC01} associated with \eqref{eqDiscreteChemicalPotential01}  and \eqref{eqDecoupledDiscreteVelocityStar}, is dissipated with time steps, i.e.
\begin{eqnarray}\label{eqDecoupledDiscreteTotalEnergyDecay}
E^{k+1}+\mathcal{F}^{k+1} \leq E^{k}+\mathcal{F}^{k}.
\end{eqnarray}

\end{thm}
\begin{proof}
We first estimate the difference between $|H^{k+1}|^2$ and $|H^k|^2$ using \eqref{eqDiscreteChemicalPotential01B} as
\begin{align}\label{eqCHEConvexSplittingPoreEnergyDecayProof03}
|H^{k+1}|^2-|H^{k}|^2&= (H^{k+1}+H^k)\(H^{k+1}-H^{k}\)\nonumber\\
&= \sum_{i=1}^M\(\frac{\(H^{k+1}+H^k\)\mu_i^{b}\(\n^k\)}{2\sqrt{F_b(\n^k)+C_{T,i}N_i^t}},n_i^{k+1}-n_i^k\).
\end{align}
Since the influence parameter matrix $\big(c_{ij}\big)_{i,j=1}^M$ is symmetric and it is positive definite or positive semi-definite,  we have
\begin{align}\label{eqDiscreteMultiCompononentMassConserveProof01}
F_\grad^{k+1}-F_\grad^{k}&=\frac{1}{2}\int_{\Omega}\sum_{i,j=1}^Mc_{ij}\(\grad n_i^{k+1}\cdot\grad n_j^{k+1}-\grad n_i^k\cdot\grad n_j^k\)d\x\nonumber\\
&=\frac{1}{2}\int_{\Omega}\sum_{i,j=1}^Mc_{ij}\(\grad \(n_i^{k+1}-n_i^k\)\cdot\grad n_j^{k+1}+\grad n_i^k\cdot\grad \(n_j^{k+1}-n_j^k\)\)d\x\nonumber\\
&=\sum_{i,j=1}^M\(\grad \(n_i^{k+1}-n_i^k\),c_{ij}\grad n_j^{k+1}\)- \sum_{i,j=1}^M\(c_{ij}\grad \(n_i^{k+1}-n_i^k\),\grad \(n_j^{k+1}-n_j^k\)\)\nonumber\\
&\leq\sum_{i,j=1}^M\(\grad \(n_i^{k+1}-n_i^k\),c_{ij}\grad n_j^{k+1}\)\nonumber\\
&\leq-\sum_{i,j=1}^M\(n_i^{k+1}-n_i^{k},\div\(c_{ij}\grad{n_j^{k+1}}\)\).
\end{align}
The inequalities \eqref{eqCHEConvexSplittingPoreEnergyDecayProof03} and  \eqref{eqDiscreteMultiCompononentMassConserveProof01} yield
\begin{equation}\label{eqDiscreteMultiCompononentMassConserveProof02}
\mathcal{F}^{k+1}-\mathcal{F}^{k}=|H^{k+1}|^2-|H^{k}|^2+F_\grad^{k+1}-F_\grad^{k}\leq\sum_{i=1}^M\(\mu_i^{k+1} ,n_i^{k+1}-n_i^k\).
\end{equation}
Substituting \eqref{eqDecoupledDiscreteMultiCompononentMassConserve01} into \eqref{eqDiscreteMultiCompononentMassConserveProof02}, we derive
\begin{align}\label{eqDiscreteMultiCompononentMassConserveProof03}
\frac{\mathcal{F}^{k+1}-\mathcal{F}^{k}}{\delta t_{k}}&\leq-\sum_{i=1}^M\(\div (n_i^{k}\u_\star^{k})+\div\J_i^{k+1},\mu_i^{k+1}\)\nonumber\\
&\leq-\sum_{i=1}^M\(\div (n_i^{k}\u_\star^{k}),\mu_i^{k+1}\) -\sum_{i,j=1}^{M}\big(\mathcal{M}^{k}_{ij}\nabla \mu_{i}^{k+1},\grad\mu_j^{k+1}\big).
\end{align}

We now turn to consider  the difference between $E^{k+1}$ and $E^k$.  We introduce the intermediate  kinetic energy as $$E_\star^{k}=\frac{1}{2}\(\rho^k\u_\star^k,\u_\star^k\).$$
The difference between $E^{k+1}$ and $E_\star^k$ is estimated as
 \begin{align}\label{eqDecoupledDiscreteMultiComponentMomentumConserveProof01}
E^{k+1}-E_\star^{k}&=\frac{1}{2}\(\rho^{k+1},|\u^{k+1}|^2\)-\frac{1}{2}\(\rho^k,|\u_\star^{k}|^2\)\nonumber\\
&=\frac{1}{2}\(\rho^{k},|\u^{k+1}|^2-|\u_\star^{k}|^2\)+\frac{1}{2}\(\rho^{k+1}-\rho^{k},|\u^{k+1}|^2\)\nonumber\\
   &= \(\rho^{k}\(\u^{k+1}-\u_\star^{k}\),\u^{k+1}\)-\frac{1}{2}\(\rho^{k},|\u^{k+1}-\u_\star^{k}|^2\)\nonumber\\
   &~~~+\frac{1}{2}\(\rho^{k+1}-\rho^{k},|\u^{k+1}|^2\)\nonumber\\
    &\leq \(\rho^{k}\(\u^{k+1}-\u_\star^{k}\),\u^{k+1}\)+\frac{1}{2}\(\rho^{k+1}-\rho^{k},|\u^{k+1}|^2\).
\end{align}
On the other hand,  we have the overall mass balance equation as
\begin{eqnarray}\label{eqDecoupledDiscreteMultiCompononentMassConserveProof02}
\frac{ \rho^{k+1}-\rho^{k}}{\delta t_{k}}+\div(\rho^{k}\u_\star^{k})+\sum_{i=1}^MM_{w,i}\div\J_i^{k+1}=0,
\end{eqnarray}
and taking into account the definition of $\u_\star^k$, we rewrite \eqref{eqDecoupledDiscreteMultiComponentMomentumConserveC01} as
\begin{align}\label{eqDecoupledDiscreteMultiCompononentMassConserveProof03}
 \rho^{k}\frac{\u^{k+1}-\u_\star^{k}}{\delta t_{k}}&=-\rho^{k}\u_\star^{k}\cdot\grad\u^{k+1}-\sum_{i=1}^MM_{w,i}\J_i^{k+1}\cdot\grad\u^{k+1}\nonumber\\
   &~~+\nabla\(\lambda^k\div\u^{k+1}\)+\div\eta^k\(\nabla\u^{k+1}+\(\nabla\u^{k+1}\)^T\).
\end{align}
Substituting \eqref{eqDecoupledDiscreteMultiCompononentMassConserveProof02} and \eqref{eqDecoupledDiscreteMultiCompononentMassConserveProof03} into \eqref{eqDecoupledDiscreteMultiComponentMomentumConserveProof01} yields  
\begin{align}\label{eqDecoupledDiscreteMultiCompononentMassConserveProof04}
\frac{E^{k+1}-E_\star^{k}}{\delta t^k}&\leq -\(\rho^{k}\u_\star^{k}\cdot\grad\u^{k+1}+\sum_{i=1}^MM_{w,i}\J_i^{k+1}\cdot\grad\u^{k+1},\u^{k+1}\)\nonumber\\
   &~~+\(\nabla\(\lambda^k\div\u^{k+1}\)+\div\eta^k\(\nabla\u^{k+1}+\(\nabla\u^{k+1}\)^T\),\u^{k+1}\)\nonumber\\
   &~~-\frac{1}{2}\(\div(\rho^{k}\u_\star^{k})+\sum_{i=1}^MM_{w,i}\div\J_i^{k+1},|\u^{k+1}|^2\)\nonumber\\
   &\leq- \left\|\sqrt{\lambda^k}\div\u^{k+1}\right\|^2 -\frac{1}{2} \left\|\sqrt{\eta^{k}}\(\nabla\u^{k+1}+\(\nabla\u^{k+1}\)^T\)\right\|^2.
\end{align}
We apply  the definition of $\u_\star^k$ to derive
\begin{eqnarray}\label{eqDecoupledDiscreteMultiCompononentMassConserveProof05}
E_\star^{k}-E^{k}&=&\(\rho^k\(\u_\star^{k} - \u^k\),\u_\star^k\)-\frac{1}{2}\(\rho^k,|\u_\star^{k} - \u^k|^2\)\nonumber\\
&\leq&\(\rho^k\(\u_\star^{k} - \u^k\),\u_\star^k\)\nonumber\\
&\leq&- \delta t_{k}\sum_{i=1}^M \(n_i^{k}\grad \mu_i^{k+1},\u_\star^k\)\nonumber\\
&\leq&\delta t_{k}\sum_{i=1}^M \(\div\(n_i^{k}\u_\star^k\), \mu_i^{k+1}\).
\end{eqnarray}
Combining   \eqref{eqDiscreteMultiCompononentMassConserveProof03}, \eqref{eqDecoupledDiscreteMultiCompononentMassConserveProof04} and \eqref{eqDecoupledDiscreteMultiCompononentMassConserveProof05} yields
 \begin{align}\label{eqDecoupledDiscreteMultiCompononentEnergyDecay}
\frac{E^{k+1}-E^{k}+\mathcal{F}^{k+1}-\mathcal{F}^{k}}{\delta t_{k}}
&\leq-\sum_{i,j=1}^{M}\big(\mathcal{M}^{k}_{ij}\nabla \mu_{i}^{k+1},\grad\mu_j^{k+1}\big)- \left\|\sqrt{\lambda^k}\div\u^{k+1}\right\|^2\nonumber\\
   &~~~~~~~ -\frac{1}{2} \left\|\sqrt{\eta^{k}}\(\nabla\u^{k+1}+\(\nabla\u^{k+1}\)^T\)\right\|^2\leq0,
   \end{align}
which yields  the energy dissipation \eqref{eqDecoupledDiscreteTotalEnergyDecay}. 
\end{proof}

\subsection{Component-wise, decoupled  semi-implicit  scheme}\label{secComponentwisescheme}
For the case that diffusion fluxes have a diagonal mobility tensor,    we can design  a component-wise, decoupled  semi-implicit  scheme, which  not only uncouples the tight relationship between molar densities and velocity, but also solves the mass balance equations  by a component-wise way. 

 We still use $\n^{k}= [n_1^k,n_2^k,\cdots,n_M^k]^T$ to denote the molar density vector at the integer time step $k$. Furthermore, we introduce the molar density vector at the fractional time step $\(k+\frac{i}{M}\)$ and denote it by  $ \n^{k+\frac{i}{M}}=\[n_1^{k+1},\cdots,n_i^{k+1},n_{i+1}^{k},\cdots,n_M^{k}\]^T$, where $0\leq i\leq M$;  in particular,  $ \n^{k+\frac{i}{M}}=\n^k$ for $i=0$ and $ \n^{k+\frac{i}{M}}=\n^{k+1}$ for $i=M$. The discrete chemical potential  $\mu_i^{k+\frac{i}{M}}(1\leq i\leq M)$ of component $i$ is defined  as
   \begin{subequations}\label{eqFullyDecoupledDiscreteChemicalPotential01}
 \begin{align}\label{eqFullyDecoupledDiscreteChemicalPotential01A}
\mu_i^{k+\frac{i}{M}}&=\frac{H^{k+\frac{i}{M}}+H^{k+\frac{i-1}{M}}}{2\sqrt{F_b(\n^{k})+\sum_{j=1}^MC_{T,j}N_j^t}}\mu_i^{b}\(\n^{k+\frac{i-1}{M}}\)\nonumber\\
&~~~-\sum_{j=1}^i\div\(c_{ij}\grad{n_j^{k+1}}\)-\sum_{j=i+1}^M\div\(c_{ij}\grad{n_j^{k}}\),
  \end{align}
  \begin{align}\label{eqFullyDecoupledDiscreteChemicalPotential01B}
\frac{H^{k+\frac{i}{M}}-H^{k+\frac{i-1}{M}}}{\delta t_k}=\int_\Omega\frac{\mu_i^{b}\(\n^{k+\frac{i-1}{M}}\)}{2\sqrt{F_b(\n^k)+\sum_{j=1}^MC_{T,j}N_j^t}}\frac{n_i^{k+1}-n_i^k}{\delta t_k}d\x.
  \end{align}
 \end{subequations}
   A component-wise intermediate  velocity $\u_\star^{k+\frac{i}{M}}$ is defined as
\begin{eqnarray}\label{eqFullyDecoupledDiscreteVelocityStar}
\u_\star^{k+\frac{i}{M}} = \u_\star^{k+\frac{i-1}{M}}-\frac{\delta t_{k}}{\rho^k}n_i^{k}\grad \mu_i^{k+\frac{i}{M}},~~~1\leq i\leq M,
\end{eqnarray}
where $\u_\star^{k+0}=\u_\star^{k}=\u^k$.  Let $\rho_i=M_{w,i}n_i$ be the mass density of component $i$, and then we introduce a mean intermediate  velocity $\u_{\star\star}^k$ as
\begin{eqnarray}\label{eqFullyDecoupledDiscreteVelocityStar02}
\u_{\star\star}^k = \sum_{i=1}^M\frac{\rho_i^{k}}{\rho^k}\u_\star^{k+\frac{i}{M}}.
\end{eqnarray}

We construct the semi-implicit time   scheme for the molar density balance equation \eqref{eqGeneralNSEQMass}  of component $i$  as
\begin{subequations}\label{eqFullyDecoupledScheme01}
 \begin{equation}\label{eqFullyDecoupledSchemeComponent}
 \frac{ n_i^{k+1}-n_i^{k}}{\delta t_{k}}+\div(n_i^{k}\u_\star^{k+\frac{i}{M}})+\div\J_i^{k+\frac{i}{M}} =0, 
 \end{equation}
 \begin{equation}\label{eqFullyDecoupledSchemeDiffusion}
\J_i^{k+\frac{i}{M}} = -\mathcal{M}^k_{i}\nabla \mu_{i}^{k+\frac{i}{M}},
\end{equation}
\end{subequations}
which is a linear equation of $n_i^{k+1}$ only and can be solved  sequently  from $i=1$ to $M$. The semi-implicit time scheme for the momentum balance equation is
\begin{eqnarray}\label{eqFullyDecoupledSchemeVelocity}
&& \rho^{k}\frac{\u^{k+1}-\u^{k}}{\delta t_{k}}+\rho^{k}\u_{\star\star}^k \cdot\grad\u^{k+1}+\sum_{i=1}^MM_{w,i}\J_i^{k+\frac{i}{M}}\cdot\grad\u^{k+1}=-\sum_{i=1}^Mn_i^{k}\grad \mu_i^{k+\frac{i}{M}}\nonumber\\
   &&~~+\nabla\(\lambda^k\div\u^{k+1}\)+\div\eta^k\(\nabla\u^{k+1}+\(\nabla\u^{k+1}\)^T\).
\end{eqnarray}
Summing \eqref{eqFullyDecoupledDiscreteVelocityStar} from $i=1$ to $M$ yields  
\begin{eqnarray}\label{eqFullyDecoupledDiscreteVelocityStar03}
\u_\star^{k+1} = \u^k-\frac{\delta t_{k}}{\rho^k}\sum_{i=1}^Mn_i^{k}\grad \mu_i^{k+\frac{i}{M}}.
\end{eqnarray}
Consequently, the equation \eqref{eqFullyDecoupledSchemeVelocity} can be reformulated as
\begin{eqnarray}\label{eqFullyDecoupledSchemeVelocity02}
&& \rho^{k}\frac{\u^{k+1}-\u_\star^{k+1}}{\delta t_{k}}+\rho^{k}\u_{\star\star}^k \cdot\grad\u^{k+1}+\sum_{i=1}^MM_{w,i}\J_i^{k+\frac{i}{M}}\cdot\grad\u^{k+1}\nonumber\\
   &&~~=\nabla\(\lambda^k\div\u^{k+1}\)+\div\eta^k\(\nabla\u^{k+1}+\(\nabla\u^{k+1}\)^T\).
\end{eqnarray}
This is a linear equation of velocity $\u^{k+1}$ and easy to be solved. 
In the convection term of \eqref{eqFullyDecoupledSchemeVelocity02}, we use the mean  intermediate  velocity $\u_{\star\star}^{k}$ instead of $\u^k$ or  $\u_\star^{k+\frac{i}{M}}$ to match  the mass balance equations.

It is apparent that the above component-wise  approach can be directly applied for the IEQ-based component-wise schemes and for the Cahn-Hilliard-type models studied in \cite{kousun2015CHE}.
 
We now prove that the component-wise, decoupled scheme satisfies the discrete energy dissipation law.

\begin{thm}\label{thmFullyDecoupledSchemeEnergy}
The sum of    the modified  Helmholtz free energy and   kinetic energy determined by  \eqref{eqFullyDecoupledScheme01} and \eqref{eqFullyDecoupledSchemeVelocity02} associated with \eqref{eqFullyDecoupledDiscreteChemicalPotential01},  \eqref{eqFullyDecoupledDiscreteVelocityStar} and \eqref{eqFullyDecoupledDiscreteVelocityStar02} is dissipated with time steps, i.e.
\begin{eqnarray}\label{eqFullyDecoupledDiscreteTotalEnergyDecay}
E^{k+1}+\mathcal{F}^{k+1} \leq E^{k}+\mathcal{F}^{k},
\end{eqnarray}
where $E^k$ and $\mathcal{F}^k$ are still defined as in \eqref{eqDiscreteTotalEnergy}.
\end{thm}
\begin{proof}
Using \eqref{eqFullyDecoupledDiscreteChemicalPotential01B}, we derive the   difference between $|H^{k+\frac{i}{M}}|^2$ and $|H^{k+\frac{i-1}{M}}|^2~ (1\leq i\leq M)$ as 
\begin{align}\label{eqFullyDecoupledSchemeEnergyProof01}
|H^{k+\frac{i}{M}}|^2-|H^{k+\frac{i-1}{M}}|^2&=   \(H^{k+\frac{i}{M}}+H^{k+\frac{i-1}{M}}\)\(H^{k+\frac{i}{M}}-H^{k+\frac{i-1}{M}}\)\nonumber\\
&=  \(\frac{\(H^{k+\frac{i}{M}}+H^{k+\frac{i-1}{M}}\)\mu_i^{b}\(\n^{k+\frac{i-1}{M}}\)}{2\sqrt{F_b(\n^k)+\sum_{j=1}^MC_{T,j}N_j^t}},n_i^{k+1}-n_i^k\).
\end{align}
The gradient contribution of Helmholtz free energy at the time step $(k+\frac{i}{M})$ can be expressed as
\begin{eqnarray}\label{eqFullyDecoupledSchemeEnergyProof02}
F_\grad^{k+\frac{i}{M}}&=&\frac{1}{2}\int_{\Omega}\sum_{j,l=1}^i c_{jl}\grad n_j^{k+1}\cdot\grad n_l^{k+1}d\x+\frac{1}{2}\int_{\Omega}\sum_{j,l=i+1}^Mc_{jl}\grad n_j^k\cdot\grad n_l^kd\x\nonumber\\
&&+\sum_{j=1}^i\sum_{l=i+1}^M\int_{\Omega}c_{jl}\grad n_j^{k+1}\cdot\grad n_l^kd\x.
\end{eqnarray}
Taking into account  $c_{ij}=c_{ji}$ and $c_{ij}>0$,  we derive
\begin{eqnarray}\label{eqFullyDecoupledSchemeEnergyProof03}
F_\grad^{k+\frac{i}{M}}-F_\grad^{k+\frac{i-1}{M}}&=&\frac{1}{2}\int_{\Omega}c_{ii}\(\grad n_i^{k+1}\cdot\grad n_i^{k+1}-\grad n_i^k\cdot\grad n_i^k\)d\x\nonumber\\
&&+\int_{\Omega}\sum_{j=1}^{i-1} c_{ij}\grad \(n_i^{k+1}-n_i^k\)\cdot\grad n_j^{k+1}d\x\nonumber\\
&&+\int_{\Omega}\sum_{j=i+1}^{M} c_{ij}\grad \(n_i^{k+1}-n_i^k\)\cdot\grad n_j^{k}d\x\nonumber\\
%&=&\(c_{ii}\grad \(n_i^{k+1}-n_i^k\),\grad n_i^{k+1}\)- \left\|c_{ii}^{1/2}\grad \(n_i^{k+1}-n_i^k\)\right\|^2\nonumber\\
%&&+\int_{\Omega}\sum_{j=1}^{i-1} c_{ij}\grad \(n_i^{k+1}-n_i^k\)\cdot\grad n_j^{k+1}d\x\nonumber\\
%&&+\int_{\Omega}\sum_{j=1}^{i+1} c_{ij}\grad \(n_i^{k+1}-n_i^k\)\cdot\grad n_j^{k}d\x\nonumber\\
&\leq&\sum_{j=1}^{i}\(\grad \(n_i^{k+1}-n_i^k\),c_{ij}\grad n_j^{k+1}\)\nonumber\\
&&+\sum_{j=i+1}^{M}\(\grad \(n_i^{k+1}-n_i^k\),c_{ij}\grad n_j^{k}\)\nonumber\\
&\leq&-\(n_i^{k+1}-n_i^{k},\sum_{j=1}^{i}\div c_{ij}\grad{n_j^{k+1}}+\sum_{j=i+1}^{M}\div c_{ij}\grad{n_j^{k}}\).
\end{eqnarray}
By the definition of $\mu_i^{k+\frac{i}{M}}$ given in \eqref{eqFullyDecoupledDiscreteChemicalPotential01}, we obtain from the estimates \eqref{eqFullyDecoupledSchemeEnergyProof01} and  \eqref{eqFullyDecoupledSchemeEnergyProof03} that
\begin{align}\label{eqFDCHEConvexSplittingPoreEnergyDecayProof05}
&|H^{k+\frac{i}{M}}|^2-|H^{k+\frac{i-1}{M}}|^2+F_\grad^{k+\frac{i}{M}}-F_\grad^{k+\frac{i-1}{M}}\nonumber\\
&~~\leq \(\mu_i^{k+\frac{i}{M}} ,n_i^{k+1}-n_i^k\)\nonumber\\
&~~\leq -\delta t_k\(\mu_i^{k+\frac{i}{M}} ,\div(n_i^{k}\u_\star^{k+\frac{i}{M}})-\div\mathcal{M}^k_{i}\nabla \mu_{i}^{k+\frac{i}{M}}\)\nonumber\\
&~~\leq \delta t_k\(\u_\star^{k+\frac{i}{M}} , n_i^{k}\grad\mu_i^{k+\frac{i}{M}}\) -\delta t_k\left\|\sqrt{\mathcal{M}^k_{i}}\nabla \mu_{i}^{k+\frac{i}{M}}\right\|^2,
\end{align}
where the second equality is obtained by using \eqref{eqFullyDecoupledScheme01}.
Summing  up \eqref{eqFDCHEConvexSplittingPoreEnergyDecayProof05} from $i=1$ to $M$   yields 
\begin{align}\label{eqFDCHEConvexSplittingPoreEnergyDecayProof06}
\mathcal{F}^{k+1}-\mathcal{F}^{k}&=|H^{k+1}|^2-|H^{k}|^2+F_\grad^{k+1} -F_\grad^{k}\nonumber\\
&=\sum_{i=1}^M\(|H^{k+\frac{i}{M}}|^2-|H^{k+\frac{i-1}{M}}|^2+F_\grad^{k+\frac{i}{M}}-F_\grad^{k+\frac{i-1}{M}}\)\nonumber\\
&\leq \delta t_k\sum_{i=1}^M\(\u_\star^{k+\frac{i}{M}} , n_i^{k}\grad\mu_i^{k+\frac{i}{M}}\) -\delta t_k\sum_{i=1}^M\left\|\sqrt{\mathcal{M}^k_{i}}\nabla \mu_{i}^{k+\frac{i}{M}}\right\|^2.
\end{align}

We define the intermediate  kinetic energy as $$E_\star^{k+\frac{i}{M}}=\frac{1}{2}\(\rho^k\u_\star^{k+\frac{i}{M}},\u_\star^{k+\frac{i}{M}}\).$$ 
Using the definition \eqref{eqFullyDecoupledDiscreteVelocityStar} of intermediate velocities, we derive
\begin{eqnarray}\label{eqFullyDecoupledDecoupledExplicitDiscreteVelocityStarProof01}
E_\star^{k+\frac{i}{M}}-E_\star^{k+\frac{i-1}{M}}&=&\frac{1}{2}\(\rho^k\u_\star^{k+\frac{i}{M}},\u_\star^{k+\frac{i}{M}}\)-\frac{1}{2}\(\rho^k\u_\star^{k+\frac{i-1}{M}},\u_\star^{k+\frac{i-1}{M}}\)\nonumber\\
&=&\(\rho^k\(\u_\star^{k+\frac{i}{M}} - \u_\star^{k+\frac{i-1}{M}}\),\u_\star^{k+\frac{i}{M}}\)-\frac{1}{2}\(\rho^k,\left|\u_\star^{k+\frac{i}{M}} - \u_\star^{k+\frac{i-1}{M}}\right|^2\)\nonumber\\
&\leq&\(\rho^k\(\u_\star^{k+\frac{i}{M}} - \u_\star^{k+\frac{i-1}{M}}\),\u_\star^{k+\frac{i}{M}}\)\nonumber\\
&=&- \delta t_{k} \(n_i^{k}\grad \mu_i^{k+\frac{i}{M}},\u_\star^{k+\frac{i}{M}}\).
\end{eqnarray}
The sum of \eqref{eqFullyDecoupledScheme01} multiplied by $M_{w,i}$ leads to the overall  mass balance equation 
\begin{eqnarray}\label{eqFullyDecoupledDiscreteMultiCompononentMassConserveProof02}
\frac{ \rho^{k+1}-\rho^{k}}{\delta t_{k}}=-\div\(\rho^{k}\u_{\star\star}^k\) -\sum_{i=1}^MM_{w,i}\div\J_i^{k+\frac{i}{M}},
\end{eqnarray}
where \eqref{eqFullyDecoupledDiscreteVelocityStar02} is also used to get the first term on the right-hand side.
We  estimate the difference between $E^{k+1}$ and $E_\star^k$ as
\begin{align}\label{eqFullyDecoupledDiscreteMultiComponentMomentumConserveProof01}
\frac{E^{k+1}-E_\star^{k+1}}{\delta t_k}&=\frac{1}{2\delta t_k}\(\rho^{k+1},|\u^{k+1}|^2\)-\frac{1}{2\delta t_k}\(\rho^k,|\u_\star^{k+1}|^2\)\nonumber\\
    &\leq \(\rho^{k}\frac{\u^{k+1}-\u_\star^{k+1}}{\delta t_k},\u^{k+1}\)+\frac{1}{2}\(\frac{\rho^{k+1}-\rho^{k}}{\delta t_k},|\u^{k+1}|^2\)\nonumber\\
&\leq -\(\rho^{k}\u_{\star\star}^k \cdot\grad\u^{k+1}+\sum_{i=1}^MM_{w,i}\J_i^{k+\frac{i}{M}}\cdot\grad\u^{k+1},\u^{k+1}\)\nonumber\\
    &~~~+\(\nabla\(\lambda^k\div\u^{k+1}\)+\div\eta^k\(\nabla\u^{k+1}+\(\nabla\u^{k+1}\)^T\),\u^{k+1}\)\nonumber\\
    &~~~-\frac{1}{2}\(\div\(\rho^{k}\u_{\star\star}^k\)+\sum_{i=1}^MM_{w,i}\div\J_i^{k+\frac{i}{M}},|\u^{k+1}|^2\)\nonumber\\
    &\leq- \left\|\sqrt{\lambda^k}\div\u^{k+1}\right\|^2 -\frac{1}{2} \left\|\sqrt{\eta^{k}}\(\nabla\u^{k+1}+\(\nabla\u^{k+1}\)^T\)\right\|^2,
\end{align}
where the third equality is obtained by using \eqref{eqFullyDecoupledSchemeVelocity02} and \eqref{eqFullyDecoupledDiscreteMultiCompononentMassConserveProof02}.
Combining  \eqref{eqFullyDecoupledDecoupledExplicitDiscreteVelocityStarProof01} and \eqref{eqFullyDecoupledDiscreteMultiComponentMomentumConserveProof01} yields 
 \begin{align}\label{eqFullyDecoupledDiscreteMultiComponentMomentumConserveProof02}
E^{k+1}-E^{k}&=E^{k+1}-E_\star^{k+1}+\sum_{i=1}^M\(E_\star^{k+\frac{i}{M}}-E_\star^{k+\frac{i-1}{M}}\)\nonumber\\
    &\leq - \delta t_k\left\|\sqrt{\lambda^k}\div\u^{k+1}\right\|^2 -\frac{1}{2}\delta t_k \left\|\sqrt{\eta^{k}}\(\nabla\u^{k+1}+\(\nabla\u^{k+1}\)^T\)\right\|^2\nonumber\\
    &~~~- \delta t_{k}\sum_{i=1}^M \(n_i^{k}\grad \mu_i^{k+\frac{i}{M}},\u_\star^{k+\frac{i}{M}}\).
\end{align}

Finally,   it is derived from \eqref{eqFDCHEConvexSplittingPoreEnergyDecayProof06} and  \eqref{eqFullyDecoupledDiscreteMultiComponentMomentumConserveProof02} that
 \begin{align}\label{eqDecoupledDiscreteMultiCompononentEnergyDecay}
\frac{E^{k+1}-E^{k}+\mathcal{F}^{k+1}-\mathcal{F}^{k}}{\delta t_{k}}&\leq-\left\|\sqrt{\mathcal{M}^k_{i}}\nabla \mu_{i}^{k+\frac{i}{M}}\right\|^2- \left\|\sqrt{\lambda^k}\div\u^{k+1}\right\|^2\nonumber\\
   & ~~~-\frac{1}{2} \left\|\sqrt{\eta^{k}}\(\nabla\u^{k+1}+\(\nabla\u^{k+1}\)^T\)\right\|^2,
   \end{align}
which yields  the energy dissipation \eqref{eqFullyDecoupledDiscreteTotalEnergyDecay}. 
\end{proof}

 %%%%%%%%%%%%%%% Conclusions %%%%%%%%%
\section{Numerical tests}

%%%%%%%%%%%%%%%%%%%%%%%%%%%%%%%%%%%%%%%%
In this section,  the proposed methods are applied to simulate  multi-component two-phase   flow problems.  We consider a binary   mixture and a ternary mixture in   a square  domain $\Omega$ with the length $20$~nm. The  boundary conditions are taken as $\u=0,~\J_i\cdot\bm{\gamma}_{\partial\Omega}=0$ and $\grad n_i\cdot\bm{\gamma}_{\partial\Omega}=0$ on the boundary $\partial\Omega$, where $\bm\gamma_{\partial\Omega}$ is the normal unit outward vector  to  $\partial \Omega$. 
For spatial discretization,     a  uniform  rectangular mesh with $40\times40$ elements is used;  the cell-centered finite difference method and the upwind scheme are employed  to discretize  the mass balance equation;  the finite volume method  on the staggered mesh \cite{Tryggvason2011book}  is applied for the momentum balance equation. We note that  the above spatial discretization methods  have  equivalent relationships with special mixed finite element methods under specified  quadrature rules \cite{arbogast1997mixed,Girault1996Mac}.  The energy parameters are chosen as $C_{T,i}=0$. 

\subsection{Binary mixture}
In this example, we consider a binary   mixture composed of methane (C$_1$) and pentane (C$_5$) at a constant temperature 310 K.  At the initial time, a square shape  droplet is located in   the center of the domain.  The initial gas molar densities of C$_1$ and C$_5$     are $7.4302$ kmol/m$^3$ and $0.6736$ kmol/m$^3$ respectively, while the initial liquid molar densities of C$_1$ and C$_5$  are $6.8663$ kmol/m$^3$ and $ 4.7915$ kmol/m$^3$  respectively. The initial molar density distributions for  C$_1$ and C$_5$ are illustrated in  Figure \ref{SquareC1andC5MolarDensityOfC5Temperature310K}(a) and (d) respectively.  We use the diffusion mobility formulation given by  \eqref{eqMultiCompononentMassConserveDiffusionChoiceB}  with the coefficients $\mathcal{D}_{12}=\mathcal{D}_{21}=10^{-8}$ m$^2$/s. The volumetric  viscosity and    the shear viscosity are taken  as $\xi=\eta=10^{-4}$ Pa$\cdot$s.  The time step size is taken as $10^{-12}$ s, and 200  time steps are simulated. 

The  velocity-density decoupled  method proposed in Sub-section \ref{secVelocityDensityDecoupledScheme} is applied to simulate the dynamical evolution of the square-shaped droplet. In Figure \ref{SquareC1andC5EnergyTemperature310}, we show the evolution profiles of the modified  total energy (i.e., the sum of    the modified  Helmholtz free energy and   kinetic energy) with time steps;  we also depict the original total energy (i.e., the sum of    the original  Helmholtz free energy and   kinetic energy) for the sake of comparison. 
It is observed from Figure \ref{SquareC1andC5EnergyTemperature310}(a) that although the modified total (free) energy is slightly less than the original energy, both  total (free) energies are   strictly dissipated with time steps, and moreover,   Figure \ref{SquareC1andC5EnergyTemperature310}(b), which is a zoom-in plot of Figure \ref{SquareC1andC5EnergyTemperature310}(a)  in the later time steps, demonstrates  that both total (free) energies remain to decrease.   As a result, the proposed method can preserve the energy-dissipation  feature.

  Figure \ref{SquareC1andC5MolarDensityOfC5Temperature310K} depicts the evolution process  of each component molar density, and it   is clearly observed that the droplet is gradually reshaping  to a circle from its initial square shape due to   chemical potential gradients.  In Figures \ref{SquareC1andC5VelocityTemperature310K}, we show  the fluid motion  driven by chemical potential gradients, including  the velocity field and  magnitudes of both velocity components.

\begin{figure}
%            \centering \subfigure[]{
%            \begin{minipage}[b]{0.45\textwidth}
%               \centering
%             \includegraphics[width=0.95\textwidth,height=2in]{figs/SquareC1andC5freeEnergyTemperature310.eps}
%            \end{minipage}
%            }
            \centering \subfigure[]{
            \begin{minipage}[b]{0.45\textwidth}
            \centering
             \includegraphics[width=0.95\textwidth,height=2in]{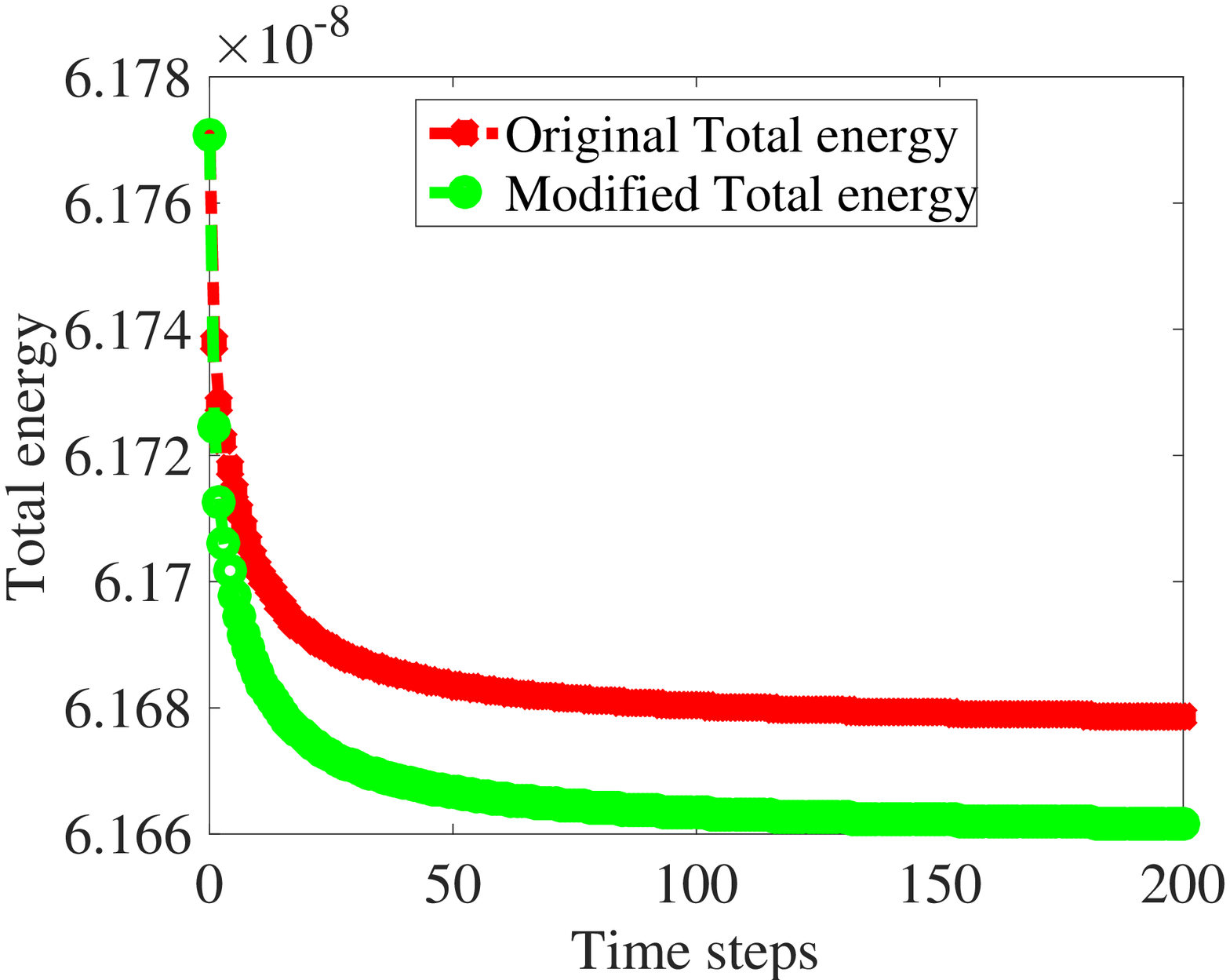}
            \end{minipage}
            }
            \centering \subfigure[]{
            \begin{minipage}[b]{0.45\textwidth}
            \centering
             \includegraphics[width=0.95\textwidth,height=2in]{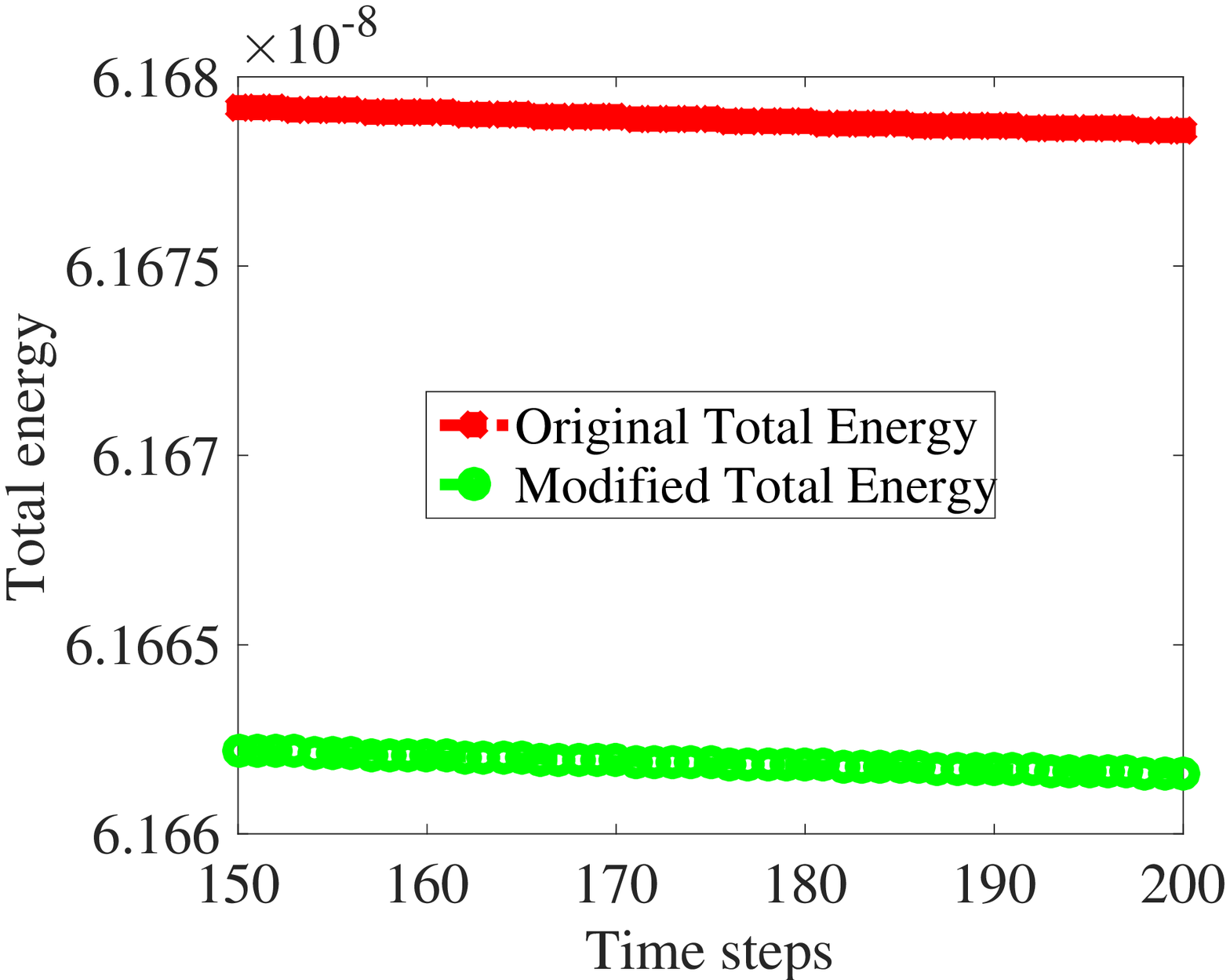}
            \end{minipage}
            }
           \caption{Binary mixture:  total energy dissipation with time steps.}
            \label{SquareC1andC5EnergyTemperature310}
 \end{figure}

\begin{figure}
            \centering \subfigure[C$_1$ at the initial time]{
            \begin{minipage}[b]{0.3\textwidth}
               \centering
             \includegraphics[width=\textwidth,height=0.9\textwidth]{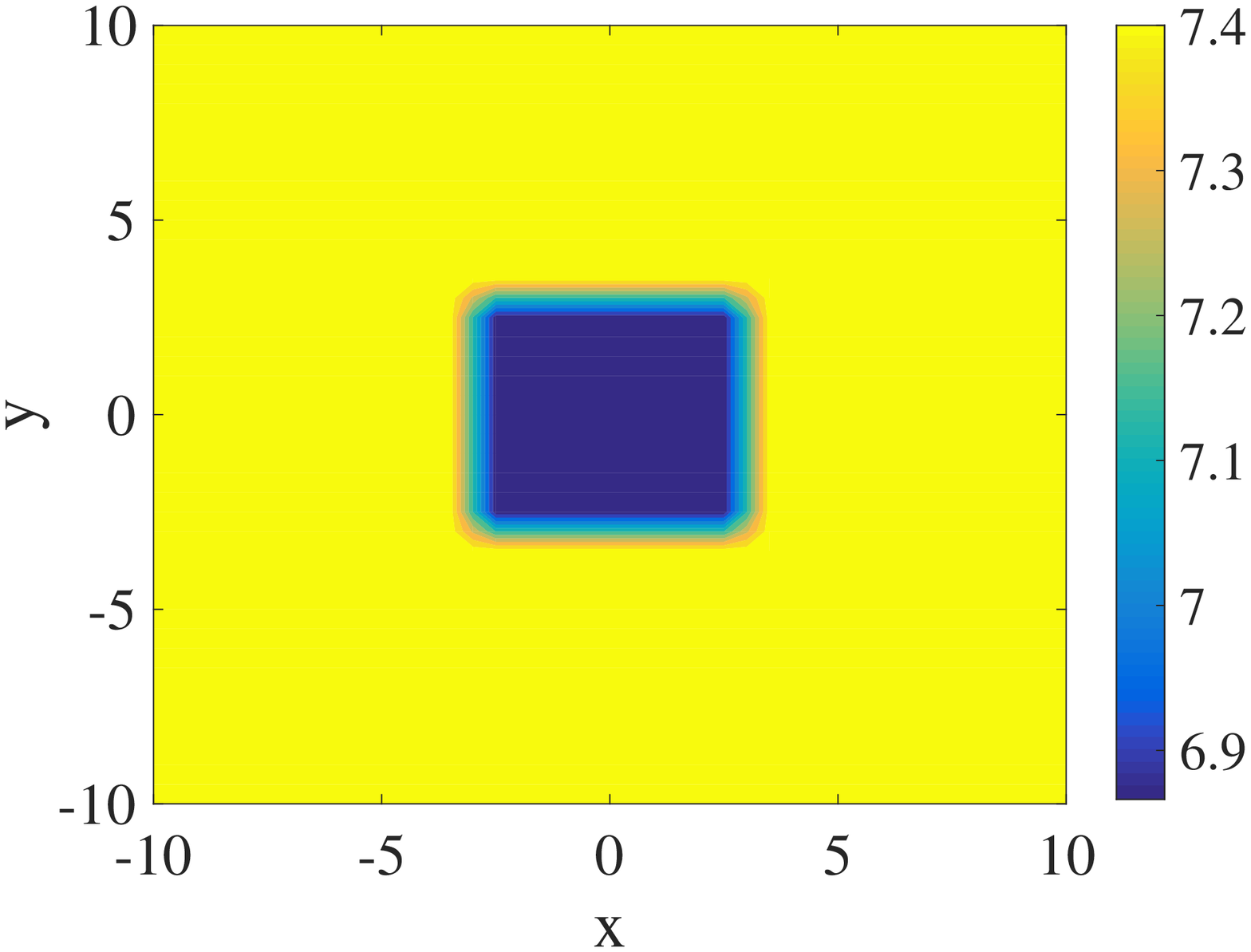}
            \end{minipage}
            }
            \centering \subfigure[C$_1$ at the 80th time step]{
            \begin{minipage}[b]{0.3\textwidth}
            \centering
             \includegraphics[width=\textwidth,height=0.9\textwidth]{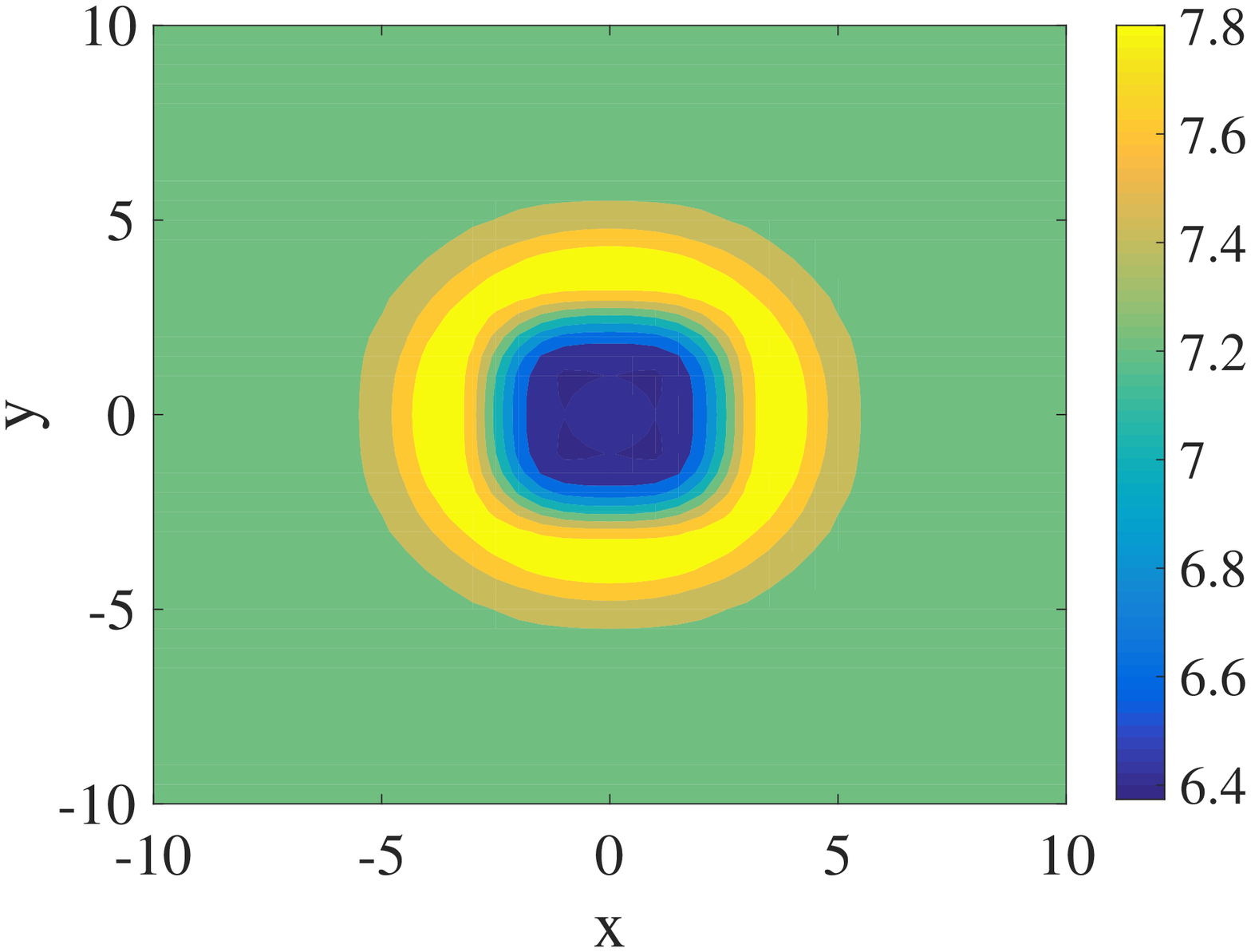}
            \end{minipage}
            }
           \centering \subfigure[C$_1$ at the 200th time step]{
            \begin{minipage}[b]{0.3\textwidth}
               \centering
             \includegraphics[width=\textwidth,height=0.9\textwidth]{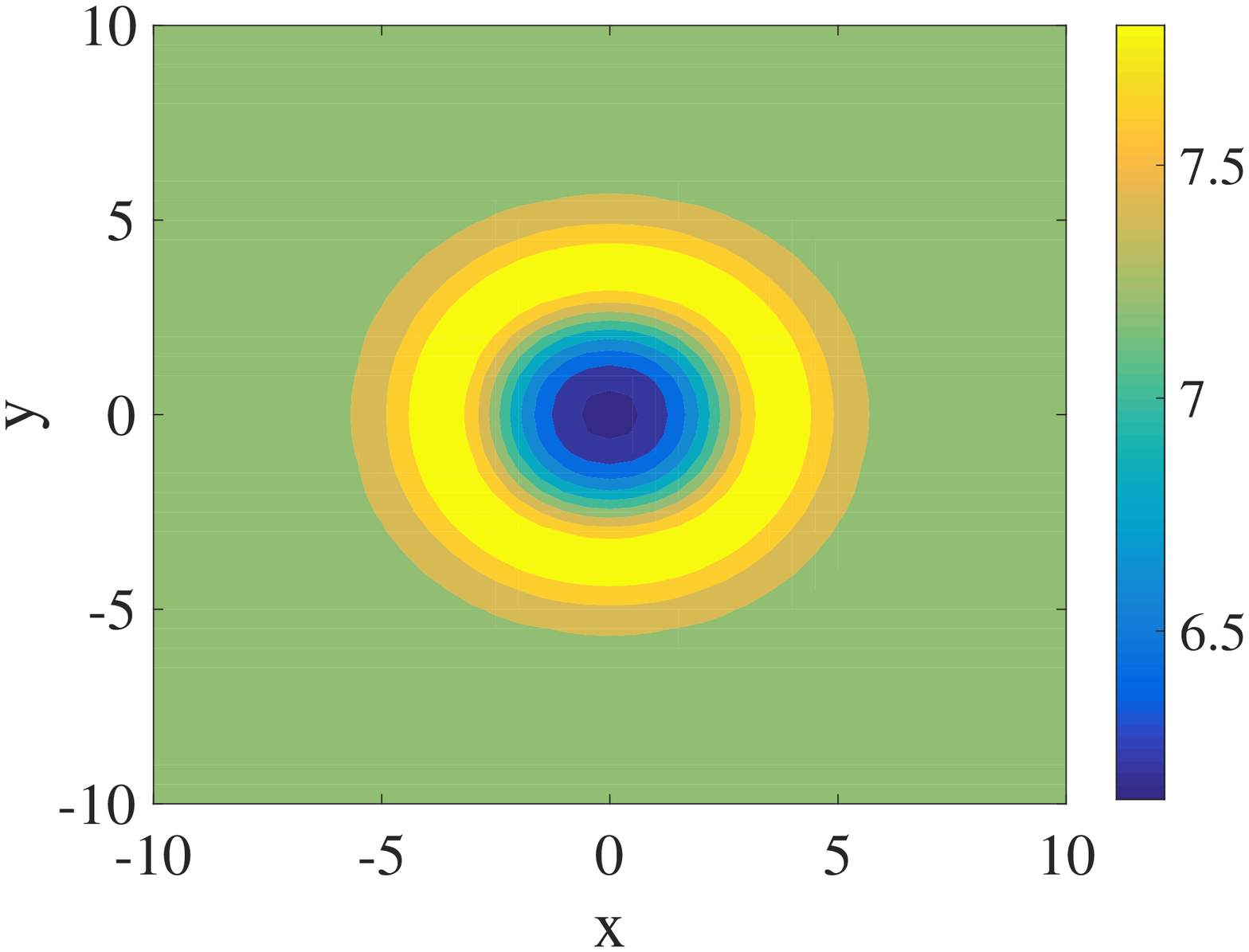}
            \end{minipage}
            }
            \centering \subfigure[C$_5$ at the initial time]{
            \begin{minipage}[b]{0.3\textwidth}
               \centering
             \includegraphics[width=\textwidth,height=0.9\textwidth]{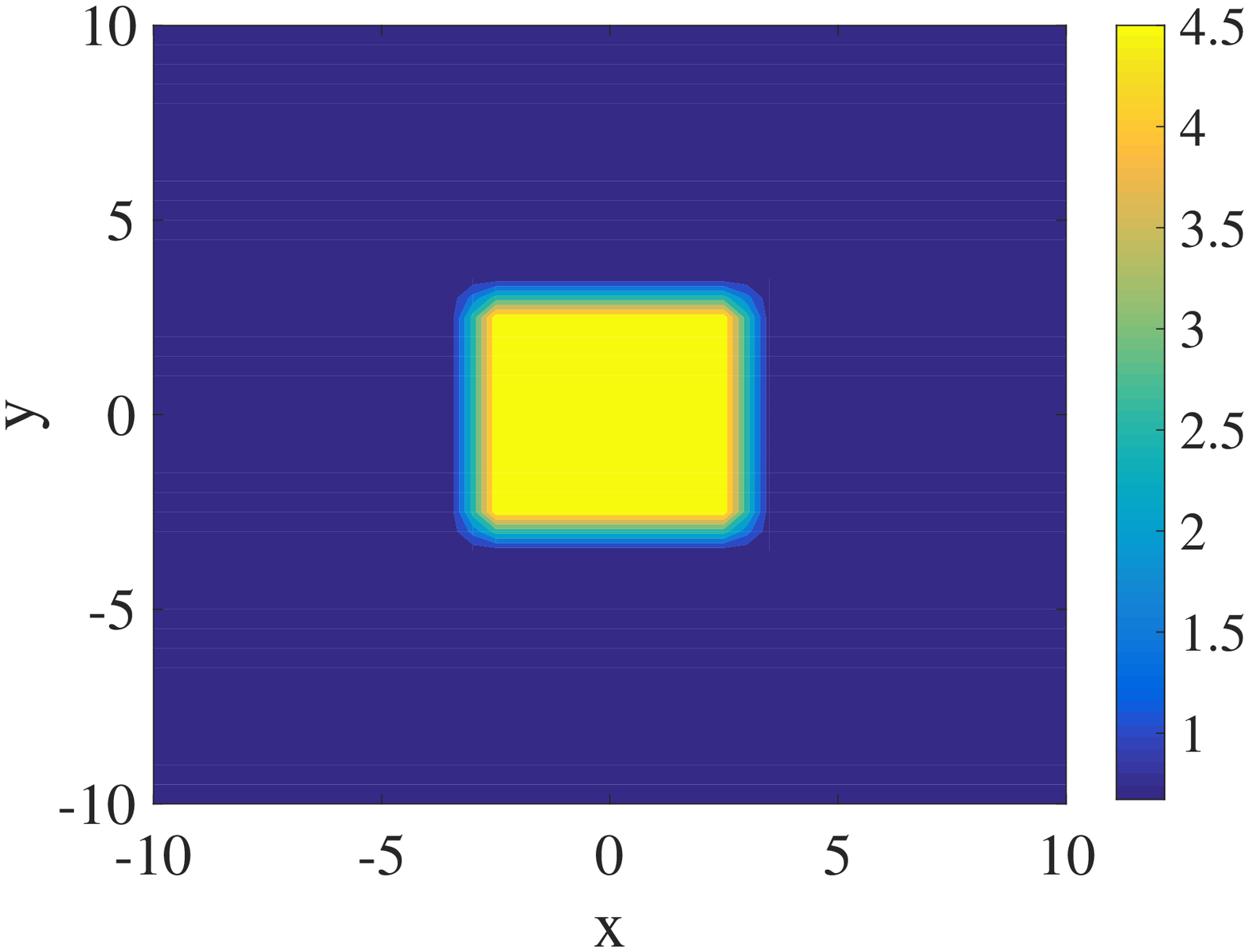}
            \end{minipage}
            }
            \centering \subfigure[C$_5$ at the 80th time step]{
            \begin{minipage}[b]{0.3\textwidth}
            \centering
             \includegraphics[width=\textwidth,height=0.9\textwidth]{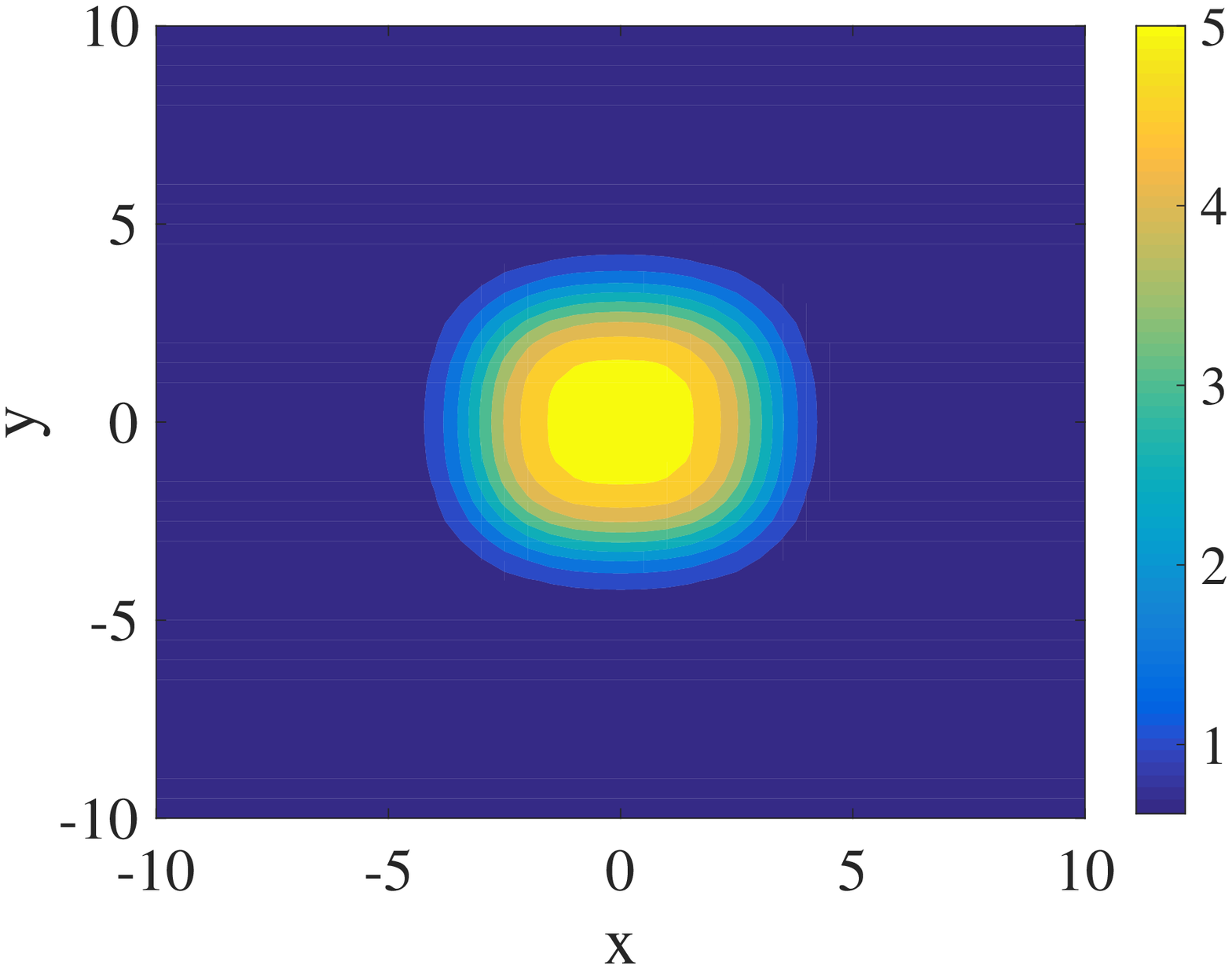}
            \end{minipage}
            }
           \centering \subfigure[C$_5$ at the 200th time step]{
            \begin{minipage}[b]{0.3\textwidth}
               \centering
             \includegraphics[width=\textwidth,height=0.9\textwidth]{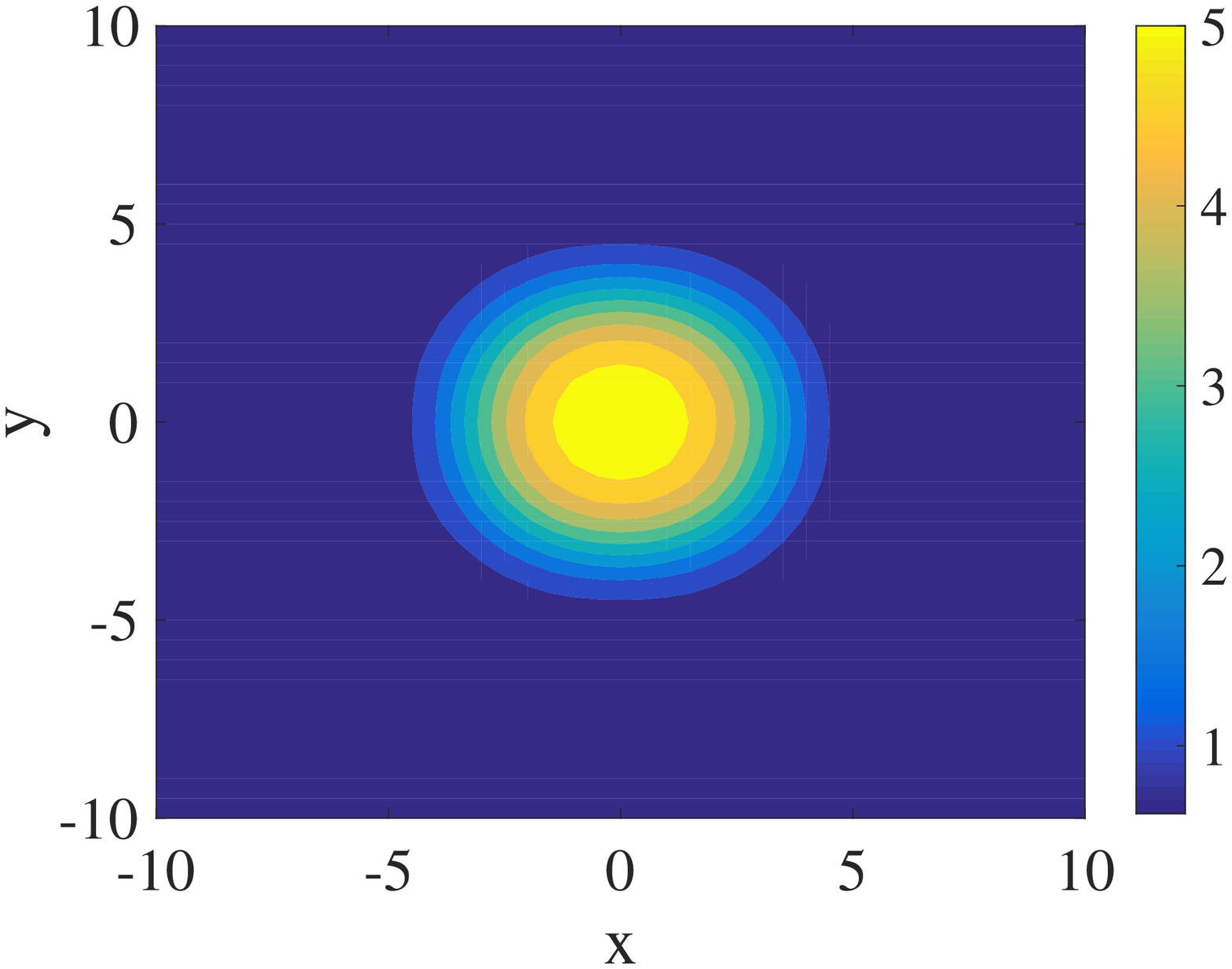}
             \end{minipage}
            }
            \caption{Binary mixture:  molar densities of C$_1$ and C$_5$  at  different time steps.}
            \label{SquareC1andC5MolarDensityOfC5Temperature310K}
 \end{figure}

\begin{figure}
           \centering \subfigure[]{
            \begin{minipage}[b]{0.3\textwidth}
            \centering
             \includegraphics[width=0.95\textwidth,height=0.9\textwidth]{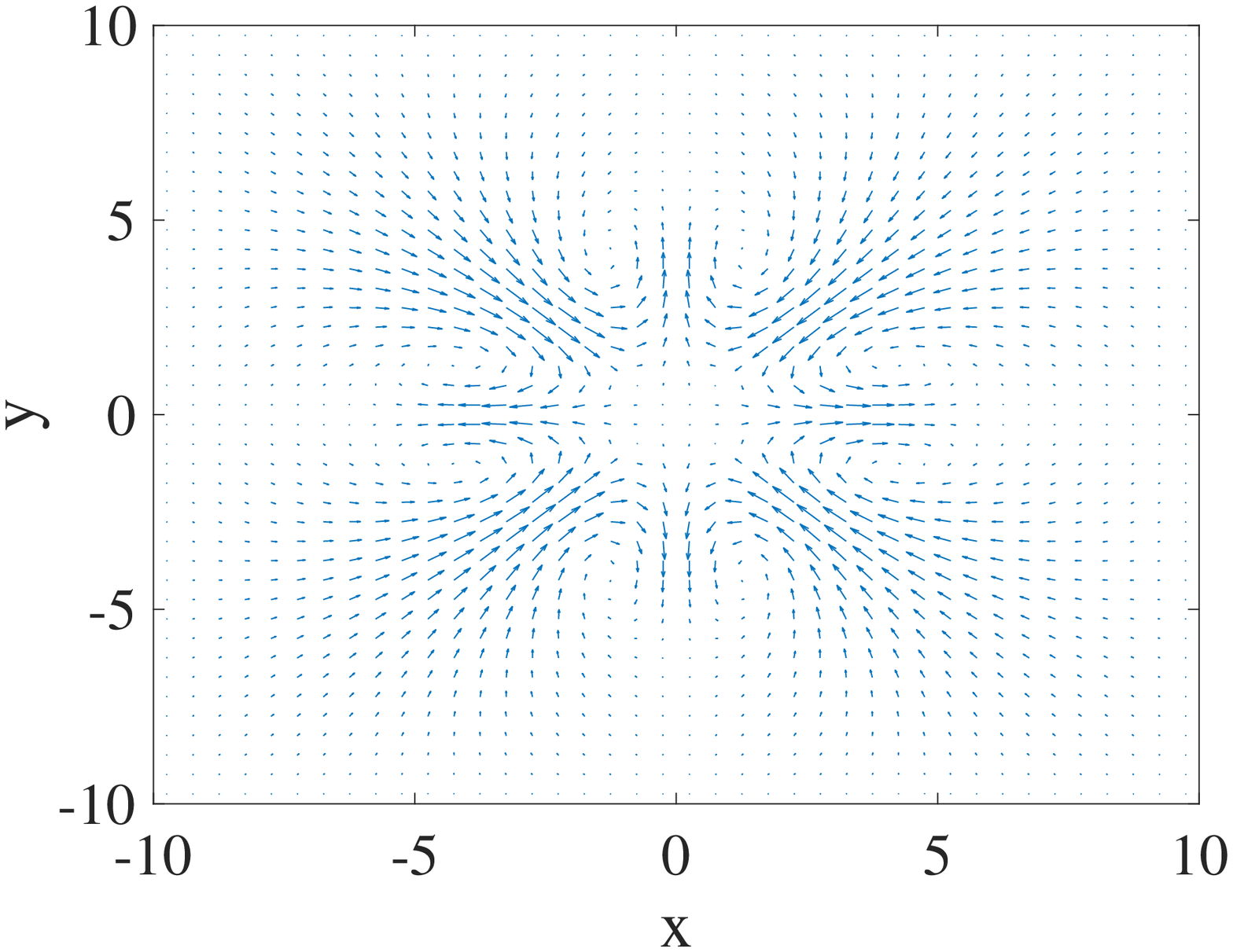}
            \end{minipage}
            }
           \centering \subfigure[]{
            \begin{minipage}[b]{0.3\textwidth}
            \centering
             \includegraphics[width=0.95\textwidth,height=0.9\textwidth]{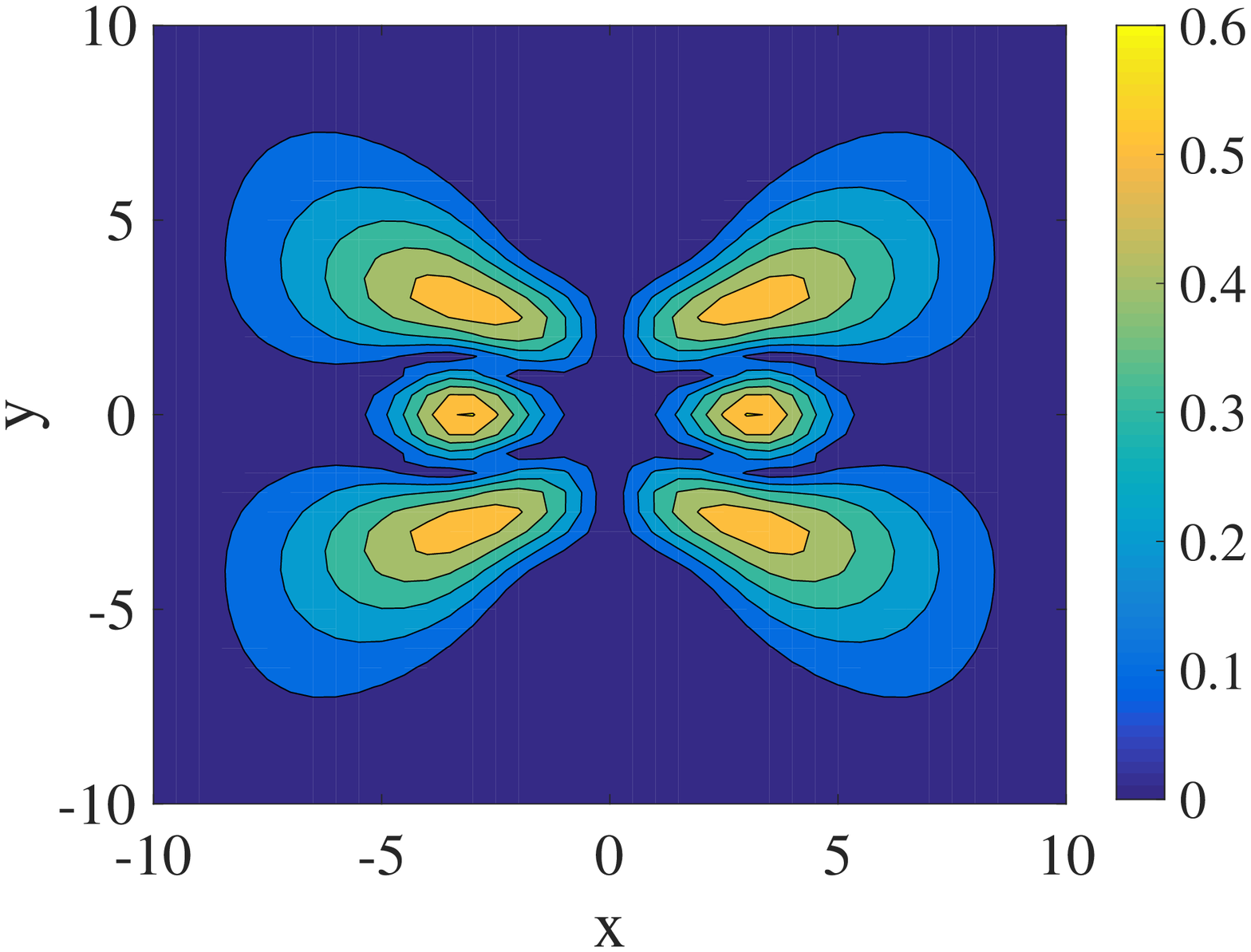}
            \end{minipage}
            }
           \centering \subfigure[]{
            \begin{minipage}[b]{0.3\textwidth}
            \centering
             \includegraphics[width=0.95\textwidth,height=0.9\textwidth]{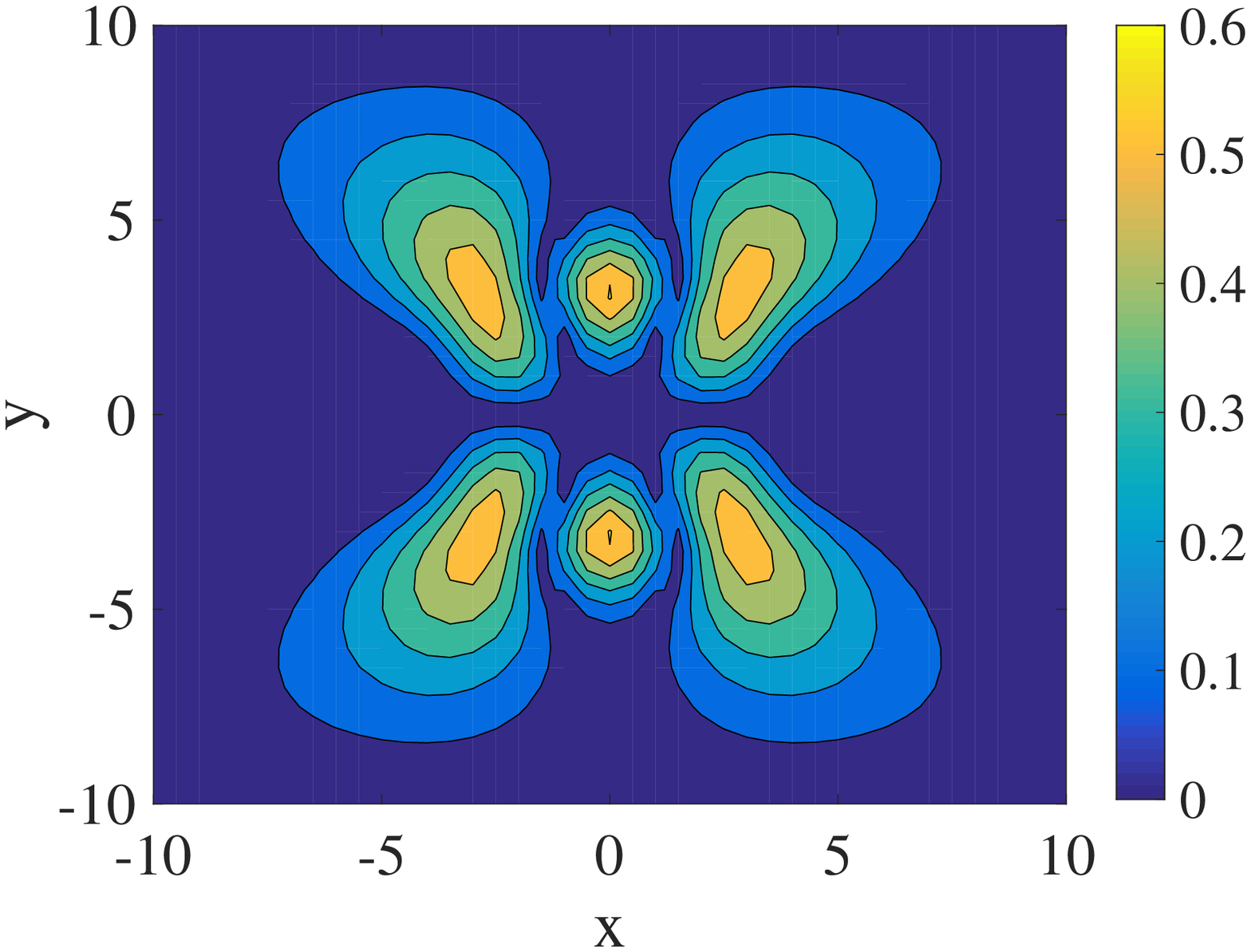}
            \end{minipage}
            }
%           \caption{Example 1:   flow quivers (left column), magnitude contours of $x$-direction velocity component (center column), and magnitude contours of $y$-direction velocity component (right column)   at the 1th(top row), 20th(center row), and 45th(bottom row) time step  respectively.}
           \caption{Binary mixture:   (a) flow quiver, (b) magnitude contour of $x$-direction velocity component, and (c) magnitude contour of $y$-direction velocity component   at the  80th  time step.}
            \label{SquareC1andC5VelocityTemperature310K}
 \end{figure}

\subsection{Ternary mixture}
In this example, we consider a ternary   mixture composed of methane (C$_1$) pentane (C$_5$) and  decane (C$_{10}$) at a constant temperature 323 K.   The initial gas molar densities of C$_1$, C$_5$  and C$_{10}$   are $10.516$ kmol/m$^3$,  $0.77$ kmol/m$^3$ and  $0.184$ kmol/m$^3$ respectively, while the initial liquid molar densities of C$_1$, C$_5$  and C$_{10}$  are $7.8412$ kmol/m$^3$,  $1.9925$ kmol/m$^3$ and $1.433$ kmol/m$^3$  respectively.  At the initial time, there are two square-shaped  droplets in   the  domain, as shown in the first figures of Figures \ref{MCTwoSquareCh4andTwoHydrocarbonsMolarDensityOfCH4Temperature323K}, \ref{MCTwoSquareCh4andTwoHydrocarbonsMolarDensityOfnC5Temperature323K} and \ref{MCTwoSquareCh4andTwoHydrocarbonsMolarDensityOfnC10Temperature323K} respectively.  The diffusion fluxes are  formulated by  \eqref{eqMultiCompononentMassConserveDiffusionChoiceA}  with  the diffusion coefficients $D_i=3\times10^{-8}$ m$^2$/s $(1\leq i\leq3)$. The volumetric  viscosity and    the shear viscosity are set  as $\xi=\eta=10^{-4}$ Pa$\cdot$s.  We take the time step size equal to $10^{-12}$ s, and we simulate the evolution process for 1000  time steps. 

We employ the  component-wise,  decoupled numerical  scheme proposed in Sub-section \ref{secComponentwisescheme}. The original and modified  total energies and their zoom-in plots are shown in Figure \ref{MCTwoSquareCh4andTwoHydrocarbonsTotalEnergyTemperature323}. We still see that both of total energies are dissipated  with time steps. We also note that in practical computations, the component-wise method is really effective for the mixtures composed of multiple components since it only needs to solve one more mass-balance equation as a new component is added.

Figures \ref{MCTwoSquareCh4andTwoHydrocarbonsMolarDensityOfCH4Temperature323K}, \ref{MCTwoSquareCh4andTwoHydrocarbonsMolarDensityOfnC5Temperature323K} and \ref{MCTwoSquareCh4andTwoHydrocarbonsMolarDensityOfnC10Temperature323K} illustrate  the  molar density configurations of  three components.  In Figure \ref{MCTwoSquareCh4andTwoHydrocarbonsVelocityTemperature323K}, we depict the fluid motion, including   the velocity fields and  magnitudes of  velocity components,  at different time steps. The simulation results show that due to   chemical potential gradients,   two droplets      are first  emerging  with each other, and at the later time,  the merged  droplets are gradually  reshaping  into a circle.

\begin{figure}
            \centering \subfigure[]{
            \begin{minipage}[b]{0.45\textwidth}
               \centering
             \includegraphics[width=0.95\textwidth,height=2in]{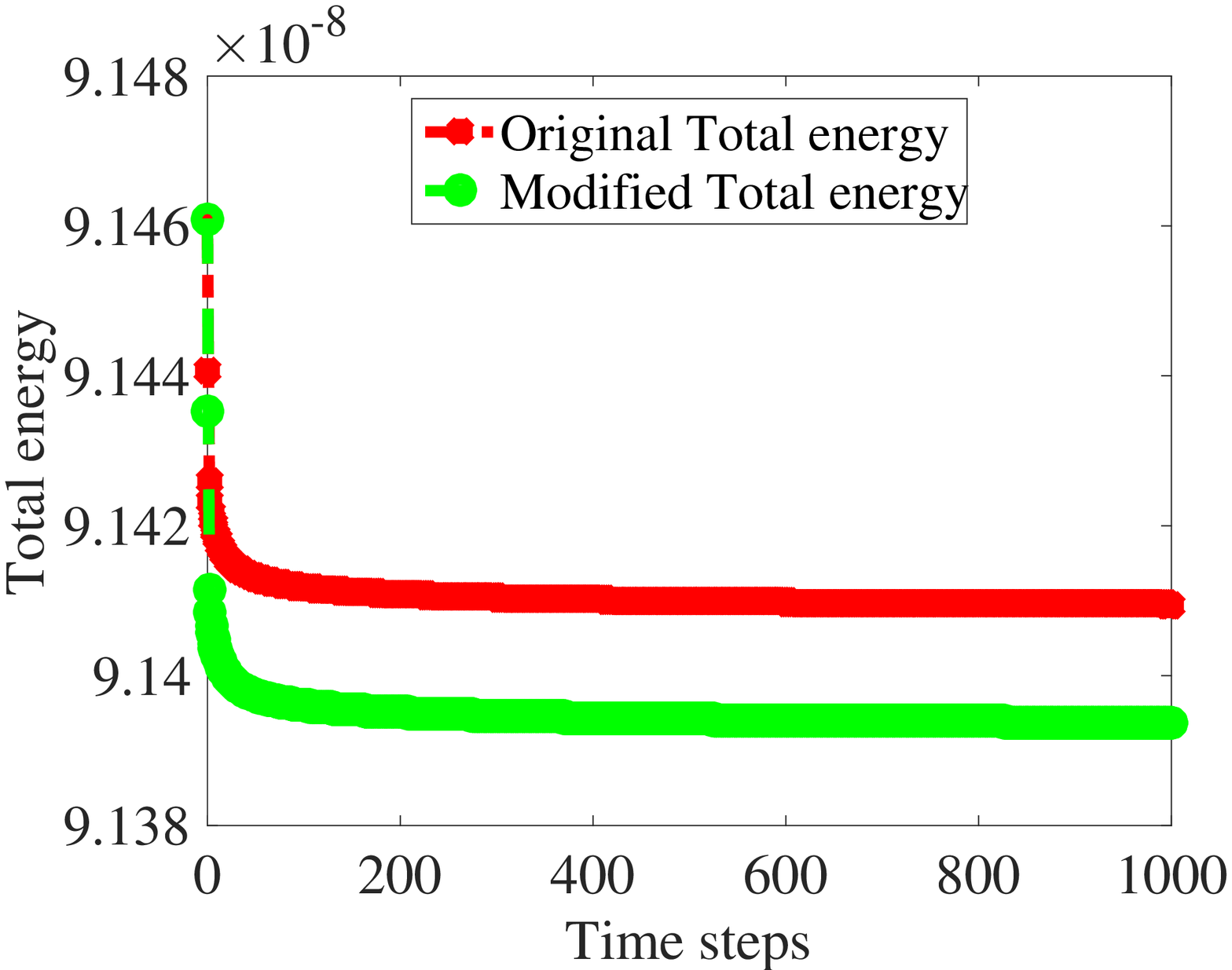}
            \end{minipage}
            }
            \centering \subfigure[]{
            \begin{minipage}[b]{0.45\textwidth}
            \centering
             \includegraphics[width=0.95\textwidth,height=2in]{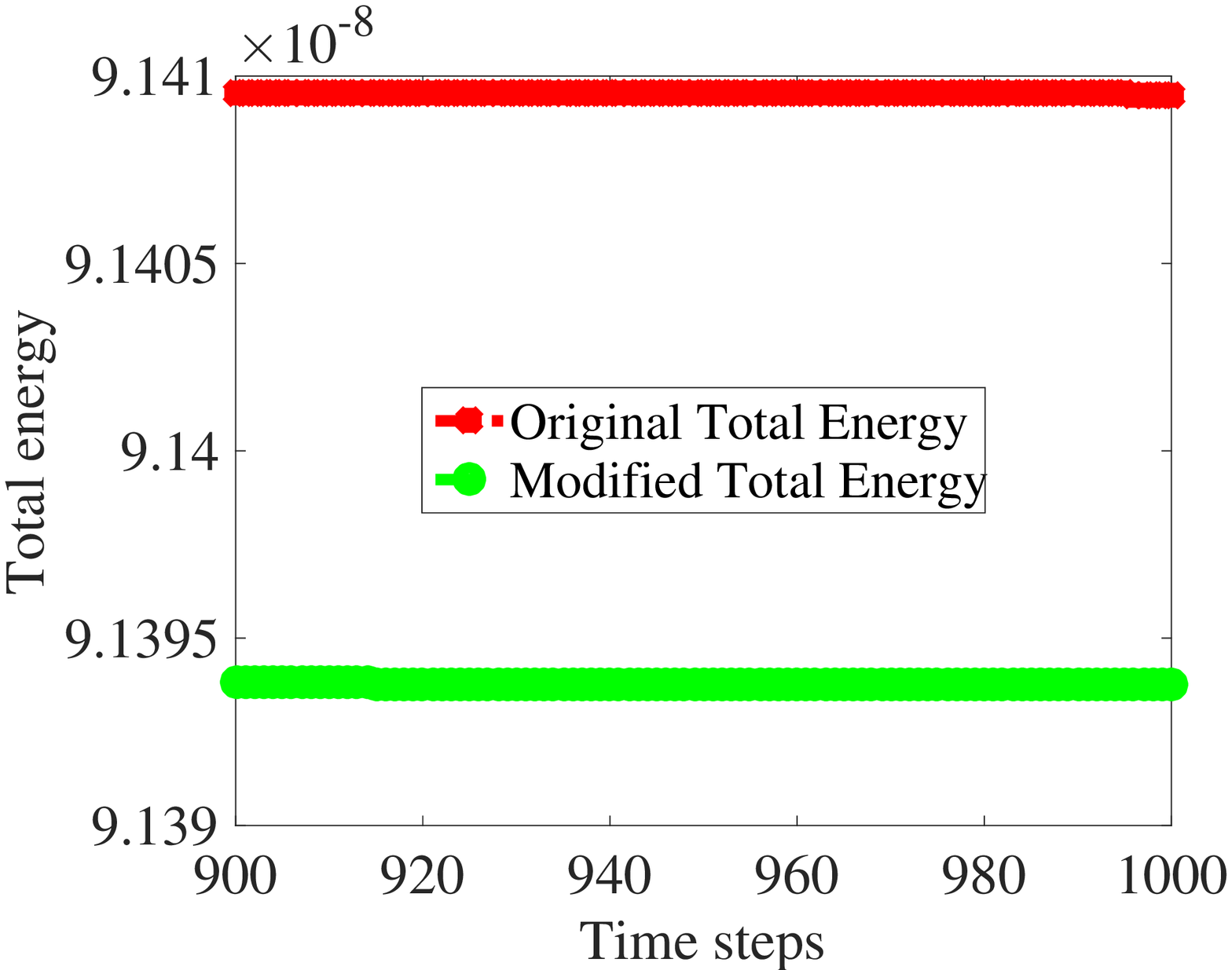}
            \end{minipage}
            }
           \caption{Ternary mixture: energy dissipation with time steps.}
            \label{MCTwoSquareCh4andTwoHydrocarbonsTotalEnergyTemperature323}
 \end{figure}

\begin{figure}
            \centering \subfigure[]{
            \begin{minipage}[b]{0.3\textwidth}
               \centering
             \includegraphics[width=\textwidth,height=0.9\textwidth]{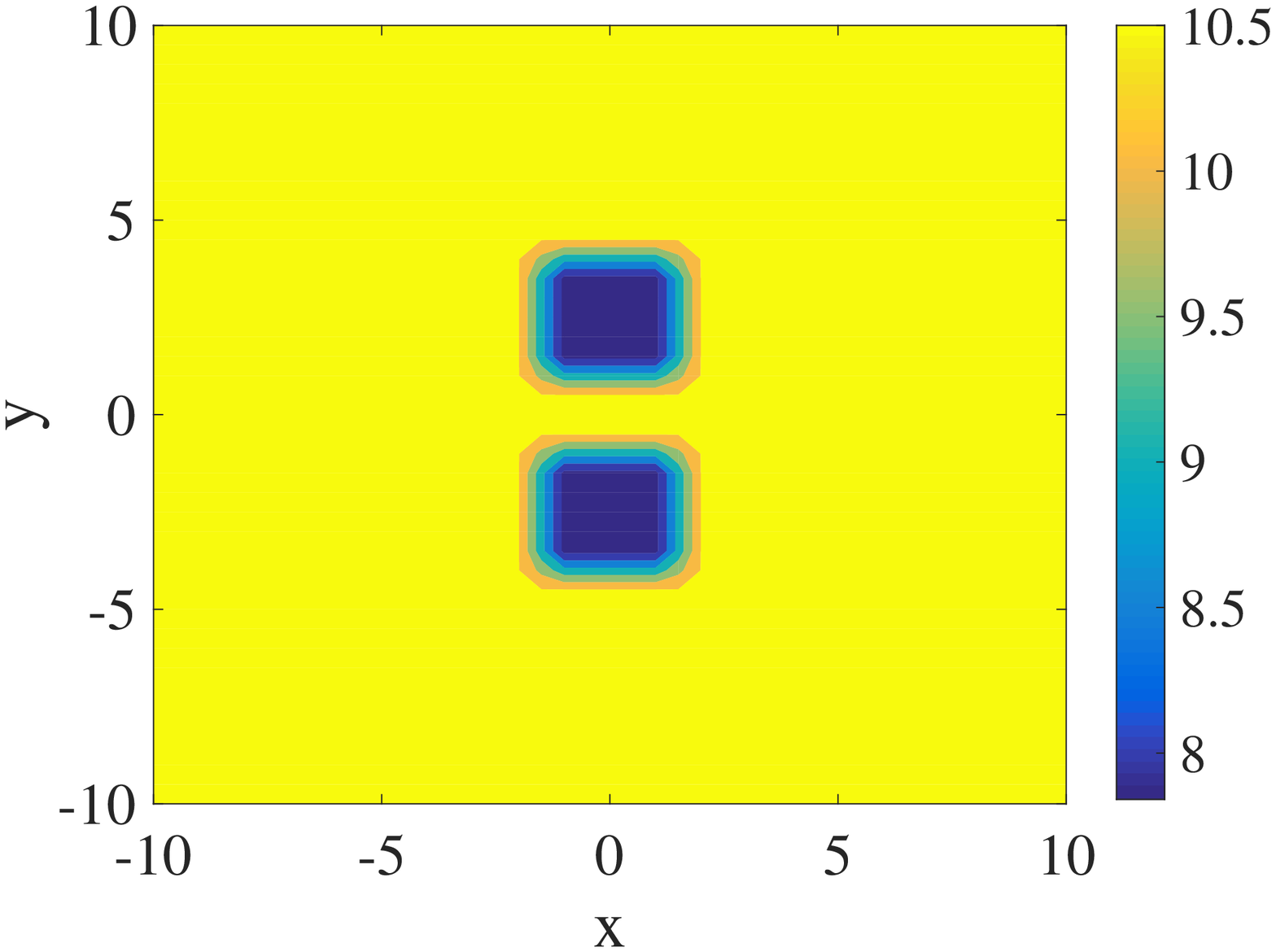}
            \end{minipage}
            }
            \centering \subfigure[]{
            \begin{minipage}[b]{0.3\textwidth}
            \centering
             \includegraphics[width=\textwidth,height=0.9\textwidth]{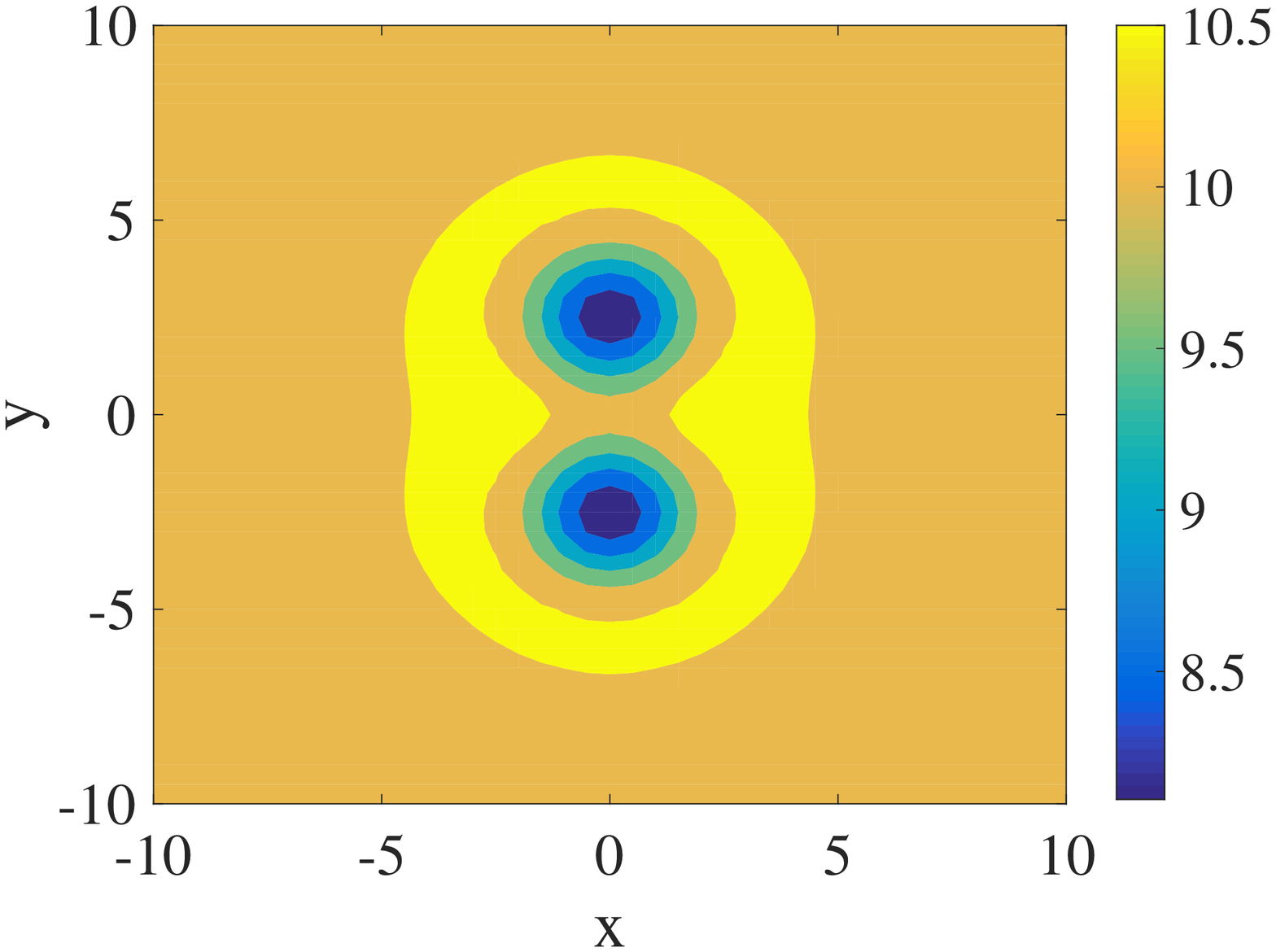}
            \end{minipage}
            }
           \centering \subfigure[]{
            \begin{minipage}[b]{0.3\textwidth}
            \centering
             \includegraphics[width=\textwidth,height=0.9\textwidth]{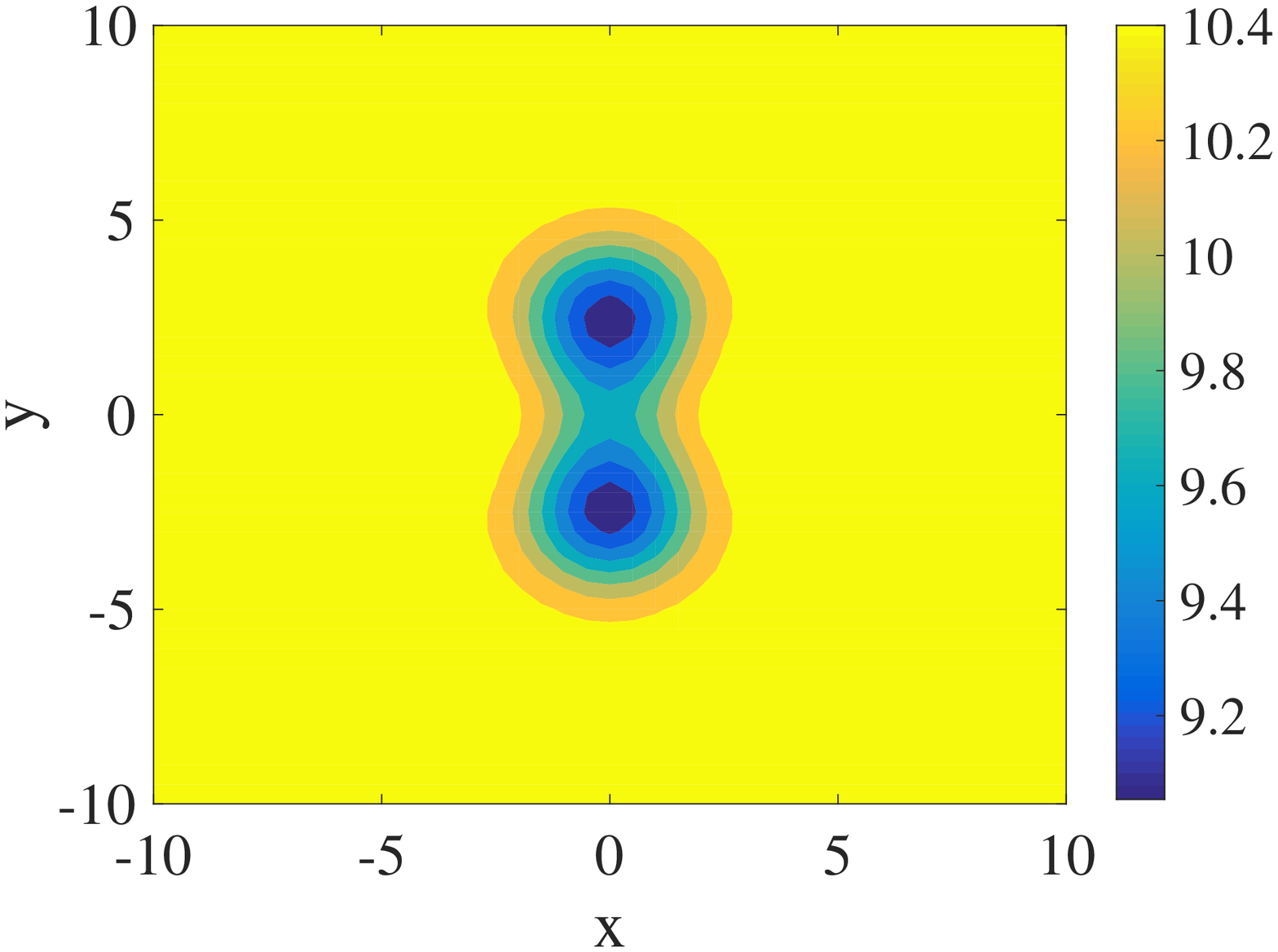}
            \end{minipage}
            }
           \centering \subfigure[]{
            \begin{minipage}[b]{0.3\textwidth}
               \centering
             \includegraphics[width=\textwidth,height=0.9\textwidth]{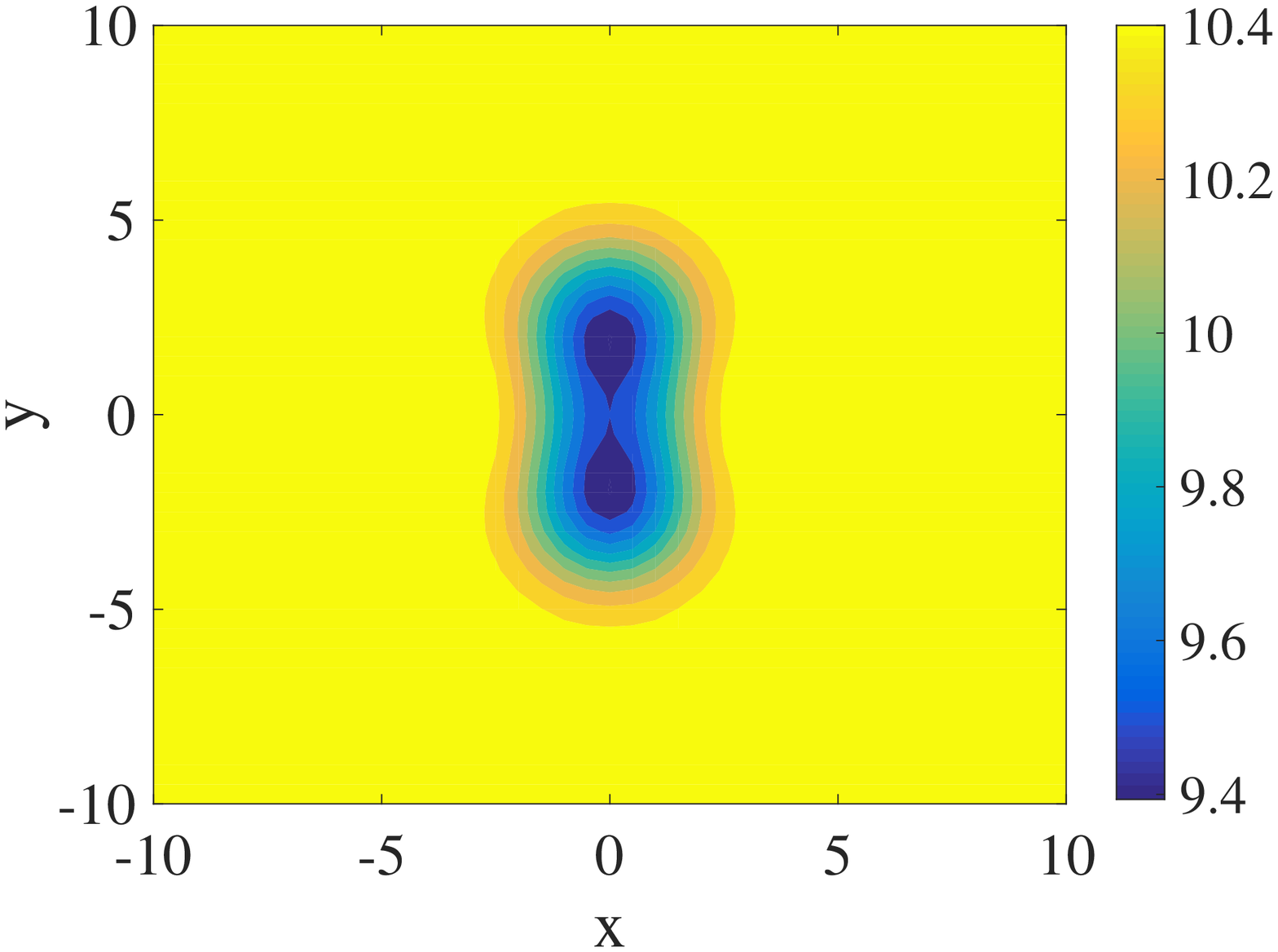}
            \end{minipage}
            }
            \centering \subfigure[]{
            \begin{minipage}[b]{0.3\textwidth}
            \centering
             \includegraphics[width=\textwidth,height=0.9\textwidth]{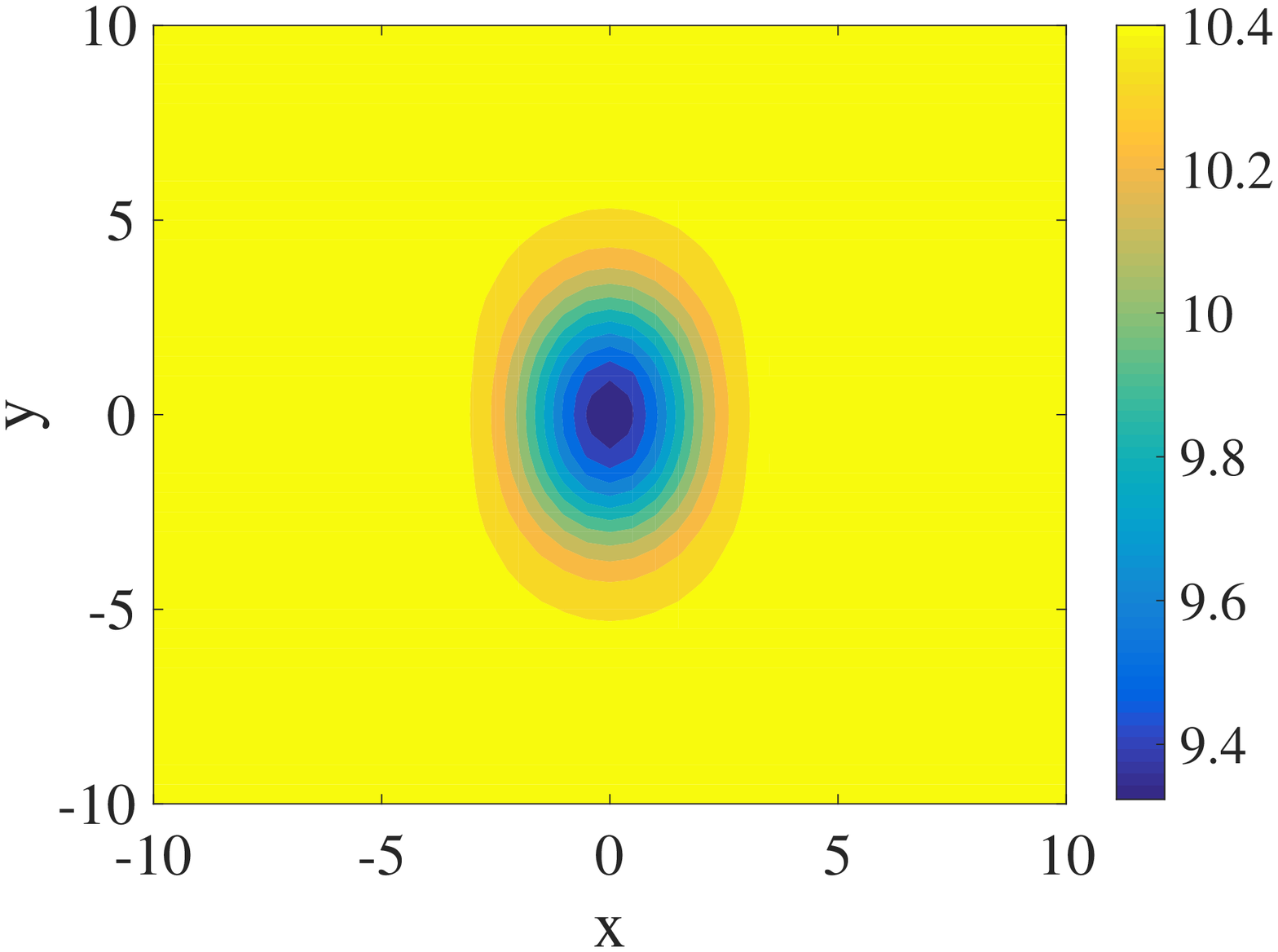}
            \end{minipage}
            }
           \centering \subfigure[]{
            \begin{minipage}[b]{0.3\textwidth}
            \centering
             \includegraphics[width=\textwidth,height=0.9\textwidth]{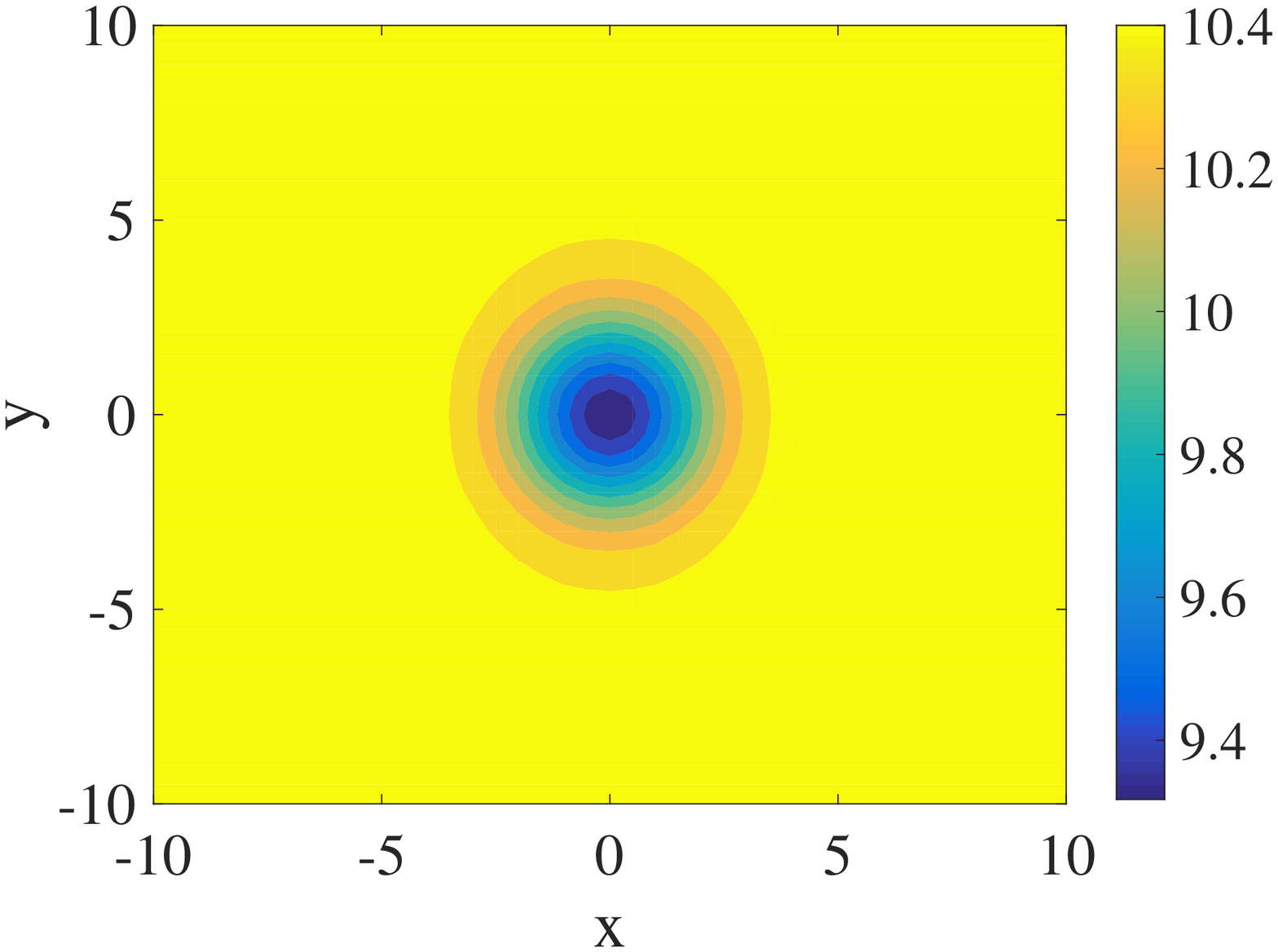}
            \end{minipage}
            }
           \caption{Ternary mixture: C$_1$ molar densities at  the the initial(a), 30th(b), 100th(c), 200th(d), 500th(e) and 1000th(f)   time step  respectively.}
            \label{MCTwoSquareCh4andTwoHydrocarbonsMolarDensityOfCH4Temperature323K}
 \end{figure}

\begin{figure}
            \centering \subfigure[]{
            \begin{minipage}[b]{0.3\textwidth}
               \centering
             \includegraphics[width=\textwidth,height=0.9\textwidth]{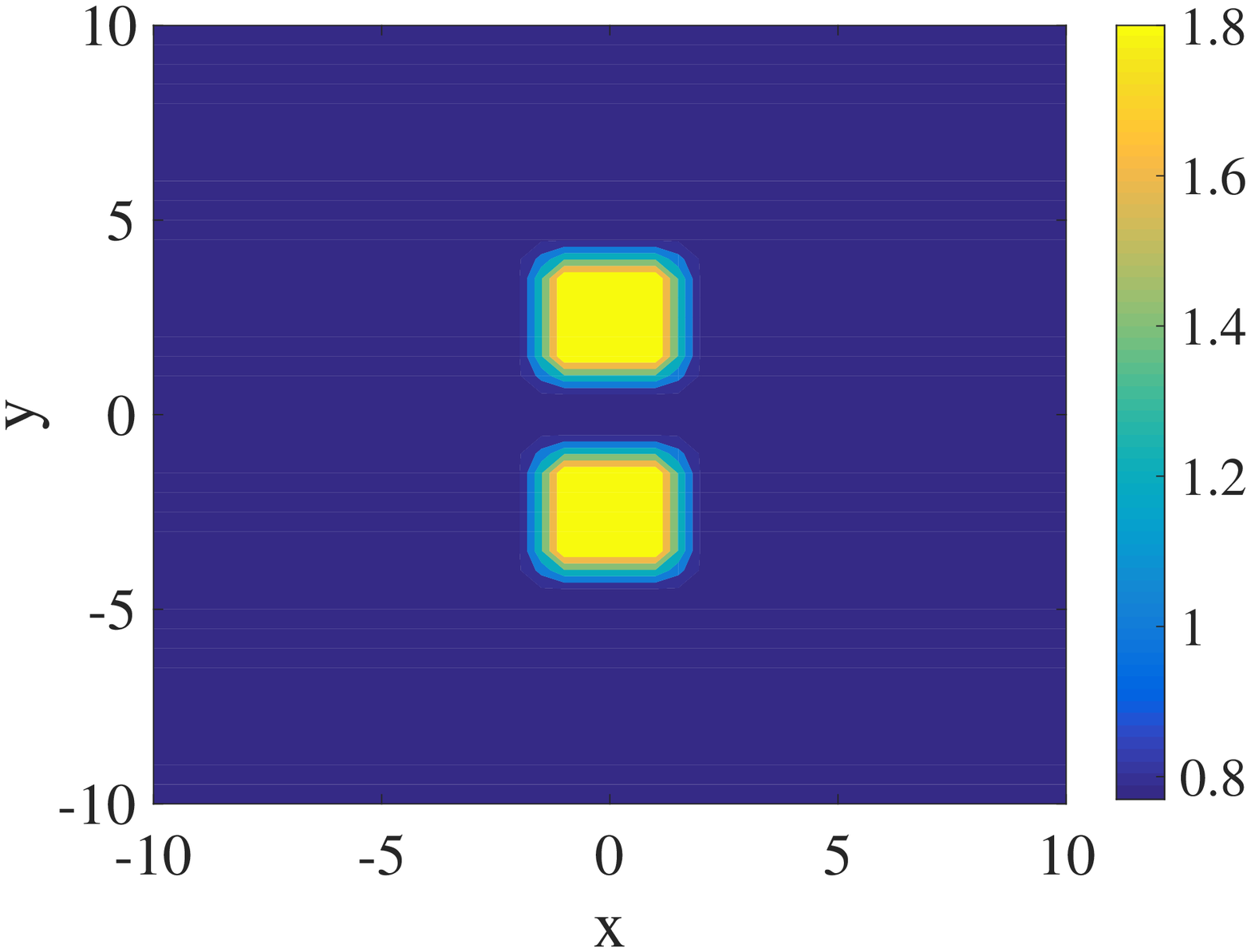}
            \end{minipage}
            }
            \centering \subfigure[]{
            \begin{minipage}[b]{0.3\textwidth}
            \centering
             \includegraphics[width=\textwidth,height=0.9\textwidth]{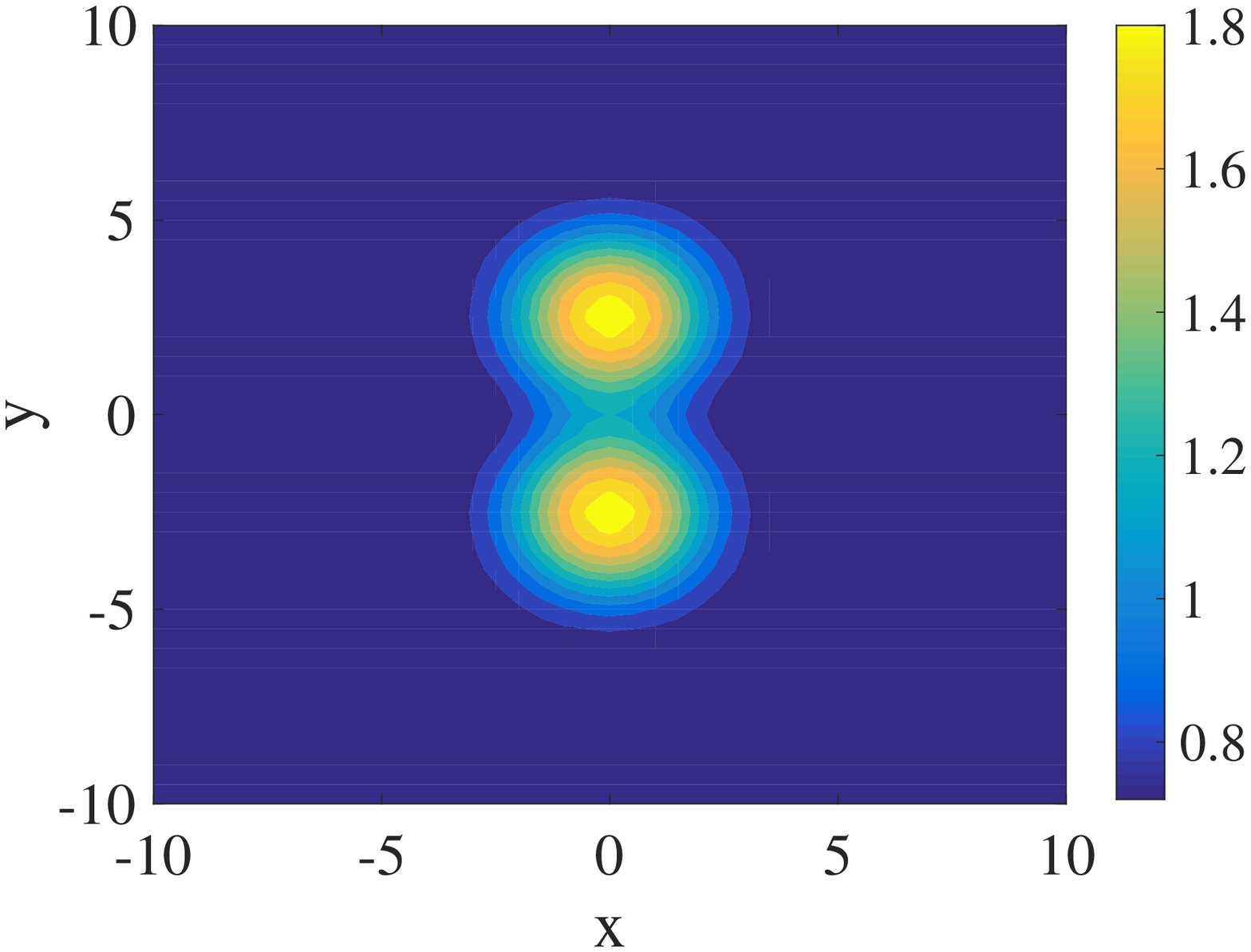}
            \end{minipage}
            }
           \centering \subfigure[]{
            \begin{minipage}[b]{0.3\textwidth}
            \centering
             \includegraphics[width=\textwidth,height=0.9\textwidth]{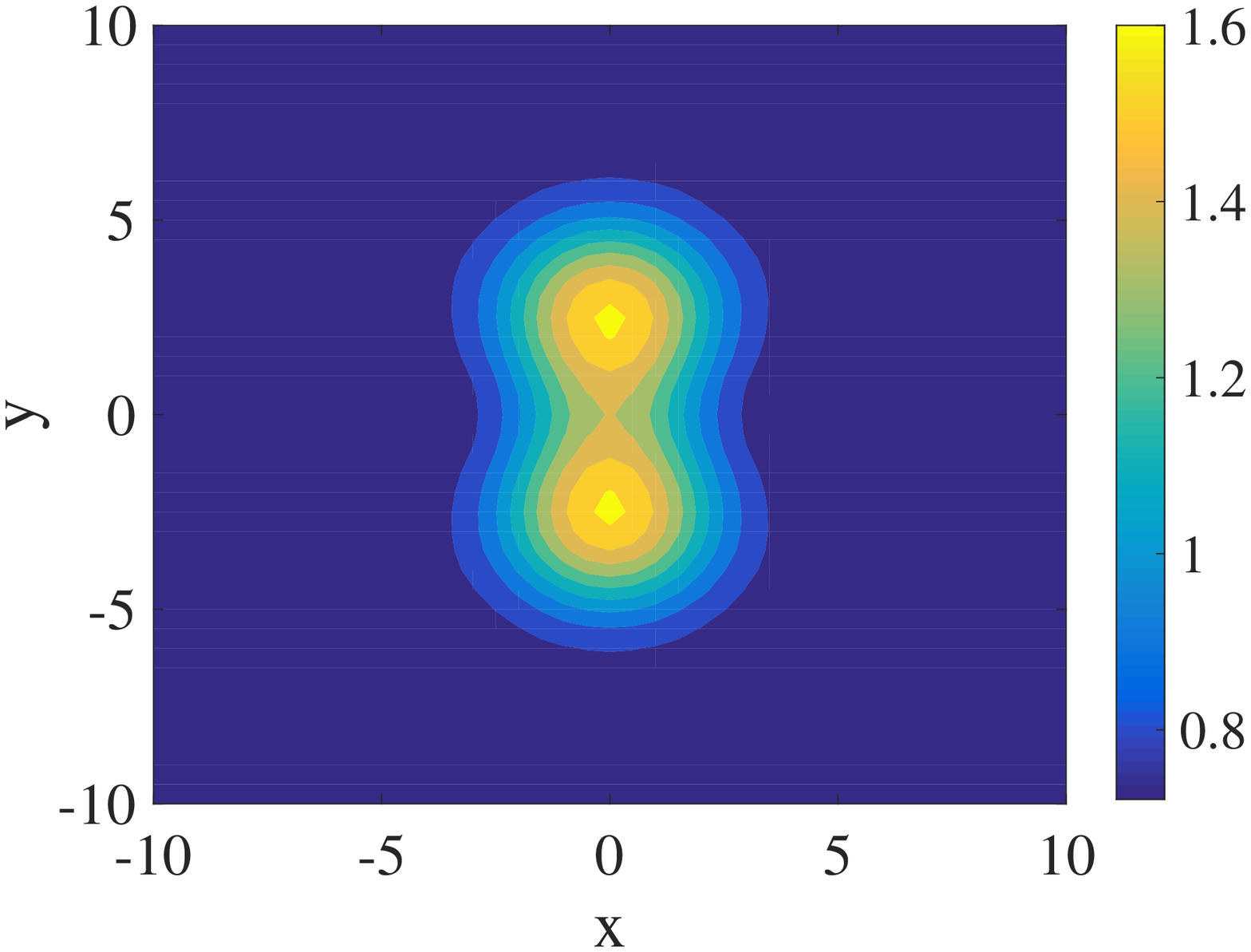}
            \end{minipage}
            }
           \centering \subfigure[]{
            \begin{minipage}[b]{0.3\textwidth}
               \centering
             \includegraphics[width=\textwidth,height=0.9\textwidth]{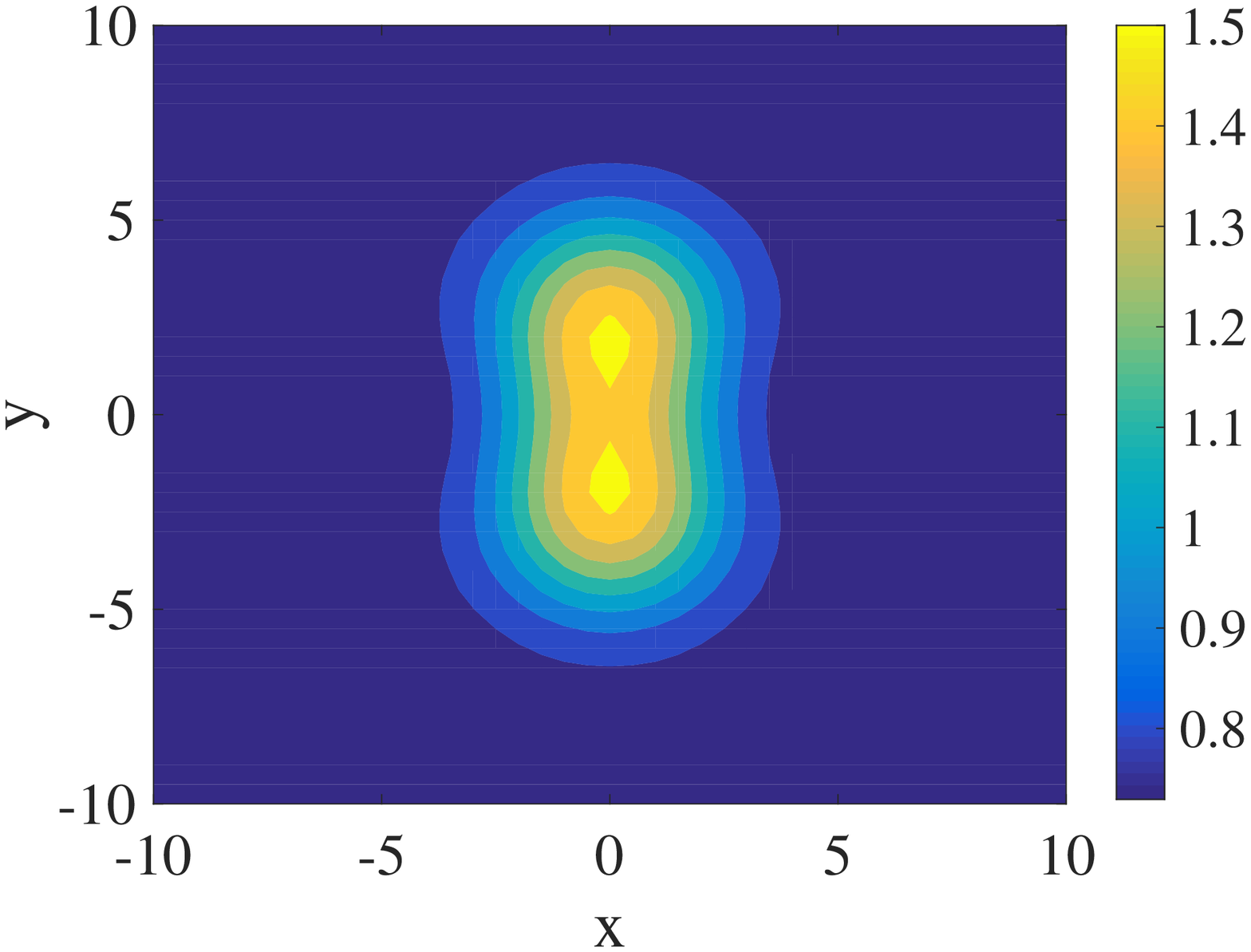}
            \end{minipage}
            }
            \centering \subfigure[]{
            \begin{minipage}[b]{0.3\textwidth}
            \centering
             \includegraphics[width=\textwidth,height=0.9\textwidth]{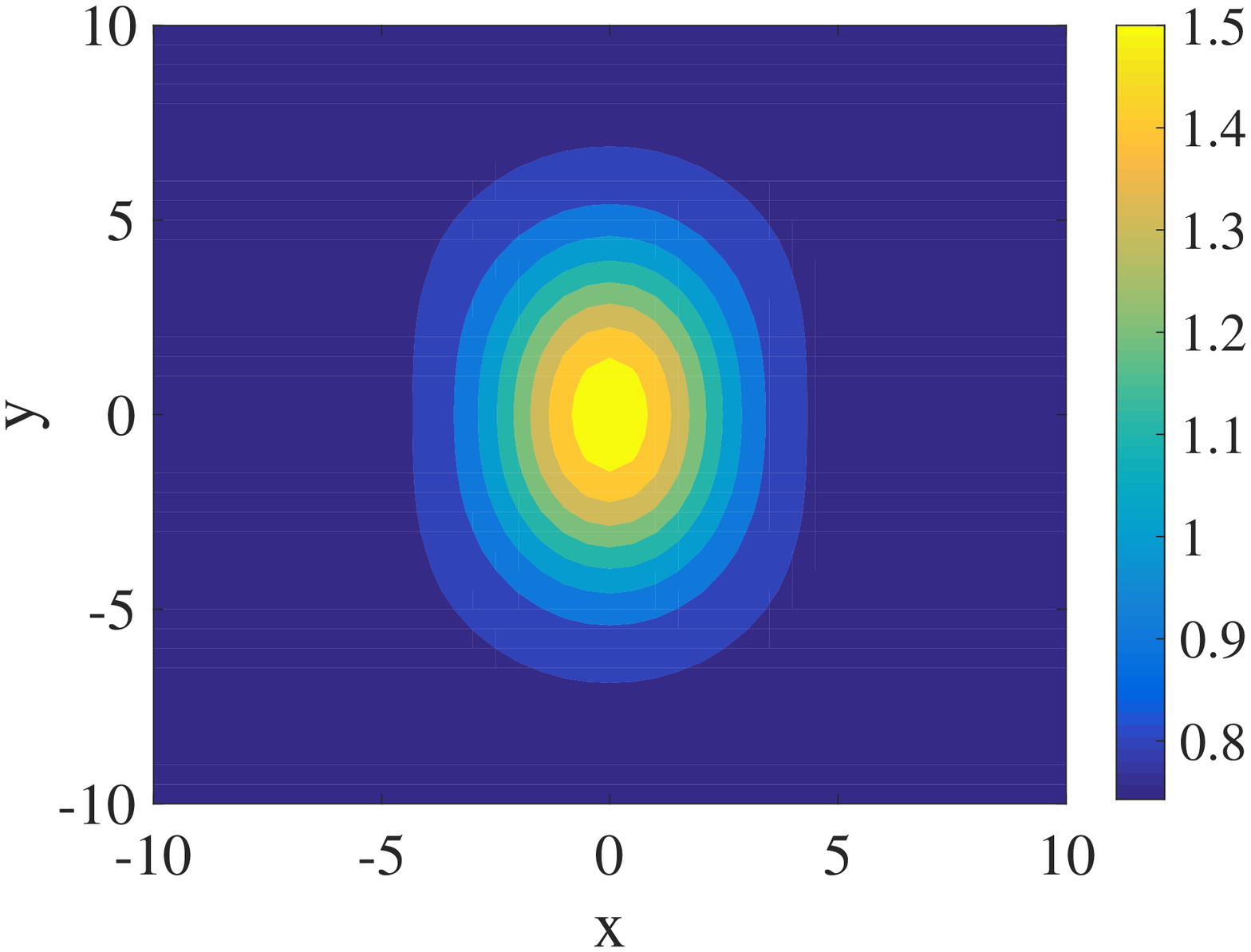}
            \end{minipage}
            }
           \centering \subfigure[]{
            \begin{minipage}[b]{0.3\textwidth}
            \centering
             \includegraphics[width=\textwidth,height=0.9\textwidth]{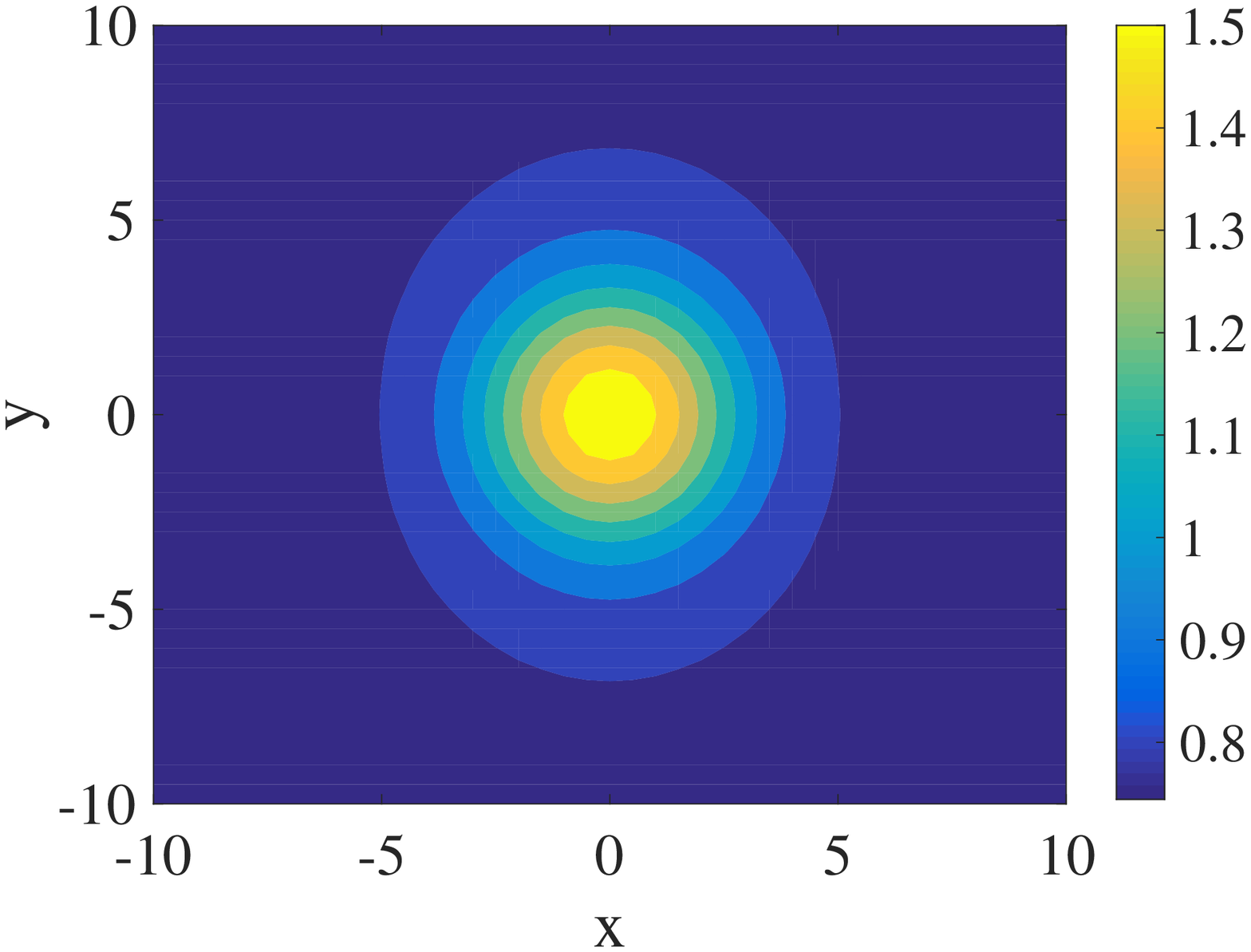}
            \end{minipage}
            }
           \caption{Ternary mixture: C$_{5}$ molar densities at  the initial(a),  30th(b), 100th(c), 200th(d), 500th(e) and 1000th(f)   time step  respectively.}
            \label{MCTwoSquareCh4andTwoHydrocarbonsMolarDensityOfnC5Temperature323K}
 \end{figure}

\begin{figure}
            \centering \subfigure[]{
            \begin{minipage}[b]{0.3\textwidth}
               \centering
             \includegraphics[width=0.95\textwidth,height=0.9\textwidth]{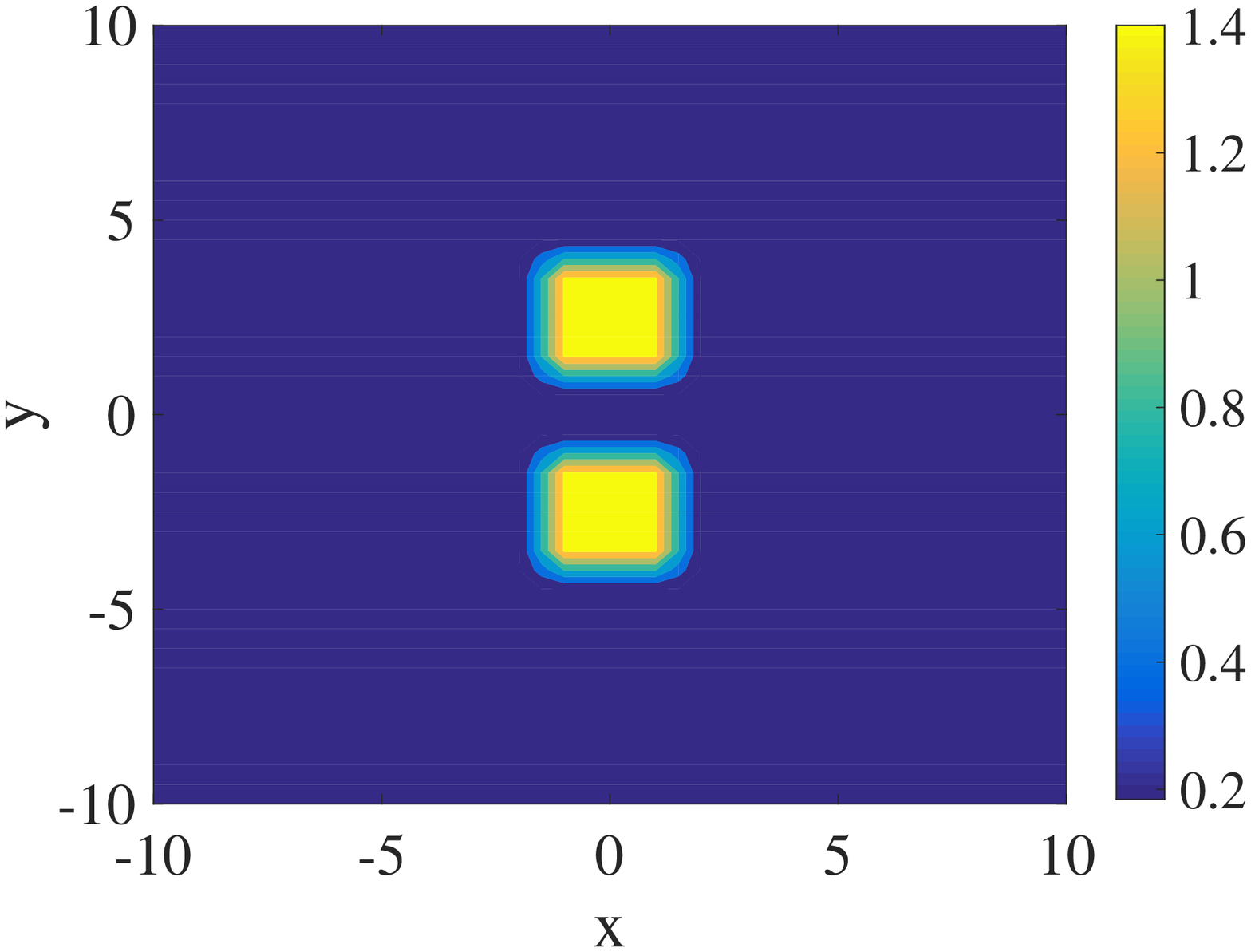}
            \end{minipage}
            }
            \centering \subfigure[]{
            \begin{minipage}[b]{0.3\textwidth}
            \centering
             \includegraphics[width=0.95\textwidth,height=0.9\textwidth]{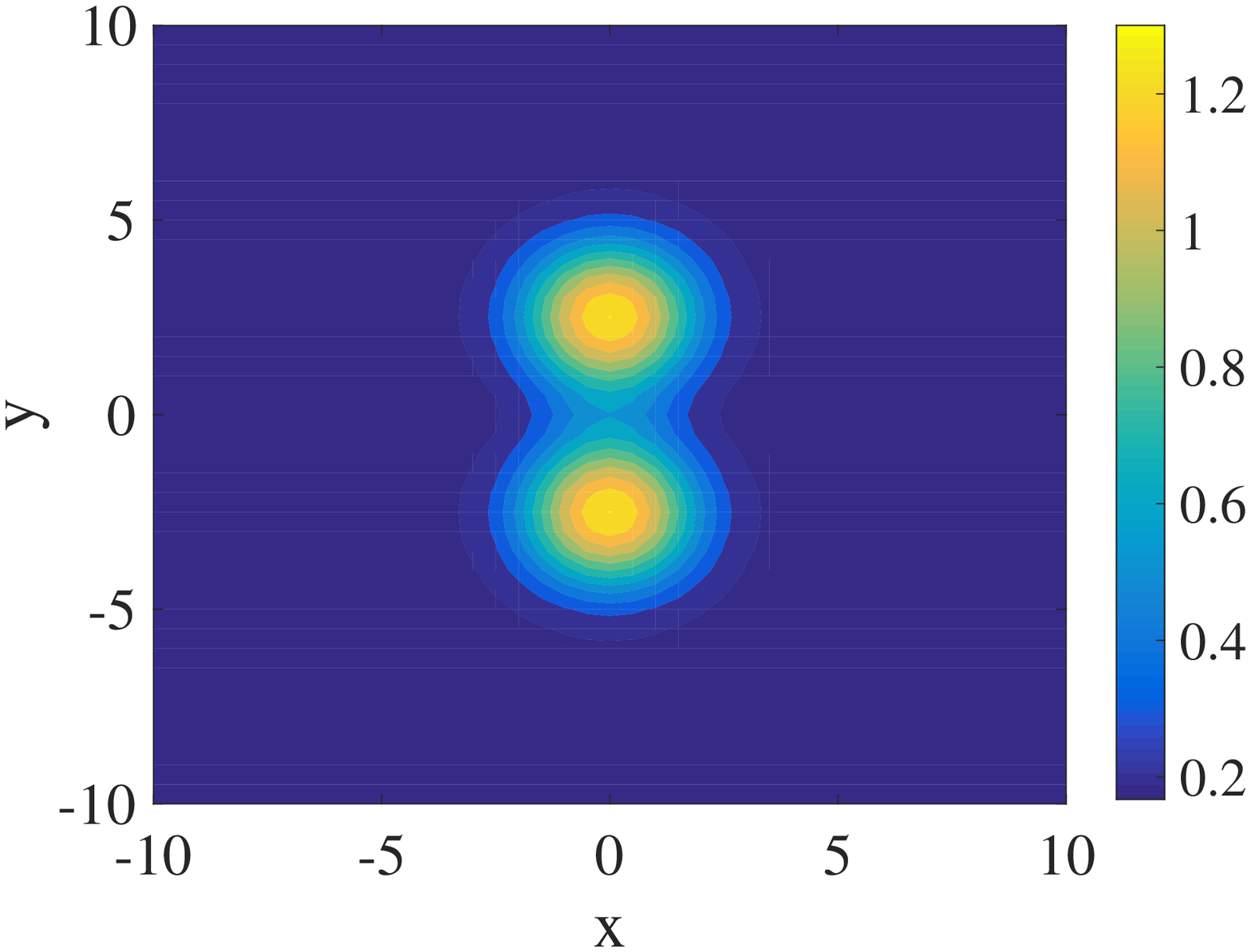}
            \end{minipage}
            }
           \centering \subfigure[]{
            \begin{minipage}[b]{0.3\textwidth}
            \centering
             \includegraphics[width=0.95\textwidth,height=0.9\textwidth]{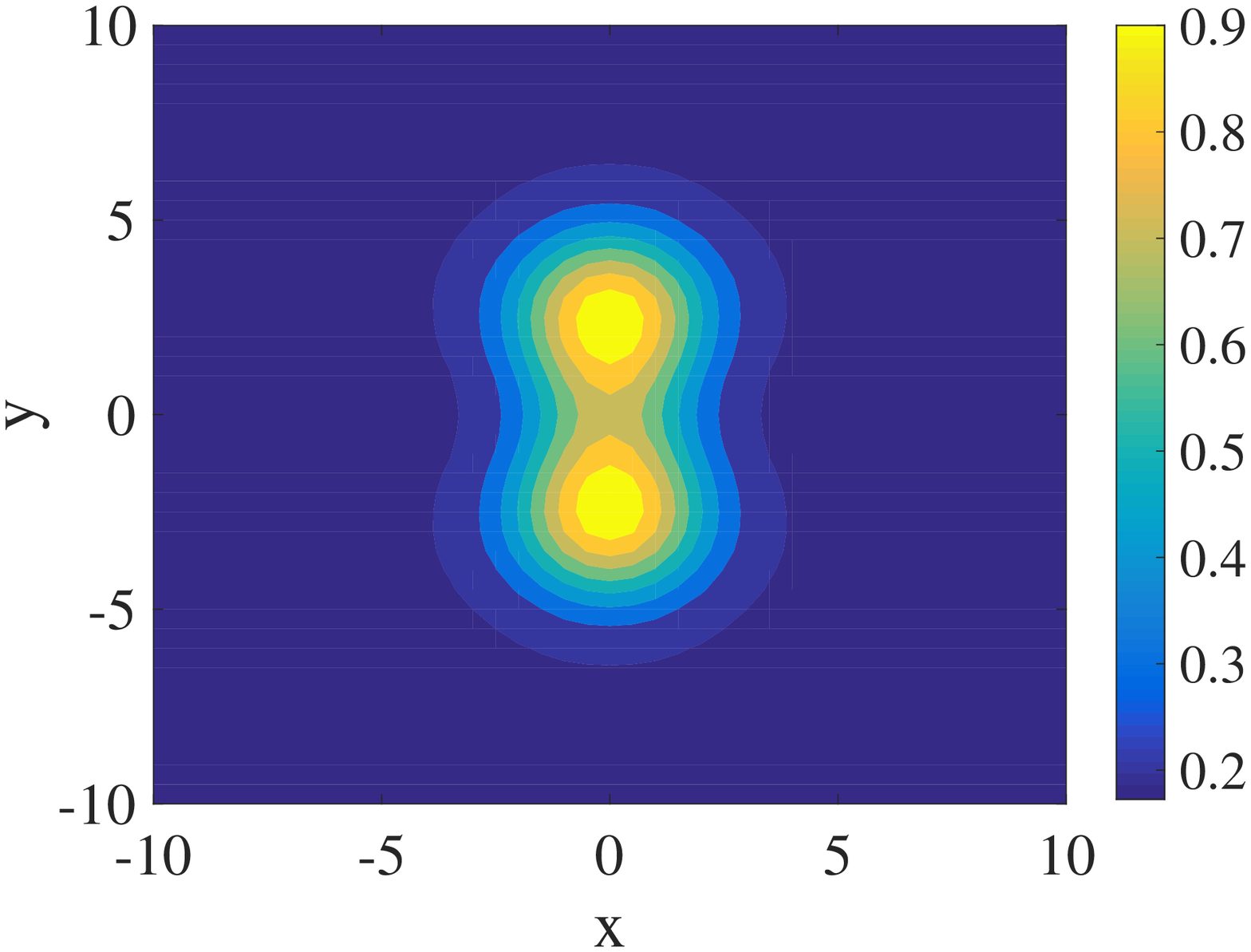}
            \end{minipage}
            }
           \centering \subfigure[]{
            \begin{minipage}[b]{0.3\textwidth}
               \centering
             \includegraphics[width=0.95\textwidth,height=0.9\textwidth]{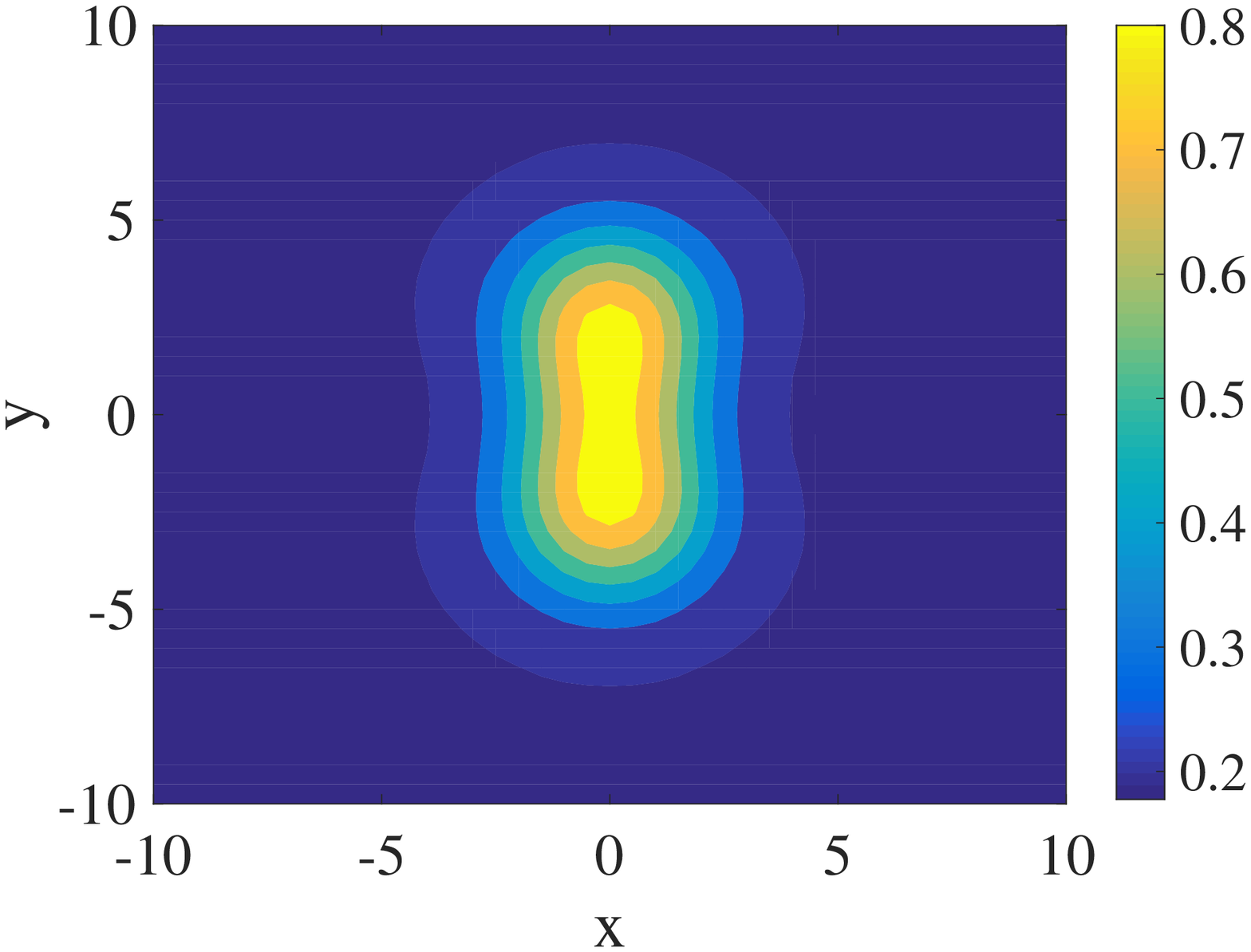}
            \end{minipage}
            }
            \centering \subfigure[]{
            \begin{minipage}[b]{0.3\textwidth}
            \centering
             \includegraphics[width=0.95\textwidth,height=0.9\textwidth]{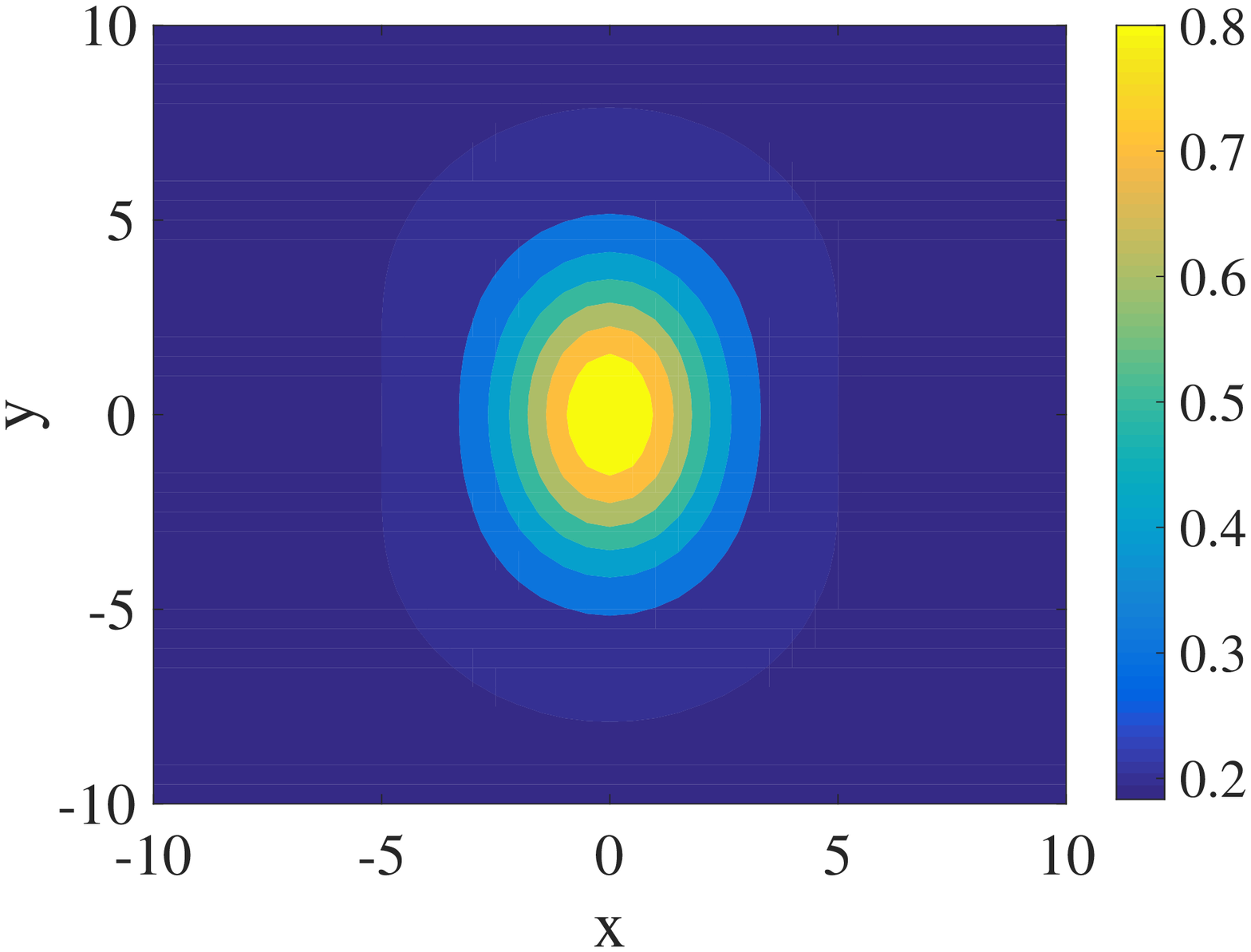}
            \end{minipage}
            }
           \centering \subfigure[]{
            \begin{minipage}[b]{0.3\textwidth}
            \centering
             \includegraphics[width=0.95\textwidth,height=0.9\textwidth]{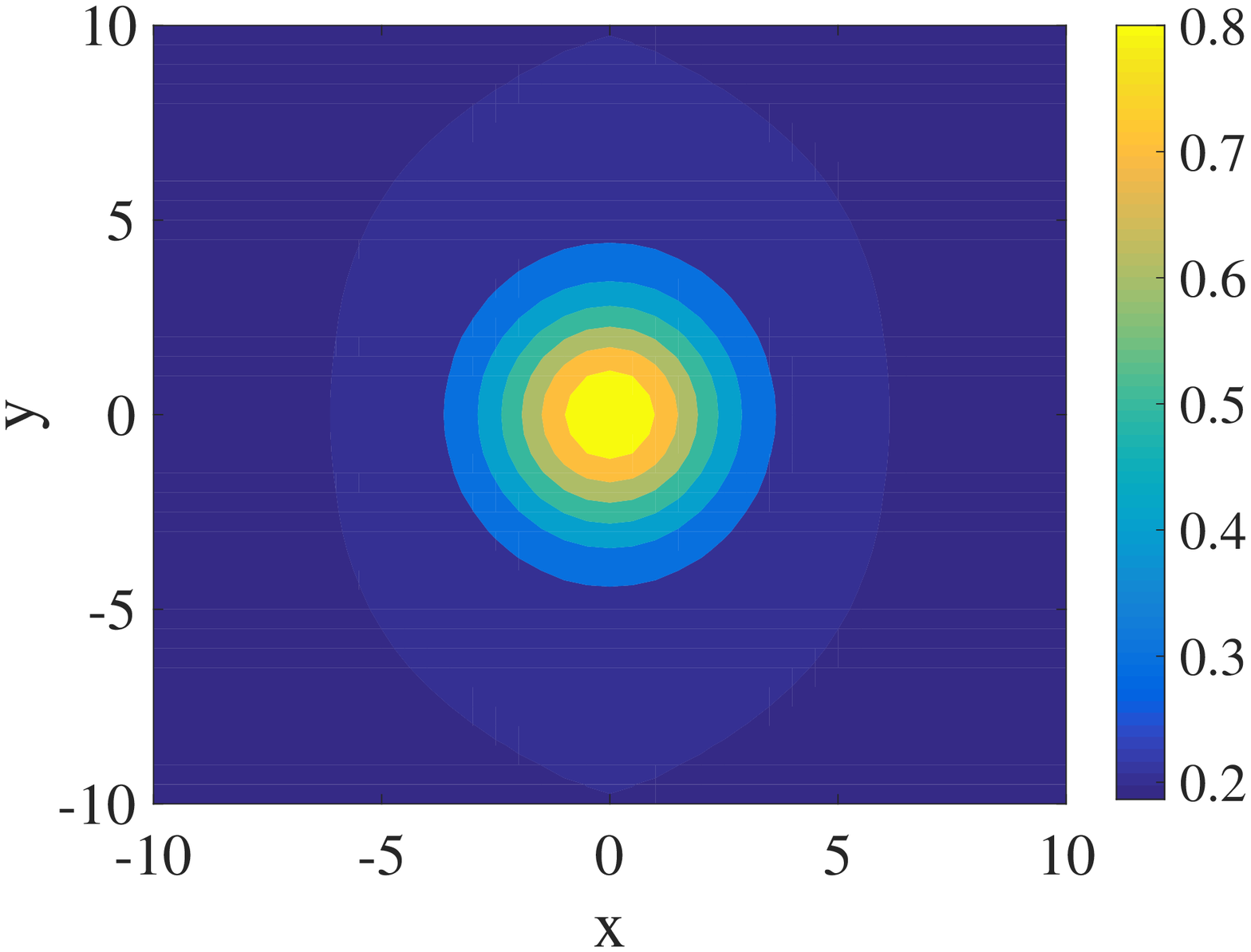}
            \end{minipage}
            }
           \caption{Ternary mixture: C$_{10}$ molar densities at  the initial(a),   30th(b), 100th(c), 200th(d), 500th(e) and 1000th(f)     time step  respectively.}
            \label{MCTwoSquareCh4andTwoHydrocarbonsMolarDensityOfnC10Temperature323K}
 \end{figure}

\begin{figure}
            \centering \subfigure[]{
            \begin{minipage}[b]{0.3\textwidth}
               \centering
             \includegraphics[width=0.95\textwidth,height=0.9\textwidth]{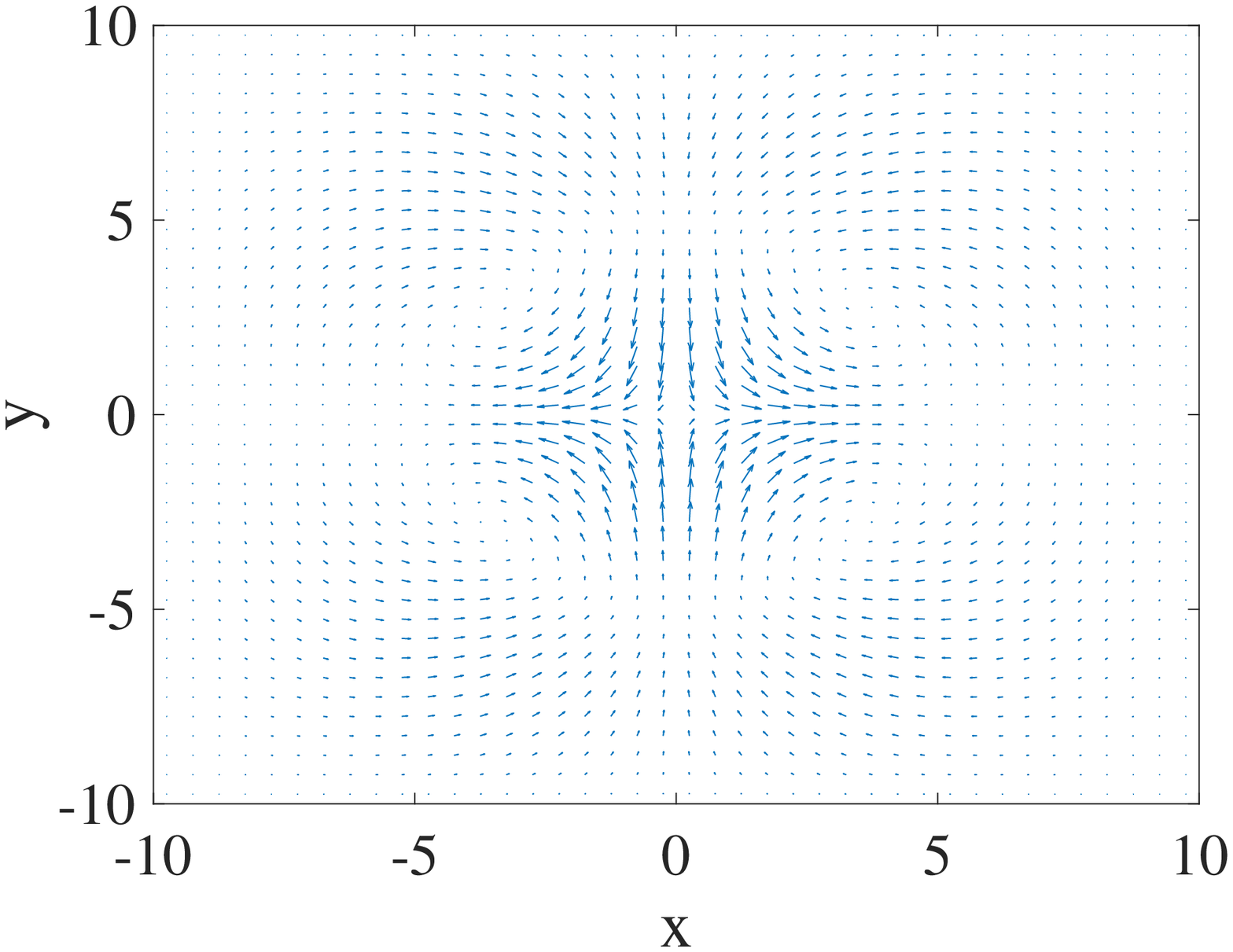}
            \end{minipage}
            }
            \centering \subfigure[]{
            \begin{minipage}[b]{0.3\textwidth}
               \centering
             \includegraphics[width=0.95\textwidth,height=0.9\textwidth]{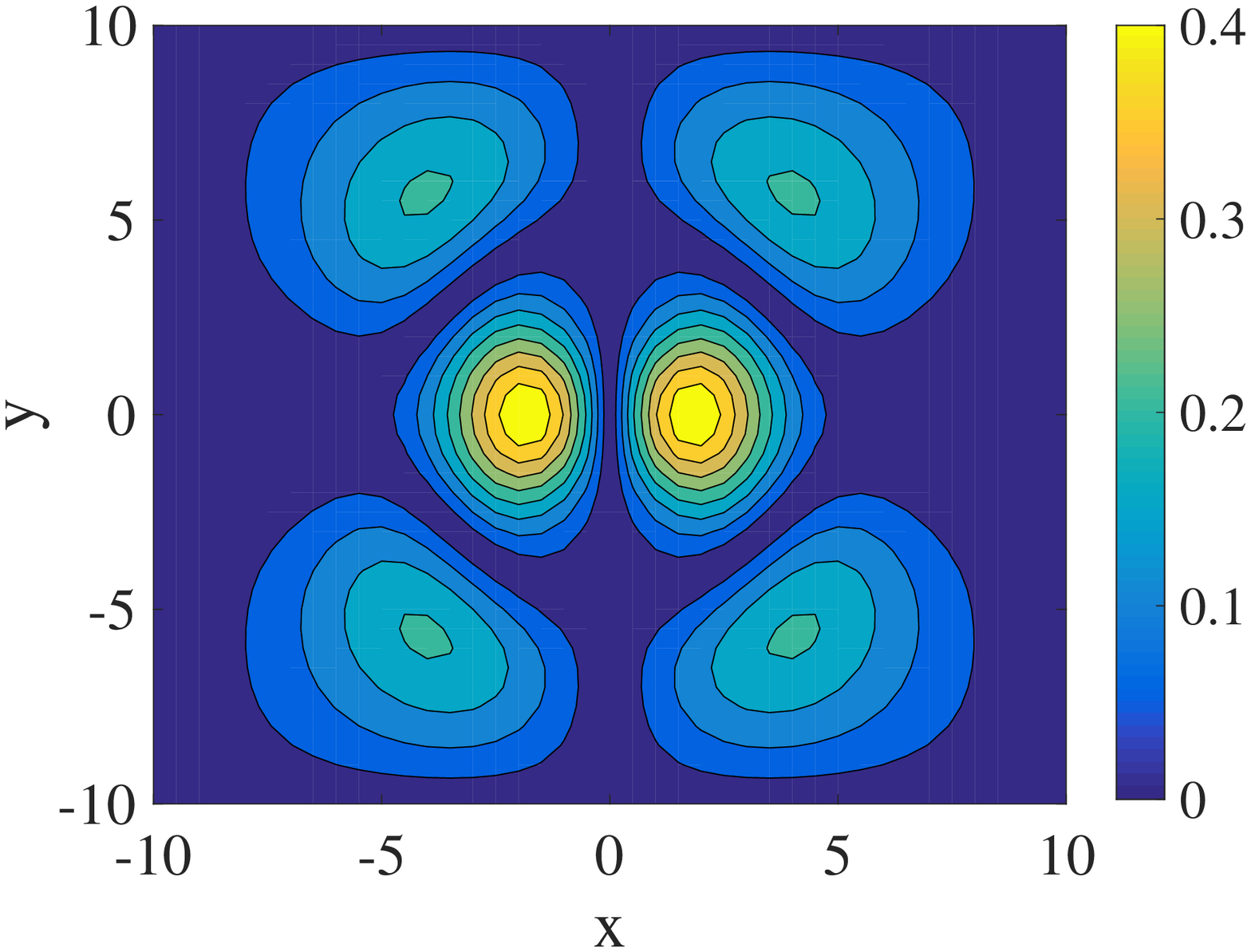}
            \end{minipage}
            }
            \centering \subfigure[]{
            \begin{minipage}[b]{0.3\textwidth}
            \centering
             \includegraphics[width=0.95\textwidth,height=0.9\textwidth]{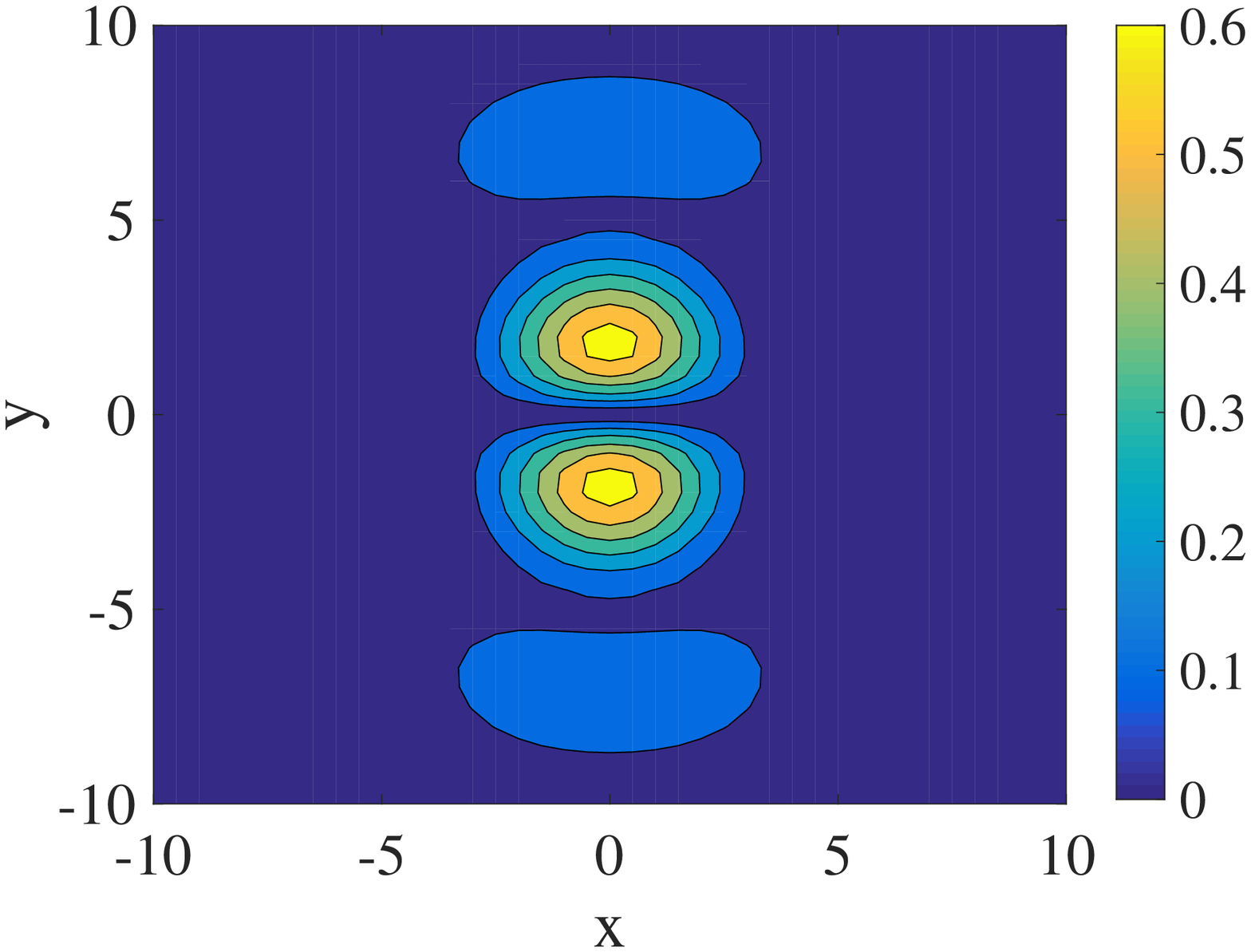}
            \end{minipage}
            }
           \centering \subfigure[]{
            \begin{minipage}[b]{0.3\textwidth}
               \centering
             \includegraphics[width=0.95\textwidth,height=0.9\textwidth]{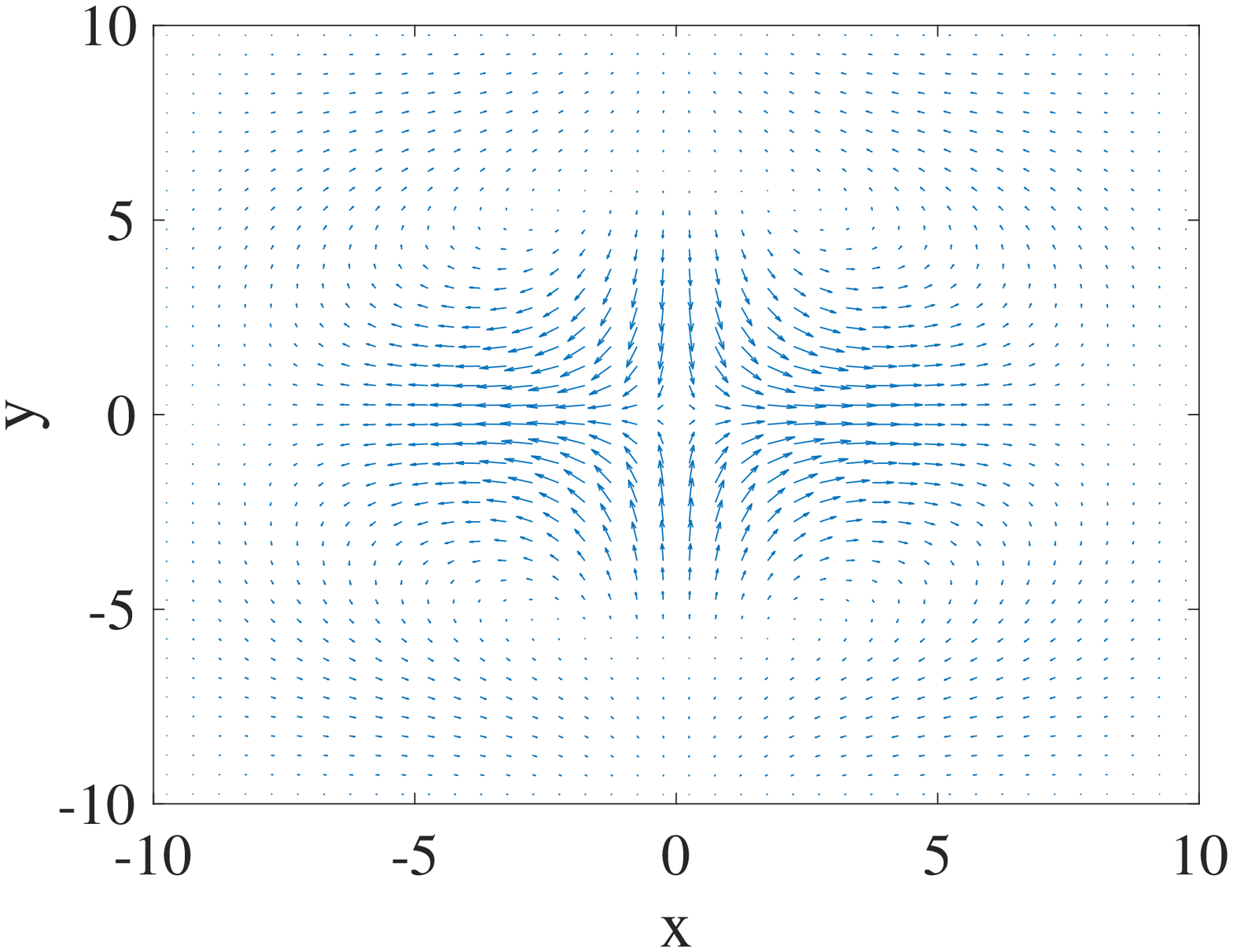}
            \end{minipage}
            }
            \centering \subfigure[]{
            \begin{minipage}[b]{0.3\textwidth}
               \centering
             \includegraphics[width=0.95\textwidth,height=0.9\textwidth]{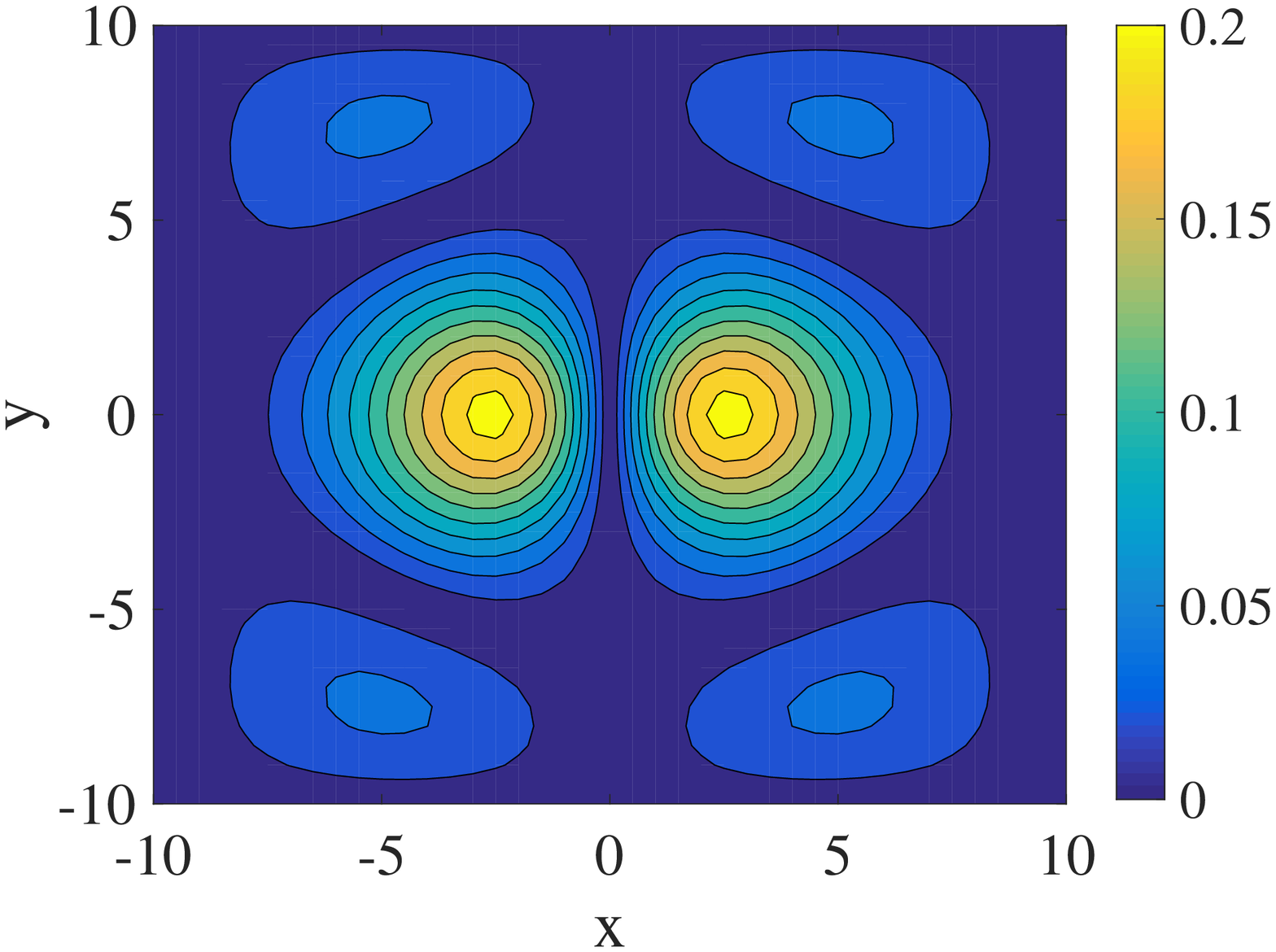}
            \end{minipage}
            }
            \centering \subfigure[]{
            \begin{minipage}[b]{0.3\textwidth}
            \centering
             \includegraphics[width=0.95\textwidth,height=0.9\textwidth]{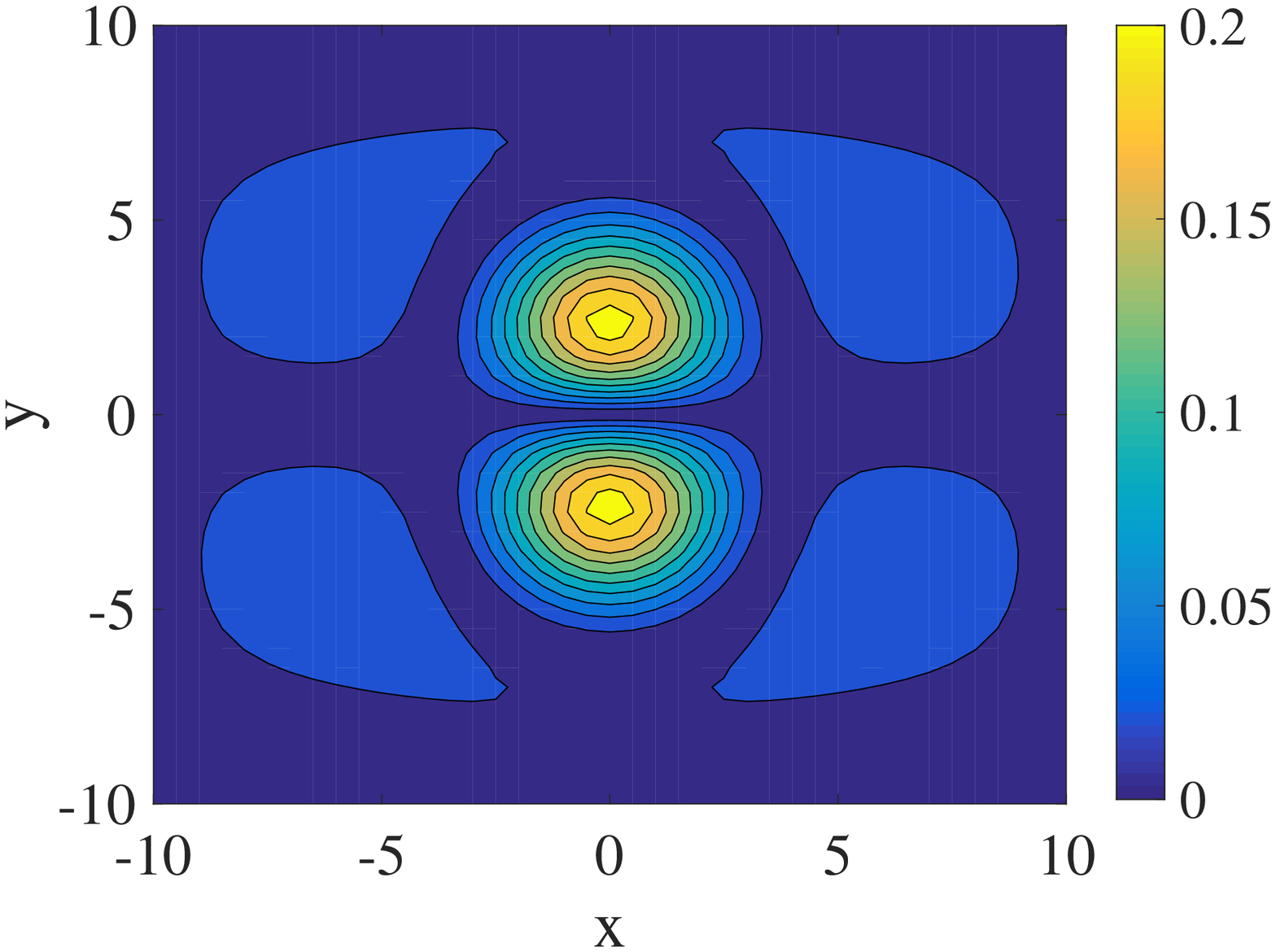}
            \end{minipage}
            }
           \centering \subfigure[]{
            \begin{minipage}[b]{0.3\textwidth}
               \centering
             \includegraphics[width=0.95\textwidth,height=0.9\textwidth]{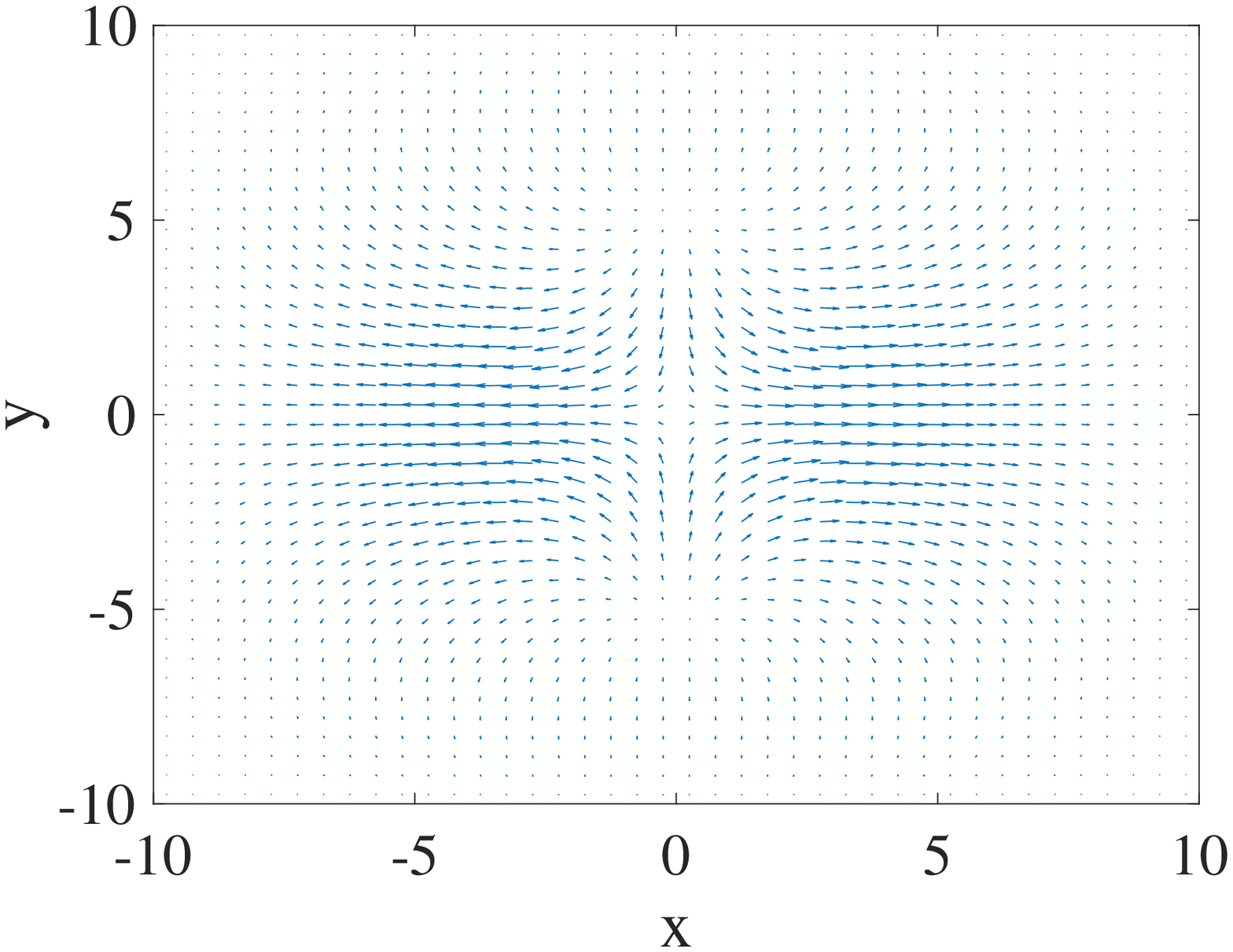}
            \end{minipage}
            }
            \centering \subfigure[]{
            \begin{minipage}[b]{0.3\textwidth}
               \centering
             \includegraphics[width=0.95\textwidth,height=0.9\textwidth]{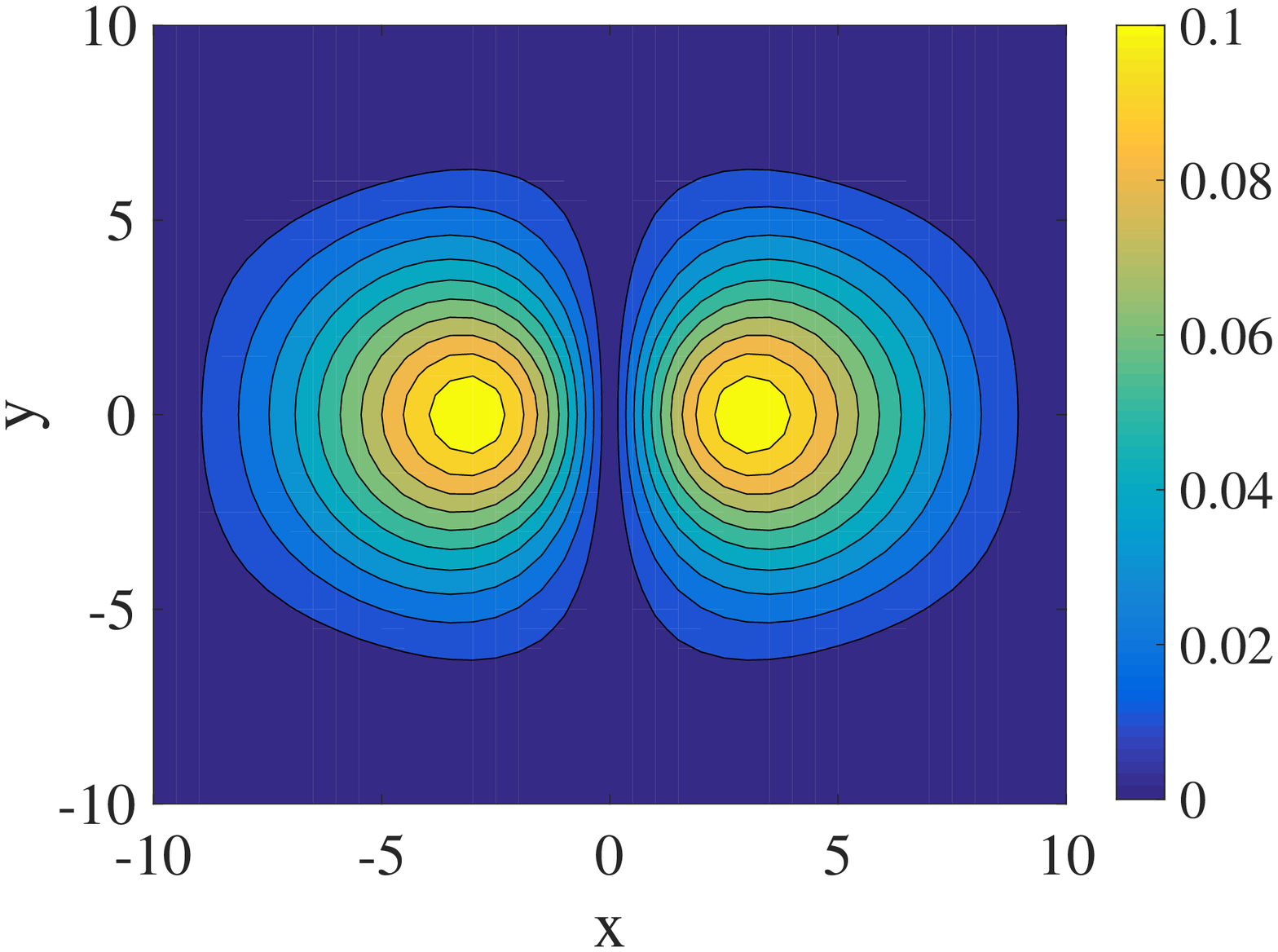}
            \end{minipage}
            }
            \centering \subfigure[]{
            \begin{minipage}[b]{0.3\textwidth}
            \centering
             \includegraphics[width=0.95\textwidth,height=0.9\textwidth]{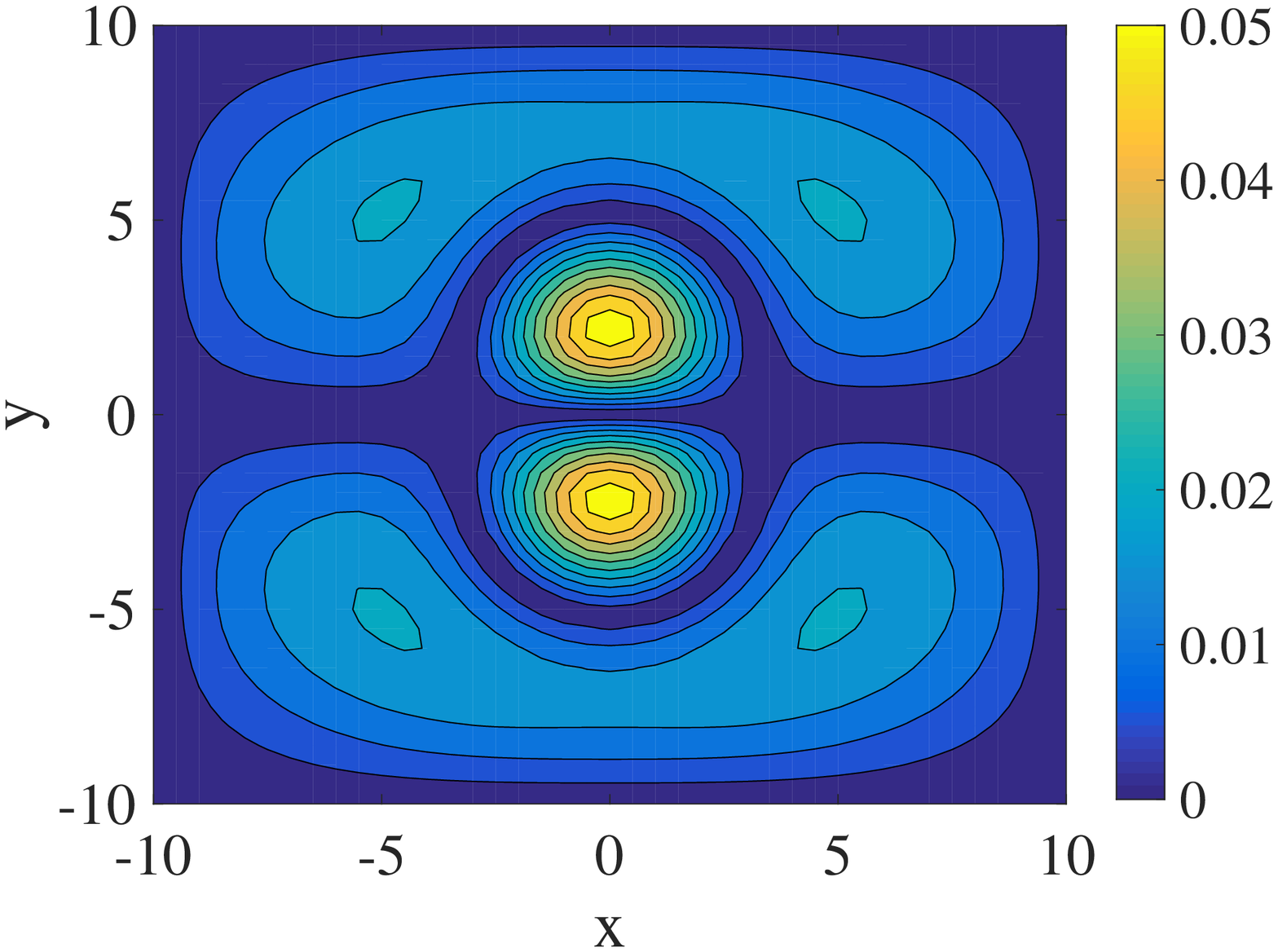}
            \end{minipage}
            }
           \caption{Ternary mixture:   flow quivers (left column), magnitude contours of $x$-direction velocity component (center column), and magnitude contours of $y$-direction velocity component (right column)   at the 100th(top row), 500th(center row), and 1000th(bottom row) time step  respectively.}
            \label{MCTwoSquareCh4andTwoHydrocarbonsVelocityTemperature323K}
 \end{figure}

 %%%%%%%%%%%%%%% Conclusions %%%%%%%%%
\section{Conclusions}

%%%%%%%%%%%%%%%%%%%%%%%%%%%%%%%%%%%%%%%%
 
Two decoupled  energy-stable numerical  schemes are developed for    multi-component two-phase compressible flows with a  realistic equation of state (e.g. Peng-Robinson equation of state). In these methods,  the scalar auxiliary variable (SAV) approach is applied to deal with the bulk Helmholtz free energy, and moreover, we propose a component-wise SAV approach, which is extremely  efficient and easy-to-implement for multi-component flows. In order to uncouple the tight  relationship between velocity and molar densities, we introduce two intermediate  velocities, one  of which has a component-wise form matching the component-wise SAV approach.  
The intermediate  velocities are involved in the discrete  formulation of the  momentum balance equation, which establishes  the consistent relationships with  the mass balance equations.
The proposed numerical  schemes  only need to solve  a sequence of linear equations at each time step.  The discrete unconditional energy dissipation laws of the proposed methods are proved rigorously.  Numerical results validate  the effectiveness of the proposed methods.

%\section*{Appendix}
\begin{appendix}\label{appendix}

\section{Helmholtz free energy density}\label{appendixHelmholtz}
Let $R$ be the universal gas constant and let $T$ be the specified temperature. The three contributions of  Helmholtz free energy density $f_{b}(\n)$ based on Peng-Robinson equation of state are formulated as
 \begin{eqnarray}\label{eqHelmholtzEnergy_a0_01}
    f_b^{\textnormal{ideal}}(\n)= RT\sum_{i=1}^{M}n_i\(\ln n_i-1\),
\end{eqnarray}
\begin{eqnarray}\label{eqHelmholtzEnergy_a0_02}
    f_b^{\textnormal{repulsion}}(\n)=-nRT\ln\(1-bn\),
\end{eqnarray}
\begin{eqnarray}\label{eqHelmholtzEnergy_a0_03}
    f_b^{\textnormal{attraction}}(\n)= \frac{a(T)n}{2\sqrt{2}b}\ln\(\frac{1+(1-\sqrt{2})b n}{1+(1+\sqrt{2})b n}\),
\end{eqnarray}
where  $n=\sum_{i=1}^Mn_i$ is the overall molar density. Here,  $a$ and $b$ are the energy parameter and  the covolume respectively, which depend on  the mixture composition and  temperature.
Let us denote by $T_{c_i}$ and $P_{c_i}$  the   critical temperature and critical pressure of component $i$ respectively.  For the $i$th component, we let the reduced temperature be  $T_{r_i}=T/T_{c_i}$. 
The parameters $a_{i}$ and $b_{i}$ are calculated as
\begin{eqnarray}\label{eqHelmholtzEnergy_ab}
   a_{i}= 0.45724\frac{R^2T_{c_i}^2}{P_{c_i}}\[1+m_i(1-\sqrt{T_{r_i}})\]^2,~~~~b_{i}= 0.07780\frac{RT_{c_i}}{P_{c_i}}.
\end{eqnarray}
We denote by $\omega_i$  the acentric factor of component $i$.
 The coefficients $m_i$ are calculated  by the following relations
\begin{eqnarray*}
 m_i=0.37464 + 1.54226\omega_i-  0.26992\omega_i^2 ,~~\omega_i\leq0.49,
\end{eqnarray*}
\begin{eqnarray*}
 m_i=0.379642+1.485030\omega_i-0.164423\omega_i^2 +0.016666 \omega_i^3,~~\omega_i>0.49.
\end{eqnarray*}
  We denote by $y_i=n_i/n$ the mole fraction of component $i$ and let $k_{ij}$ be the  binary interaction coefficients for the energy parameters.  Then we calculate  $a(T)$ and $b$   as
\begin{eqnarray*}
   a =\sum_{i=1}^M\sum_{j=1}^M y_i y_j \(a_ia_j\)^{1/2}(1-k_{ij}),~~~~b =\sum_{i=1}^M  y_i b_{i}.
\end{eqnarray*}

We list some physical parameters of three substances in Table \ref{tabParametersPREOS}. 
 \begin{table}[htp]
\caption{Physical parameters}
\begin{center}
\begin{tabular}{cccccc}
\hline
Substance & $P_c$(bar) & $T_c$(K) & Acentric factor & $M_w$(g/mole)\\
\hline
methane     &  45.99                          & 190.56                            &0.011              &16.04 \\
pentane      &  33.70                          & 469.7                              &0.251              &72.15\\
decane       &  21.1                            & 617.7                              &0.489              &142.28\\
\hline
\end{tabular}
\end{center}
\label{tabParametersPREOS}
\end{table}%binary interaction coefficients

\section{Influence parameters}\label{appendixInfluenceParameters}
 The  influence parameters are generally  assumed to rely  on the temperature but  independent of    molar densities.  We  now provide  the formulations of the  influence parameters. First, we formulate the  influence parameter $c_{i}$ of  component $i$ as \cite{miqueu2004modelling}
\begin{eqnarray*}
 c_i=a_ib_i^{2/3}\[\gamma_i(1-T_{r_i})+\phi_i\],
\end{eqnarray*}
where  $a_i$ and $b_i$ are given in \eqref{eqHelmholtzEnergy_ab}, and  $\gamma_i$ and $\phi_i$ are the coefficients
 correlated merely with the acentric factor $\omega_i$ of   component $i$
 by the following relations
\begin{eqnarray*}
 \gamma_i=-\frac{10^{-16}}{1.2326+1.3757\omega_i},~~~~\phi_i=\frac{10^{-16}}{0.9051+1.5410\omega_i}.
\end{eqnarray*}

The cross influence parameter between binary components   is generally  calculated  as a modified geometric mean of the pure component
influence parameters $c_i$ and $c_j$ 
\begin{eqnarray*}\label{eqCij}
 c_{ij}= (1-\beta_{ij})\sqrt{c_ic_j},
\end{eqnarray*}
where  $\beta_{ij}$ are the binary interaction coefficients   satisfying the symmetry $c_{ij}=c_{ji}$ and $\beta_{ii}=0,~0\leq\beta_{ij}<1$.   In  numerical tests, we take $\beta_{ij}=0.5$ for $i\neq j$.

\end{appendix}

\small
%\bibliographystyle{plain}
%\bibliography{References}

\end{document}